\documentclass[amsfonts]{report}

\usepackage{amssymb,amsfonts,amscd,eucal,mathrsfs,amsmath}
\usepackage{glossaries}
\usepackage[square,sort,comma,numbers]{natbib}
\usepackage{amsthm}
\usepackage{tocloft}
\usepackage[all,arc]{xy}
\usepackage{enumerate}
\usepackage{mathrsfs}
\usepackage{hyperref} 
\usepackage{graphicx}
\usepackage{tikz}
\usepackage{lastpage}
\usepackage{bmpsize}
\usetikzlibrary{matrix,arrows,decorations.pathmorphing}

\theoremstyle{plain}
\newtheorem{thm}{Theorem}[section]
\newtheorem{cor}[thm]{Corollary}

\newtheorem{lem}[thm]{Lemma}

\newtheorem{res}[thm]{Result}

\theoremstyle{definition}
\newtheorem{defn}[thm]{Definition}

\theoremstyle{remark}
\newtheorem{rem}[thm]{Remark}
\newtheorem{rems}[thm]{Remarks}

\def\R {{\mathbb R}}

\def\Q {{\mathbb Q}}

\def\N {{\mathbb N}}
\def\Z {{\mathbb Z}}

\makeatletter
\let\c@equation\c@thm

\numberwithin{equation}{section}

\title{The Lifting Problem is NP Complete}

\author{Doron Ben-Hadar}

\date{}

\makeglossaries

\newglossaryentry{DADG}
{
    name=DADG,
    description={Digital arrowed daisy graph - a way to present an arrowed daisy graph as a
    formal data type. See definitions \ref{Pref}, \ref{SurfDADG} and \ref{DADG}}
}

\newglossaryentry{ADG}
{
    name=ADG,
    description={Arrowed daisy graph - an enhanced graph structure that captures some properties
   	of the interaction graph of an oriented generic surface. More refined than a DG. Defined by
   	 Li in \cite{Li1}}
}

\newglossaryentry{DG}
{
    name=DG,
    description={Daisy graph - an enhanced graph structure that captures some properties of the
    interaction graph of a generic surface. Defined by Li in \cite{Li1}}
}

\newglossaryentry{proper}
{
    name=proper 3-sat formula,
    description={A 3-sat formula for which the literals in every clause are distinct and ordered,
    and whose clauses are distinct and lexicographically ordered. See definition \ref{reduced}}
}

\newglossaryentry{symmetric}
{
    name=symmetric 3-sat formula,
    description={A 3-sat formula whose clauses are divided into pairs of opposite clauses. See
    definition \ref{Sym3Sat}}
}

\newglossaryentry{sym&pro}
{
    name=proper symmetric 3-sat formula,
    description={A 3-sat formula that is both symmetric and proper}
}

\newglossaryentry{mirror}
{
    name=mirror clause,
    description={Two 3-clauses ``mirror" each other if they have the same variables but a
    variable is negated in one iff it is not negated in the other. Each of these clauses may be
    referred to as the mirror of the other clause. See definition \ref{Sym3Sat}}
}

\newglossaryentry{genericsurface}
{
    name=generic surface,
    description={A surface in a 3-manifold in which 2 or 3 sheets of the surface may intersect 
    transversely. May also contain branch values and regular and double boundary values. See
    definition \ref{Generic}}
}

\newglossaryentry{lifting}
{
    name=lifting,
    description={A lifting of a generic surface $S$ in a 3-manifold $M$ is a knotted surface in
    $M \times \R$ whose projection into the $M$ component is equal to $S$ \ref{Generic}. see
    definition \ref{BSD}}
}

\newglossaryentry{liftingAt}
{
    name=lifting attempt,
    description={A choice, for every double arc of a given generic surface, which of the two
    intersecting surface strips is ``higher" than the other one. A lifting attempt is successful
    if the surface has a lifting that realizes it. See definition \ref{Attempt}}
}

\newglossaryentry{BSD}
{
    name=broken surface diagram,
    description={A technique used to draw a lifting of a generic surface. See
    definition \ref{Attempt}}
}

\newglossaryentry{gradable}
{
    name=gradable,
    description={A gradable DADG is a DADG that has a grading as per definition \ref{Grading}}
}

\newglossaryentry{Height1}
{
    name=height-1,
    description={A height-1 DADG is a DADG that has no indecisive edges. See
    definition \ref{H1}}
}

\newglossaryentry{EoE}
{
    name=end of edge,
    description={One of the ends of an edge in a topological graph, or the topological
    realization of a graph-theoretical graph. The structure of a DADG distinguishes between the
    different ends of each edge by indicating which end of which edge lies on which vertex. See
    definition \ref{Pref}}
}

\newglossaryentry{consecutive}
{
    name=consecutive,
    description={Two ends of edges in the intersection graph of a generic surface are said to be
    consecutive if they lie on opposite sides of the same triple value. The structure of a DADG
    indicates which ends of edges are consecutive. See definition \ref{consec}}
}

\newglossaryentry{preferred}
{
    name=preferred,
    description={I a trice oriented generic surface, the orientation of the surface points
    towards one in each pair of consecutive ends of edges. This end of edge is said to be
    preferred. The structure of a DADG indicates which ends of edges are preferred. See
    definition \ref{Pref}. A edge is said to be preferred if both of its ends are preferred and
    non-preferred if non of it edges are preferred. See definition \ref{H1}}
}

\newglossaryentry{indecisive}
{
    name=indecisive,
    description={An edge in a DADG is said to be indecisive if one of its ends is preferred the
    other is non-preferred. See definition \ref{H1}}
}

\newglossaryentry{Thrice}
{
    name=thrice-oriented,
    description={A thrice-oriented generic surface is an oriented generic surface inside an 
    oriented 3 manifold that additionally has an orientation on each of its double arcs. See
    definition \ref{Trice}}
}

\begin{document} 

\maketitle

\pagenumbering{gobble}
\baselineskip 15pt
\tableofcontents

\baselineskip 20pt

\begin{abstract}
\addcontentsline{toc}{chapter}{Abstract}
\thispagestyle{plain}
\pagenumbering{roman}
%
Let $M$ be a 3-manifold. Every knotted (embedded) surface in $M \times \R$ can be moved via an ambient isotopy in such a way that its projection into $M$ is a \gls{genericsurface}. A surface is generic if every point on it is either a regular, double or triple value - the transversal intersection of 1, 2 or 3 embedded surface sheets, or a ``branch value" that look like Whitney's umbrella. We elaborate on this in Definition~\ref{Generic}. The double values form arcs, and along each arc two long strips of surface intersect. In a knotted surface, the additional $\R$ coordinate distinguishes between the two strips. One of them must be "higher" than the other. We elaborate on this in Definition~\ref{SurfaceStrip}.

The lifting problem is the problem of determining if a \gls{genericsurface} in $M$ can occur as the $M$-projection of a knotted surface in 4-space in $M \times \R$. The main purpose of this thesis is to study the computational aspects of the lifting problem. We will prove that the problem is NP-complete, and devise an efficient algorithm that determines if a \gls{genericsurface} is liftable.

A surface can be lifted iff one can choose, along each of the double arcs of the surface, which of the two intersecting surface strips is "higher" without arriving at some sort of obstruction. There are two obstructions that might occur. First, what locally looks like two distinct surface strips may globally ``join" into one strip. We call a double arc in which the two surface strips join ``non-trivial". We elaborate on this in Definition~\ref{Triv}. A \gls{genericsurface} that has a non-trivial double arc cannot be lifted (see Lemma \ref{TrivLem}). If the surface has no non-trivial arcs, one can attempt to lift the surface by choosing which of the two strips at each arc is the higher one.

Three double arcs intersect at each triple value. Each pair in the trio of surface sheets that meet at this value intersect transversely along a piece of one of the three arcs. When you choose a ``higher strip" at each of these arcs, you dictate which one of the pair of surface sheets is higher than the other along their intersection. This may lead to what we call a ``cyclic height relation" on the three surface sheets - the \gls{liftingAt} dictates that one sheet is higher than the second, which is higher than the third, which is higher than the first. This is a contradiction. A \gls{liftingAt} succeeds iff it does not create a cyclic height relation in any triple value, and a \gls{genericsurface} is liftable iff it has no non-trivial arcs and it has at least one successful \gls{liftingAt} (see Theorem \ref{LiftThm}).

In order to check if a surface is liftable, we match each double arc of the surface with a binary variable. We encode a \gls{liftingAt} by choosing one of its surface strips and deciding that this variable gets the value $0$ if this strip is the higher strip and the value $1$ otherwise. Then for every triple value of the surface, we show how to find integers $j(1),j(2),j(3)$ and binary values $s(1),s(2),s(3)$ such that the \gls{liftingAt} causes a cyclic height relation iff it satisfies the formula
$$(((x_{j(k)} \leftrightarrow s(k)) \vee (x_{j(k)} \leftrightarrow s(k)) \vee (x_{j(k)} \leftrightarrow s(k))) \wedge$$
$$((x_{j(k)} \leftrightarrow \neg s(k)) \vee (x_{j(k)} \leftrightarrow \neg s(k)) \vee (x_{j(k)} \leftrightarrow \neg s(k))))$$
where $x_0,...,x_{N-1}$ are our variables.

This formula is the conjunction of two ``mirror" 3-clauses (see Definition~\ref{Sym3Sat}). Each triple value provides two such clauses and a \gls{liftingAt} succeeds iff it satisfies the conjunction of all the clauses, which is a ``symmetric" 3-sat formula. We call it the lifting formula of the surface. It follows that one can see if a surface is liftable by checking if it has a non-trivial arc, compiling the lifting formulas of the surface and checking if the formula is satisfiable using any of the known 3-sat algorithms. In Chapters~\ref{SecNP} and \ref{SecTech} we will show that the first two steps take polynomial time, and so the complexity of the lifting algorithm is determined by that of the 3-sat algorithm (which is exponential) and use similar techniques to prove that the lifting problem is NP.

In order to prove that lifting problem is NP-complete we ``reverse" this process. Instead of taking a surface and producing a formula, we take a formula and produce a matching surface. We begin by proving that the ``symmetric 3-sat problem", a variant of the 3-sat problem that focus on symmetric 3-sat formulas, is NP-complete (see Theorem \ref{SymThm}). We then reduce the symmetric 3-sat problem into the lifting problem in polynomial time using a polynomial time algorithm that receives a \gls{symmetric} and produces a \gls{genericsurface} in $D^3$ (or any given 3-manifold) such that the formula is solvable iff the surface is liftable. This is done in Chapters \ref{SecMain}, \ref{SecGraph} and \ref{SecNPHrad}.

The last chapter is dedicated to a different, though related, result. In \cite{Li1}, Li showed that the double arcs of an oriented \gls{genericsurface} in an oriented 3-manifold form an enhanced graph structure which he called an ``Arrowed Daisy Graph", or ADGs. This relates to lifting because the aformentioned algorithm use a generalization of Li's ADGs which we call \gls{DADG}s (the ``D" stands for ``Digital").

In \cite{Li1}, Li left open the question ``Which ADGs can be realized by a generic immersion in $S^3$". In Chapter \ref{BHrep}, we answer a generalized version - given a 3-manifold $M$, which \gls{DADG}s can be realized by \gls{genericsurface}s in $M$.
\end{abstract}

\thispagestyle{plain}
\pagenumbering{roman}

\setcounter{page}{4}
\addcontentsline{toc}{chapter}{List of Figures}
\listoffigures

\newpage

\pagenumbering{arabic}
\chapter{Introduction}

The subject of our study is knotted surfaces, perhaps with a boundary. A knotted surface in $M \times \R$, where $M$ is a 3-manifold, is a proper embedding $k:F \to M \times \R$ of some surface $F$. In this thesis, the manifolds are PL and the embeddings are PL and locally flat. Let $\pi:M \times \R \to M$ be the projection into the first 3 dimensions. We often depict a knotted surface by drawing the projection $i=\pi \circ k$.

It is possible to perturb $k$ so that $i$ is a ``\gls{genericsurface}". This means that the intersection set $X(i)=cl\{p \in M| \#i^{-1}(p)>1\}$ consists of the ``double arcs" - lines where two sheets of the surface intersect transversely, and several types of isolated values, namely: ``triple values", which are the transverse intersection of three surface sheets at a point; ``double boundary values" (DB values for short), the transverse intersection of two surface sheets and the boundary of $M$; and ``branch values", which are cones over the figure 8 in $S^2$.

A \gls{genericsurface} that is equal to the projection of some knotted surface in 4-space is called ``liftable". There are many examples of \gls{genericsurface}s that are not liftable. Informally, the ``lifting problem" is the problem of deciding if a given \gls{genericsurface} is liftable or not. We study the algorithmic complexity of this problem. Several authors have given necessary and sufficient conditions for a \gls{genericsurface} to be ``liftable". For instance:

1) In \cite{Car&Sai1}, Carter and Saito showed that a \gls{genericsurface} is liftable iff there is an orientation on the double arcs that upholds certain properties. See also \cite{Car&SaiB1}.

2) They also showed that a \gls{genericsurface} without branch values is liftable iff the preimages of its double arcs (which are loops in $F$) can be colored a certain way.


3) In \cite{Sat1}, Satoh showed how to encode some of the topology of a neighborhood of the intersection set by presenting this set as a form of an enriched graph, with values at the various vertices and edges. He then encoded a \gls{lifting} as additional enrichment values on the graph, and showed that a surface is liftable iff there is a consistent way to add this second layer of enrichment.

4) In \cite{Gil1}, Giller (who pioneered the study of knotted surfaces) showed that a \gls{genericsurface} is liftable iff there exists a solution of a set of ``linear inequations" that is created from the surface.

Most of these conditions required the \gls{genericsurface} $i:F \to M$ to uphold several constrictions, such as $i$ being an immersion, $F$ being orientable or $M=\R^3$.

An in-depth look at conditions (1)-(3) suggests that finding out if a surface is liftable or not should take exponential time, but no study of the computational aspects of this problem has been preformed so far. In this thesis we will prove that the lifting problem of \gls{genericsurface}s is NP-complete. We will also describe an efficient algorithm that checks if a \gls{genericsurface} is liftable. The algorithm is in exponential time, but with a small exponentiation base.

Our technique involves matching each \gls{genericsurface} with a 3-sat formula, and prove that the surface is liftable iff the formula is satisfiable. For a reader who is new to computational logic, a 3-sat formula with $n$ (boolean) variables $x_1,...,x_n$ is a formula of the form $\bigwedge_{k=1}^m(y_{k,1} \vee y_{k,2} \vee y_{k,3})$ where each $y_{k,l}$ is either one of our variables $x_j$ or its negative $\neg x_j$. The 3-sat problem - ``is a 3-sat formula satisfiable?" - is one of the 20 problems proven by Karp to be NP-complete in his 1972 paper (\cite{Kar1}). 

The paper is organized as follows:

Chapter \ref{Sec3Sat} will revolve around 3-sat formulas. We will first give the necessary background about 3-sat formulas. We will also review what is currently known about the complexity of efficient 3-sat algorithms. We will then define a new variant of the 3-sat problem - the ``symmetric 3-sat problem", and prove that it is NP-complete. This is done so we can later reduce the symmetric 3-sat problem to the lifting problem in polynomial time, thus proving that the latter is NP-hard.

In chapter \ref{SecGen}, we will formally define \gls{genericsurface}s, explain in detail what it means to lift a \gls{genericsurface}, and show that there are only two kinds of obstructions that may prevent a surface from being liftable.

In chapter \ref{SecNP}, we will explain the key parts of our ``lifting algorithm". The algorithm has three steps. The first one involves checking if the surface encounters the first of the two aforementioned obstructions. If it does not, the second step is to produce a 3-sat formula, called the lifting formula of the surface. After we define the lifting formula of a surface, we will prove that this formula is satisfiable iff the surface is liftable. The third step is to use any known 3-sat algorithm to check if the lifting formula is satisfiable.

The first two steps of the algorithm take polynomial time, which implies that the complexity of the lifting algorithm is determined by that of the 3-sat algorithm. Towards the end of chapter \ref{SecNP}, we will explain the connection between the two complexities. We will also explain why the lifting problem is NP.

Chapter \ref{SecNP} is not self-contained. A formal examination of the lifting algorithm involves a lot of technical parts. These include explaining how to encode a \gls{genericsurface} as a data type that a computer can use, and how to verify that the input is a valid \gls{genericsurface}. Additionally, while the first two steps of the algorithm are simple to perform manually for a small surface, explaining how a computer does them and proving that it takes polynomial time is another technical ordeal. The same is true for the full proof that the lifting problem is NP. In order to preserve the flow of the thesis, we moved these technical parts from chapter \ref{SecNP} to their own dedicated chapter \ref{SecTech}.

A reader who wishes to skip the technical parts should be aware that there are two small parts of chapter \ref{SecTech} that are referred to in the later sections of the thesis - the formal definition of a \gls{genericsurface}, and the short section~\ref{TGsec} that revolves around graph homeomorphisms. 

In the short chapter \ref{SecMain} we explain our strategy for proving that the lifting problem is NP-hard, and formulate the main theorem (Theorem~\ref{Thm3}). In general terms, our strategy involves reducing the symmetric 3-sat problem into the lifting problem in polynomial time. This means devising a polynomial time algorithm that receives a \gls{symmetric} and produces a \gls{genericsurface} such that the surface is liftable iff the formula is satisfiable. To guarantee this, we will ensure that the lifting formula of the surface will be equivalent to the given formula.

The lifting formula of a \gls{genericsurface} is determined by the topology of the intersection graph of the surface and its close neighborhood. In \cite{Li1} and \cite{Sat1}, Li and Satoh (respectively) defined enriched graph structures on the intersection graph that encode the topology of its neighborhood. In chapter \ref{SecGraph}, we will use a structure very similar to Li's ``arrowed daisy graphs", which we will call ``digital arrowed daisy graph" (\gls{DADG}), to encode this information in a more general setting. Unlike Li and Satoh, we define \gls{DADG}s as a formal data type that can be used by a computer, so that we can use them in algorithms.

We will prove that the \gls{DADG} structure of an orientable \gls{genericsurface} determines its lifting formula, and show how to deduce the formula from the \gls{DADG}. We will use this to give an alternative definition for the lifting formula that relies on the \gls{DADG} alone, without involving the surface. We will call this the ``graph lifting formula" of the \gls{DADG}.

In chapter \ref{SecNPHrad}, we will devise the algorithm referred to in the main theorem. This algorithm has two distinct steps. Firstly, the algorithm produces a \gls{DADG} whose graph lifting formula is equivalent to the given formula. It will then produce an orientable \gls{genericsurface} whose \gls{DADG} is equal to the \gls{DADG} produced in the previous step. We will prove that both these steps take polynomial time.

Not every \gls{DADG} can be realized with a \gls{genericsurface}. Furthermore, the surface-producing algorithm may not work even for a \gls{DADG} that can be realized. It only works on a special kind of \gls{DADG}, which we refer to as a ``\gls{Height1}" \gls{DADG}. Chapter \ref{SecNPHrad} also contains the definition of \gls{Height1} \gls{DADG}s, and a proof that the \gls{DADG}s produced by the algorithm are all \gls{Height1}.

In the final chapter, \ref{BHrep}, we will answer the question ``which \gls{DADG}s are realizable via an oriented \gls{genericsurface} in a given orientable 3-manifold $M$". This is a generalization of an open question posted by Li in \cite{Li1}. The answer depends solely on the first homology group $H_1(M;Z)$, and whether $M$ has a boundary. We note that the result of chapter \ref{BHrep} have been submitted for publication as the article \cite{BH1}.

\chapter{3-sat Formulas} \label{Sec3Sat}

In order to prove that the lifting problem is NP-complete, we will reduce it to a variant of the 3-sat problem which we call ``the symmetric 3-sat problem" - the problem of determining if a \gls{symmetric} is solvable.

In this chapter, we will provide the background about 3-sat formulas and the 3-sat problem required for this work, which includes emphasizing some nuances that others usually ignore, but are relevant here. We will then rigorously define symmetric 3-sat formulas, and the symmetric 3-sat problem and prove that the aforementioned problem is NP-complete.

\section{Background}

In this section, we will provide the background, and explain some nuances about 3-sat formulas required for this thesis.

\begin{defn} \label{3Sat}
In the framework of propositional calculus with variables $x_0,x_1...$:

1) A literal is either just a variable $x_i$, in which case it is called a positive literal, or the negation of a variable $\neg x_i$, in which case it is called a negative literal.

2) A 3-clause is the disjunction of 3 literals. For instance, $x_3 \vee x_4 \vee \neg x_7$ and $x_1 \vee \neg x_2 \vee \neg x_1$ are 3-clauses. 

3) A 3-sat formula $F$ is the conjunction of some number $K$ of 3-clauses. \\$F=\bigwedge_{k=0}^{K-1}F_k$ where each $F_k$ is a 3-clause. We denote the literals in $F_k$ $y_{k,1},y_{k,2},y_{k,3}$ so $F_k = y_{k,1} \vee y_{k,2} \vee y_{k,3}$. It follows that $F=\bigwedge_{k=0}^{K-1}(y_{k,1} \vee y_{k,2} \vee y_{k,3})$. The number $K$ of clauses is called the ``length" of $F$.

4) Clearly, each literal $y_{k,l}$ in a 3-sat formula $F$ has either the form $x_{j(k,l)}=x_{j(k,l)} \leftrightarrow 1$ or $\neg x_{j(k,l)}=x_{j(k,l)} \leftrightarrow 0$ where $j(k,l)$ is the index of the variable that appears in this literal. Note that the indexes $k$ and $l$ have a different purpose than $j=j(k,l)$. $k$ and $l$ indicate the position of the literal in the formula - it is the $l$th literal in the $k$th clause. $j(k,l)$ tells us which variable among $x_0,x_1,...$ appears in this literal.

We call $j(k,l)$ the ``index" of the $k,l$ literal. We refer to the collection of all $j(k,l)$'s as the ``index function" of the formula, since one can think of it as a function that associates each $k,l$ with the index $j(k,l)$ of the appropriate variable.

5) In general, the literal $a(k,l)$ has the form $x_{j(k,l)} \leftrightarrow s(k,l)$ where the parameter $s(k,l)$ is equal to 0 or 1 in correspondence to whether the literal is negative or positive. We call $s(k,l)$ the parameter of the $k,l$ literal.
\end{defn}

\begin{rem}
1) We will usually forgo naming the literals of a 3-sat formula and will not use the notation $y_{k,l}$. Instead, we will define the formula using its index function and parameters - $F=\bigwedge_{k=0}^{K-1}((x_{j(k,1)} \leftrightarrow s(k,1)) \vee (x_{j(k,2)} \leftrightarrow s(k,2)) \vee (x_{j(k,3)} \leftrightarrow s(k,3)))$. If we want to refer to the $k,l$'th literal we will simply write $x_{j(k,l)} \leftrightarrow s(k,l)$.

2) The number $K$ of clauses really indicates the length of the formula. The total numbers of variables, logical connectives and brackets in the formula are all $\Theta(K)$. Additionally, the number of variables used in the formula, $N$, is bounded from above by $3K$.

3) One can similarly define an $r$-clause to be the disjunction of $r$ literals, and an
$r$-sat formula to be the conjunctions of $r$-clauses.
\end{rem}

When different authors define a 3-sat formula / the 3-sat problem, they may use a stricter definition than the above. They may require every clause to have distinct literals - that the same literal will not appear twice in the same clause. They may also require 3-sat formulas to have distinct clauses - that the same clause will not appear more than once in the formula. In this work,it will be useful to carefully distinguish between different variants of the 3-sat problem. We achieve this by employing the following, \textbf{non-standard} notation:

\begin{defn} \label{reduced}
1) We give the set of all the potential literals of a formula the following strong linear order $x_0 \prec \neg x_0 \prec x_1 \prec \neg x_1 \prec x_2 \prec ...$. In other words, $x_i \leftrightarrow s \prec x_j \leftrightarrow t$ iff $i<j$ or $i=j$, $s=1$ and $t=0$. It has a matching weak linear order $\preceq$.

2) We say that a 3-sat formula is called ``reduced" if:

a) The literals in every clause are ordered according to $\preceq$. This means that for every $k=0,...,K-1$, the following order $x_{j(k,1)} \leftrightarrow s(k,1) \preceq x_{j(k,2)} \leftrightarrow s(k,2) \preceq x_{j(k,3)} \leftrightarrow s(k,3)$.

b) The clauses themselves are ordered according to the lexicographic order (on 3-tuples) induced by $\preceq$. This means that for every $k=0,...,K-2$, either $x_{j(k,1)} \leftrightarrow s(k,1) \prec x_{j(k+1,1)} \leftrightarrow s(k+1,1)$ or $j(k,1)=j(k+1,1)$, $s(k,1)=s(k+1,1)$ and $x_{j(k,2)} \leftrightarrow s(k,2) \prec x_{j(k+1,2)} \leftrightarrow s(k+1,2)$ or $j(k,1)=j(k+1,1)$, $s(k,1)=s(k+1,1)$, $j(k,2)=j(k+1,2)$, $s(k,2)=s(k+1,2)$ and $x_{j(k,3)} \leftrightarrow s(k,3) \preceq x_{j(k+1,3)} \leftrightarrow s(k+1,3)$.

c) The clauses are distinct - no two clauses in the formula can be equal (have the exact same literals).

3) A clause of a reduced 3-sat formula may include the same variable more than once. For example, clauses like $x_1 \vee x_3 \vee \neg x_3$, $x_5 \vee x_5 \vee x_5$ or $x_1 \vee x_1 \vee x_2$ can occur in a reduced 3-sat formula. We say that 3-sat formula is called ``proper" if it is reduced and the 3 variables in every clause are distinct.
\end{defn}

Conceptually, the ``3-sat problem" is the question: ``Given a 3-sat formula, is it satisfiable?". A 3-sat solving algorithm receives a 3-sat formula as an input, and returns ``Yes" if the formula is satisfiable and ``No" otherwise. The 3-sat problem has several variants depending on one's definition of a 3-sat formula. To be precise, we will add the distinction between these variants. For instance, the ``reduced 3-sat problem" is the problem of determining whether a reduced 3-sat formula is satisfiable. The only difference between it and the general 3-sat problem is that an algorithm that solves the reduced 3-sat problem has a smaller set of potential inputs - its input must be a reduced 3-sat formula. A priori, the reduced 3-sat problem could be computationally simpler than the general one - there might be a fast algorithm that checks if a reduced 3-sat formula is satisfiable, but does not work for general 3-sat problems.

Actually, all variants of the 3-sat problems we defined so far are considered to have the same complexity. In particular, they are all known to be NP-complete. An NP-complete problem is a problem that is both NP and NP-hard.

A problem is said to be NP if there is a polynomial-time algorithm that receives an input (in our case, a 3-sat formula $F$ of the given variant) and a ``possible solution of the problem", known as a certificate (in our case, a choice of binary value for each of the variables $x_0,...,x_{N-1}$, represented by a vector $x=(x_0,...,x_{N-1})$) and returns ``yes" if the certificates solves the problem (in our case, if the chosen values satisfy the formula) and ``no" otherwise.

Such an algorithm obviously exists for the 3-sat problem. The algorithm simply places the values in the formula and performs the extrapolation. Since the length of the formula is $\Theta(K)$, this takes linear $O(K)$ time, and is in particular polynomial-time. Notice that the size of the certificate is also $N=O(K)$. It follows that all variants of the 3-sat problem are NP.

NP-hardness is a more difficult matter. A decision problem $A$ is said to be NP-hard if any NP problem $B$ can be reduced to $A$ in polynomial time. This means that there is a polynomial time algorithm that receives a possible input $p$ of $A$, and returns a possible input $q$ of $B$, such that ``$A$ should return yes to $p$" iff ``$B$ should return yes to $q$". The usual method one uses to prove that problem $B$ is NP hard is to take another problem $C$ that is already known to be NP-hard, and show that $C$ can be reduced to $B$ in polynomial time. It would imply that every NP problem $A$ can be reduced to $C$ and then to $B$ in polynomial time. A detailed proof as to why this works was given in \cite{Kar1}, where Richard Karp proved that 20 known computational problems are NP-complete. Among these was the 3-sat problem (his proof works for all the variants of the problems that we described thus far).

Before we finish this chapter, we would like to discuss the complexity of the 3-sat problem in explicit terms - how fast is an efficient 3-sat solving algorithm. 

\begin{rem} \label{reduce}
1) If one attempts to solve the general 3-sat problem, then the first step of the algorithm should be to ``reduce" the given formula - ordering the literals in each clause, ordering the clauses, and deleting repeated occurrences of the same clause. These tasks take $O(K)$ time, $O(K \cdot \log(K))$ time, and $O(K)$ time, respectively. The algorithm then proceeds by solving the reduced formula, which is equivalent to the original formula.

The reduced formula has at most $8N^3$ clauses - $2N$ possible different literals to the power of 3. $8N^3$ is not a tight bound. One should really only consider clauses with ordered literals, and there are ways to reduce this number farther - for instance, a clause that contains the literals $x_j$ and $\neg x_j$ is a tautology and can be removed from the formula. However, the amount of possible clauses cannot be reduced below $O(N^3)$, even if one restricts the input to only allow \gls{proper}s - the most limiting case.

2) The next step will be an algorithm that solves the reduced 3-sat problem. Since in this case $\frac{1}{3}N \leq K \leq 8N^3$, it is common to use $N$, instead of $K$, as a measure to the size of the formula, and the problem remains NP-complete with regards to $N$ as the size parameter.

The fastest known algorithms to solve the 3-sat problem have an exponential run-time $O(c^N)$. Finding a faster algorithm, if one exists, will be a major achievement in the theory of computation field (it will contradict the exponential time hypothesis). The efficiency of a 3-sat solving algorithm is thus determined by the exponential base $c$. At the moment, it seems that new algorithms with smaller $c$s are being discovered yearly. The fastest that we are aware of was given by Kutzkov and Schederin in \cite{Kut&Sch1}, for which $c \approx 1.439$.

There are also algorithms that are designed to have a faster average run-time or expected run-time in return for a slower worst-case run-time. This means that, depending on the 3-sat formula given as input, the algorithm will usually run faster but may be slow for some small percentage of the possible inputs. These algorithms still have an exponential expected run-time $O(c^N)$, but the exponential base $c$ is smaller than even the most efficient worst-case run-times discovered so far. For instance, in \cite{Hom&Sch&Schu&Wat1}, Hofmeister, Sch\"{o}ning, Schuler and Watanabe devised a 3-sat solving algorithm for which the exponential base $c$ of the expected run-time is $c \approx 1.3302$.  
\end{rem}

In the next section we will define a new variant of the 3-sat problem, the ``symmetric 3-sat problem", and prove that it too is NP-hard.

\section{Symmetric 3-sat formulas} \label{Sym3Sec}
We define a new variant of the 3-sat formula:

\begin{defn} \label{Sym3Sat}
1) Given a literal $x_j \leftrightarrow s$, the ``mirror literal" is $\neg (x_j \leftrightarrow s) = \neg x_j \leftrightarrow s = x_j \leftrightarrow \neg s$. It has the same variable but with the opposite parameter.

2) Given a 3-clause $\bigvee_{i=1}^3 (x_j(i) \leftrightarrow s_i)$, its ``mirror 3-clause" is the 3-clause that uses the mirror literals - $\bigvee_{i=1}^3 (\neg x_j(i) \leftrightarrow s_i)=\bigvee_{i=1}^3 \neg (x_j(i) \leftrightarrow s_i)=\bigvee_{i=1}^3 (x_j(i) \leftrightarrow \neg s_i)$. For instance, the mirror clause of $x_3 \vee x_5 \vee \neg x_8$ is $\neg x_3 \vee \neg x_5 \vee x_8$. The mirror of the \gls{mirror} is clearly the original clause. Therefore, the set of 3-clauses divides into pairs of mirror clauses. One can similarly define ``mirror $r$-clauses" for any $r \in \N$.

3) A 3-sat formula is symmetric if for every clause in the formula, the \gls{mirror} is also in the formula, and it appears in the formula the same number of times as the given clause. For instance, the formula $(x_1 \vee x_3 \vee x_4) \wedge (\neg x_1 \vee \neg x_3 \vee \neg x_4)$ is symmetric, but the formula $(x_1 \vee x_2 \vee x_5) \wedge (\neg x_1 \vee \neg x_3 \vee \neg x_5) \wedge (\neg x_1 \vee x_3 \vee \neg x_4)$ is not.
\end{defn}

\begin{rem}
If the 3 variables of a 3-clause are all different, then the same will hold for its \gls{mirror}. In addition, if the variables are arranged in increasing order (for instance $x_1 \vee x_3 \vee x_4$ as opposed to $x_3 \vee x_4 \vee x_1$), then the same will hold for the \gls{mirror}. These are the precise requirements that a clause must uphold in order to appear in a \gls{proper}.

This leads us to the \gls{sym&pro}. Since \gls{proper}s may only have one copy of the same clause, a \gls{proper} is symmetric iff, for each clause of the formula, the formula also contains its \gls{mirror}.
\end{rem}

We study a new variant of the 3-sat formula - the ``proper symmetric 3-sat problem". This is the decision problem: ``Given a symmetric and proper 3-sat problem, is it satisfiable?". The ``symmetric 3-sat problem" is similarly defined. These problems are NP as are all variants of the 3-sat problem. The remainder of this chapter is dedicated to proving that:

\begin{thm} \label{SymThm}
The proper symmetric $3$-sat problem is NP-hard (and thus NP-complete). 
\end{thm}

In order to prove this, we will reduce the usual ``proper 3-sat problem", which is known to be NP-hard, to the ``proper symmetric 3-sat problem" in polynomial time. This means that, given a \gls{proper} $F$, we will produce a \gls{sym&pro} $F_{sym}$, such that $F_{sym}$ is satisfiable iff $F$ is satisfiable. This will clearly also prove that:

\begin{res} \label{SymRes}
The symmetric $3$-sat problem is NP-hard (and thus NP-complete). 
\end{res}

The reason all this is done, as will be seen later on, is that symmetric 3-sat formulas arise naturally in the context of \gls{lifting}s of \gls{genericsurface}s. In particular, Result~\ref{SymRes} is used to prove that the lifting problem is NP-hard.

Defining the said formula $F_{sym}$, and proving that it is satisfiable iff $F$ is satisfiable, requires several definitions and computations in propositional calculus. In order to start, we will need to use the equivalence provided below. This equivalence can probably be found in some textbooks, but it is not as elementary or commonly known as, for instance, DeMorgan's law, so we will prove it here:

\begin{equation} \label{Equiv1}
(x\wedge y) \vee (\neg x \wedge z) \equiv (x\vee z) \wedge (\neg x \vee y)
\end{equation}.

\begin{proof}
First of all, it is clear that

$$(x\vee z) \wedge (\neg x \vee y) \to (x\vee z) \vee (\neg x \vee y) \equiv (\neg x \vee x) \vee (y \vee z) \equiv y \vee z.$$

Using this and $(\neg x \vee x) \equiv 1$, deduce that:

$$(\neg x\vee y) \wedge (x\vee z) \equiv (y\vee \neg x) \wedge (x\vee z) \wedge (y \vee z) \wedge (x \vee \neg x).$$

Distributivity shows that the latter is equivalent to $(x`\wedge y) \vee (\neg x \wedge z)$.
\end{proof}

The first step in producing a \gls{symmetric} from an arbitrary 3-sat formula $F$, is to ``symmetrize" it as per the following definition:

\begin{defn} \label{Symmed}
1) Given a formula $F$ in $n$ variables $x_0,...,x_{n-1}$, we define the \textbf{mirror} formula $F_{\neg}$ to be $F_{\neg}(x_0,...,x_{n-1}) \equiv F(\neg x_0,...\neg x_{n-1})$. $F_{\neg}$ uses the same variables as $F$.

2) Given a formula $F$ in $n$ variables $x_0,...,x_{n-1}$, we define the \textbf{Symmetrized} formula $F_{\vee}$ to be $F_{\vee} \equiv (z \vee F) \wedge (\neg z \vee F_{\neg})$. $F_{\vee}$ uses the same variables as $F$ plus a new variable $z$.
%
\end{defn}

The equivalence (\ref{Equiv1}) implies that:

\begin{res}
Given a formula $F$, $F_{\vee}$ is equivalent to $(\neg z \wedge F) \vee (z \wedge F_{\neg})$.
\end{res}

Each of the following properties is trivial, or follows immediately from the earlier properties.

\begin{rem} \label{SymmedSolve}
1) A certificate $x=(x_0,...,x_{n-1})$ satisfies a formula $F$ iff $\neg x=(\neg x_0,...,\neg x_{n-1})$ satisfies the mirror formula $F_{\neg}$ iff $x'=(x_0,...,x_{n-1},0)$ satisfies that symmetrized formula $F_{\vee}$ iff $\neg x'=(\neg x_0,...,\neg x_{n-1},1)$ satisfies that symmetrized formula $F_{\vee}$.

2) In particular, $F$ is satisfiable iff $F_{\neg}$ is satisfiable iff $F_{\vee}$ is satisfiable.

%
3) The ``mirror" functional commutes with the elementary logical connections. Formally, if $F_1,F_2,...,F_5$ are formulas and $F_3 \equiv \neg F_1$, $F_4 \equiv F_1 \vee F_2$, and $F_5 \equiv F_1 \wedge F_2$, then $F_{3\neg} \equiv \neg F_{1\neg}$, $F_{4\neg} \equiv F_{1\neg} \vee F_{2\neg}$, and $F_{5\neg} \equiv F_{1\neg} \wedge F_{2\neg}$.

4) In particular, if a formula $F$ does not use the variable $z$, and $H \equiv z \vee F$, then $H_{\neg} \equiv \neg z \vee F_{\neg}$ and $F_{\vee} \equiv H \wedge H_{\neg}$.

5) Also, if $F_0,...,F_{K_1}$ are formulas, then, by induction, the mirror formula of $(\bigwedge_{k=0}^{K-1}F_k)$ is $(\bigwedge_{k=0}^{K-1}F_{k\neg})$.

6) The mirror formula of a clause $\bigvee (x_i \leftrightarrow s_i)$ is the \gls{mirror} from Definition~\ref{Sym3Sat}(2). $\bigvee (\neg x_i \leftrightarrow s_i) \equiv \bigvee \neg (x_i \leftrightarrow s_i) \equiv \bigvee (x_i \leftrightarrow \neg s_i)$.
\end{rem}

Let $F \equiv \bigwedge_{k=0}^{K-1} F_k$ be a 3-sat formula where the $F_k$'s are 3-clauses. Let $H \equiv z \vee F \equiv  \bigwedge_{k=0}^{K-1}(z \vee F_k)$, then Remark~\ref{SymmedSolve}(4) implies that:
$$F_{\vee} \equiv H \wedge H_{\neg} \equiv \bigwedge_{k=0}^{K-1}(z \vee F_k) \wedge \bigwedge_{k=0}^{K-1}(z \vee F_k)_{\neg} \equiv \bigwedge_{k=0}^{K-1}((z \vee F_k) \wedge (z \vee F_k)_{\neg}).$$

Unfortunately, this is not a \gls{symmetric}. It is a symmetric 4-sat formula. If each $F_k$ is a 3-clause, and has the form $(x_{j(k,1)} \leftrightarrow s(k,1)) \vee (x_{j(k,2)} \leftrightarrow s(k,2)) \vee (x_{j(k,3)} \leftrightarrow s(k,3))$, then $(z \vee F_k)$ is the 4-clause $(z \leftrightarrow 1) \vee (x_{j(k,1)} \leftrightarrow s(k,1)) \vee (x_{j(k,2)} \leftrightarrow s(k,2)) \vee (x_{j(k,3)} \leftrightarrow s(k,3))$, and $(z \vee F_k)_{\neg}$ is its \gls{mirror}, as per Remark~\ref{SymmedSolve}(5).

We now have a symmetric 4-sat formula, $F_{\vee}$, that is satisfiable iff $F$ is satisfiable, and we can deduce the solutions of $F$ from those of $F_{\vee}$ using Remark~\ref{SymmedSolve}(1). In order to modify it into a (symmetric) 3-sat formula, we use the following lemma:

\begin{lem} \label{Symmed3Sat}
Let $a_1,a_2,a_3,a_4$ and $b$ be variables and $A \equiv (a_1 \vee a_2 \vee a_3 \vee a_4)$. The formula $B \equiv (a_1 \vee a_2 \vee b) \wedge (\neg a_1 \vee \neg a_2 \vee \neg b) \wedge (\neg b \vee a_3 \vee a_4) \wedge (b \vee \neg a_3 \vee \neg a_4)$ is equivalent to $A \wedge A_{\neg}$. In other words, a valuation $(a_1,..,a_4)$ satisfies $A \wedge A_{\neg}$ iff either $(a_1,..,a_4,1)$ or $(a_1,..,a_4,0)$ satisfy $B$.
\end{lem}

\begin{proof}
Let $C \equiv (a_1 \vee a_2) \wedge (\neg a_3 \vee \neg a_4)$. Distributivity shows that $b \vee C \equiv (a_1 \vee a_2 \vee b) \wedge (b \vee \neg a_3 \vee \neg a_4)$, and a similar calculation for $\neg b \vee C_{\neg}$ shows that $B \equiv (b \vee C) \wedge (\neg b \vee C_{\neg}) \equiv C_{\vee}$.

Distributivity also shows that:
\begin{multline*}
A \wedge A_{\neg} \equiv (a_1 \vee a_2 \vee a_3 \vee a_4) \wedge (\neg a_1 \vee \neg a_2 \vee \neg a_3 \vee \neg a_4) \equiv \\
((a_1 \vee a_2) \wedge (\neg a_3 \vee \neg a_4)) \vee ((\neg a_1 \vee \neg a_2) \wedge (a_3 \vee a_4)) \equiv C \vee C_{\neg}.
\end{multline*}

According to Remark~\ref{SymmedSolve}(2), $B \equiv C_{\vee}$ is satisfiable iff $C$ is satisfiable iff $C_{\neg}$ is satisfiable iff either $C$ or $C_{\neg}$ is satisfiable iff $A \equiv C \vee C_{\neg}$ is satisfiable. Additionally, a valuation $(a_1,..,a_4)$ solves $A \equiv C \vee C_{\neg}$ iff it either solves $C$ or solves $C_{\neg}$. According to Remark~\ref{SymmedSolve}(1), the former case holds iff $(a_1,..,a_4,1)$ solves $B \equiv C_{\vee}$ and the latter case holds iff $(a_1,..,a_4,0)$ solves $B \equiv C_{\vee}$.
\end{proof}

Now, look back at the symmetrized formula $F_{\vee} \equiv \bigwedge_{k=0}^{K-1}((z \vee F_k) \wedge (z \vee F_k)_{\neg})$. A valuation $(x_0,...,x_{N-1},z)$ solves this formula iff it satisfies the expression $(z \vee F_k) \wedge (z \vee F_k)_{\neg}$ for all $k$.

For each $k$, $z \vee F_k \equiv z \vee (x_{j(k,1)} \leftrightarrow s(k,1)) \vee (x_{j(k,2)} \leftrightarrow s(k,2)) \vee (x_{j(k,3)} \leftrightarrow s(k,3))$ and $(z \vee F_k)_{\neg} \equiv \neg z \vee \neg (x_{j(k,1)} \leftrightarrow s(k,1)) \vee \neg (x_{j(k,2)} \leftrightarrow s(k,2)) \vee \neg (x_{j(k,3)} \leftrightarrow s(k,3))$.

Using Lemma~\ref{Symmed3Sat} with $a_1 \equiv z, a_2 \equiv (x_{j(k,1)} \leftrightarrow s(k,1)), a_3 \equiv (x_{j(k,2)} \leftrightarrow s(k,2))$ and $a_4 \equiv (x_{j(k,3)} \leftrightarrow s(k,3))$ shows that a valuation $(x_0,...,x_{N-1},z)$ satisfies $(z \vee F_k) \wedge (z \vee F_k)_{\neg}$ iff, for some new variable $y_k$, either the valuation $(x_0,...,x_{N-1},z,1)$ or $(x_0,...,x_{N-1},z,0)$ satisfies:
\begin{multline*}
B_k \equiv (z \vee (x_{j(k,1)} \leftrightarrow s(k,1)) \vee y_k) \wedge (\neg z \vee \neg (x_{j(k,1)} \leftrightarrow s(k,1)) \vee \neg y_k) \wedge \\ (\neg y_k \vee (x_{j(k,2)} \leftrightarrow s(k,2)) \vee (x_{j(k,3)} \leftrightarrow s(k,3))) \wedge  \\  (y_k \vee \neg (x_{j(k,2)} \leftrightarrow s(k,2)) \vee \neg (x_{j(k,3)} \leftrightarrow s(k,3))).
\end{multline*}

After some bracket-moving we see that:
\begin{multline} \label{BK}
B_k \equiv ((x_{j(k,1)} \leftrightarrow s(k,1)) \vee (z \leftrightarrow 1) \vee (y_k \leftrightarrow 1)) \wedge \\ ((x_{j(k,1)} \leftrightarrow \neg s(k,1)) \vee (z \leftrightarrow 0) \vee (y_k \leftrightarrow 0)) \wedge \\ ((x_{j(k,2)} \leftrightarrow s(k,2)) \vee (x_{j(k,3)} \leftrightarrow s(k,3)) \vee (y_k \leftrightarrow 0)) \wedge \\ ((x_{j(k,2)} \leftrightarrow \neg s(k,2)) \vee (x_{j(k,3)} \leftrightarrow \neg s(k,3)) \vee (y_k \leftrightarrow 1)))
\end{multline}
is a collection of four 3-clauses.

In general, a valuation $(x_0,...,x_{N-1},z)$ solves the symmetrized formula $F_{\vee} \equiv \bigwedge_{k=0}^{K-1}((z \vee F_k) \wedge (z \vee F_k)_{\neg})$ iff it can be extended into a bigger valuation $(x_0,...,x_{N-1},z,y_0,..,y_{K-1})$, that satisfies the 3-sat formula $F_{sym} \equiv \bigwedge_{k=0}^{K-1}(B_k)$ (it has $4K$ clauses). Recall that this happens iff either $x=(x_0,...,x_{N-1})$ or $\neg x=(\neg x_0,...,\neg x_{N-1})$ solves the original 3-sat formula $F$.

Using this, we can prove Theorem~\ref{SymThm}

\begin{proof}
Since the ``regular" proper 3-sat problem is NP-complete, the theorem can be proven by reducing it to the \gls{sym&pro}. This means providing a polynomial time algorithm that receives a \gls{proper} $F \equiv \bigwedge_{k=0}^{K-1}F_k$, where $F_k \equiv (x_{j(k,1)} \leftrightarrow s(k,1)) \vee (x_{j(k,2)} \leftrightarrow s(k,2)) \vee (x_{j(k,3)} \leftrightarrow s(k,3))$, and produces a \gls{sym&pro} $G$ such that $G$ is satisfiable iff $F$ is satisfiable. 

As per Remark~\ref{reduce}(2), the size of $F$ is indicated by either $K$ or $N$ (the number of variables), and ``polynomial time" can mean either $O(N^a)$ or $O(K^a)$ - these coincide for \gls{proper}s.

The intuitive candidate for $G$ is $F_{sym} \equiv F_{sym} \equiv \bigwedge_{k=0}^{K-1}(B_k)$. We have already shown that $F$ is satisfiable iff $F_{sym}$ is satisfiable, and that $F_{sym}$ is a \gls{symmetric} with $4K$ clauses, so writing it will take $O(K)$ - polynomial time. But is $F_{sym}$ a \gls{proper}?

It is true that each clause of $F_{sym}$ includes 3 distinct variables: For each $k$, $F_{sym}$ contains the 4 clauses seen in formula (\ref{BK}). Two of them, $x_{j(k,1)} \leftrightarrow s(k,1)) \vee (z \leftrightarrow 1) \vee (y_k \leftrightarrow 1)$, and $(x_{j(k,1)} \leftrightarrow \neg s(k,1)) \vee (z \leftrightarrow 0) \vee (y_k \leftrightarrow 0)$, have the clearly distinct variables $x_{j(k,1)}$, $z$ and $y_k$. The other two, $(x_{j(k,2)} \leftrightarrow s(k,2)) \vee (x_{j(k,3)} \leftrightarrow s(k,3)) \vee (y_k \leftrightarrow 0)$, and $(x_{j(k,2)} \leftrightarrow \neg s(k,2)) \vee (x_{j(k,3)} \leftrightarrow \neg s(k,3)) \vee (y_k \leftrightarrow 1)$ have the variables $x_{j(k,2)}$, $x_{j(k,3)}$ and $y_k$. The last one is clearly different from the first two, and these two are different since the original 3-sat formula $F$ is proper.

A \gls{proper} also needs to be reduced. This entails 3 requirements:
 
Firstly, the literals in every clause must be ordered. Since not all our variables have the form $x_0,x_1...$, we will need to specify an order for the variables - first the $x_j$s, then $z$ and then the $y_k$s, resulting in the following order on literals - $x_0 \prec \neg x_0 \prec x_1 \prec \neg x_1 \prec ...\prec z \prec \neg z \prec y_0 \prec \neg y_0 \prec y_1 \prec..$. The definition of $B_k$ (\ref{BK}), combined with the fact that for every $k$, $j(k,1)<j(k,2)<j(k,3)$ (since $F$ is proper) implies that each of the clauses in each $B_k$ is ordered.

Secondly, it is clear that the different clauses are all distinct. Clauses from different $B_k$s will be different, since they will include different $y_k$s and clauses from the same $B_k$ can be seen to be different.

The only remaining requirement is for the clauses to be ordered in lexicographic order. This may actually not hold, but we can just reorder the clauses of $F_{sym}$ - replace it with a new formula that has the exact same clauses, but in a different order - in $O(4K \log(4K))=O(K \log(K))$ time. The reordered formula will still be symmetric, as this property does not depend on the order of the clauses.

To summarize, the algorithm that receives $F$, writes $F_{sym}$, and then reorders its clauses, takes $O(K)+O(K \log(K))=O(K \log(K))$ time, which is less than $O(K^2)$, and it reduces the (NP-complete) proper 3-sat problem to the proper symmetric 3-sat problem. The theorem follows.
\end{proof}

\chapter{Generic Surfaces and Liftings} \label{SecGen}

In this chapter we will provide some background about the lifting problem. In the first section, we will explain what is a \gls{genericsurface}, how to lift a \gls{genericsurface} into a knotted surface, and how to draw \gls{genericsurface}s and \gls{lifting}s. In the second section, we will explain why some \gls{genericsurface}s have \gls{lifting}s and others do not. Particularly, we will demonstrate two obstructions that prevent a surface from being liftable, and prove that these are the only ``obstructions to liftability" - that a surface unhindered by these obstructions is indeed liftable.

\section{Preliminaries}

In order to draw a knot in 3-space, its projection needs to be drawn in 2-space. When two strands of the projected loop intersect, one indicates which of the two went above the other one before the projection. For this definition to work, one makes sure that the projection is ``generic" - that is to say, that no more than two strands of the projected loop intersect at the same point, and that this loop lacks any kind of singularity. In order to draw a knotted surface in 4-space, one should similarly draw its projection into 3-space and make sure that it is generic in the following sense:

\begin{defn} \label{Generic}
A proper map $i:F \to M$ from a compact surface $F$ to a 3-manifold $M$ is called a ``\gls{genericsurface}" in $M$ if each value $p \in M$ has a neighborhood $N(p)$ such that the pair $(N(p),N(p) \cap i(F))$ is homeomorphic to one of the following:

1) ($D^3$, the transverse intersection of 1, 2 or 3 of the coordinate planes) where $D^3$ is a ball in $\R^3$ centred at $0$. We refer to these respectively as regular, double and triple values.

2) ($D^3$, a cone over the figure 8) where the figure 8 curve is on the boundary of $D^3$. This is the image of the smooth ``Whitney's umbrella" function $(x,y)\mapsto(x,xy,y^2)$. We refer to such values as branch values. In literature, they are sometimes known as cross-caps, figure 8 cones, etc. Note that we mostly work with triangulated manifolds, and so the neighborhoods of branch values will actually be the images of PL approximations of Whitney's umbrella.

3) ($D^3_+$, the transverse intersection of the one or two of the $[xz]$ and $[yz]$ coordinate planes) where $D^3_+$ is the ``upper half" of the ball $D^3$ - the part where $z \geq 0$. We refer to these respectively as regular boundary values and double boundary values, or RB and DB values for short.
\end{defn}

The left images in Figure~\ref{fig:Figure 1} are illustrations of a double value, a triple value, a branch value and a DB value. Since two surface sheets intersect at a line, the double values of $i$ form long arcs, called ``double arcs". In Figure~\ref{fig:Figure 1}B (left), we see that three segments of double arc intersect at each triple value. These can be parts of the same arc or different arcs, which implies that each double arc is an immersed, but not necessarily embedded, 1-manifold in $M$.

A double arc may have either a DB value or a branch value at each of its ends, as in Figures~\ref{fig:Figure 1}C and D (left). It is also possible that the arc will close into a circle. We refer to arcs of the former kind as ``open" and arcs of the latter kind as ``closed". In particular, the union of all double arcs is equal to the set of all double, triple, DB, and branch values, and it is also equal to the intersection set $cl\{p \in M| \#i^{-1}(p)>1\}$. We denote this set $X(i)$ and we also refer to it as the ``intersection graph" for reasons that we will explain later on.

\begin{figure}
\begin{center}
\includegraphics{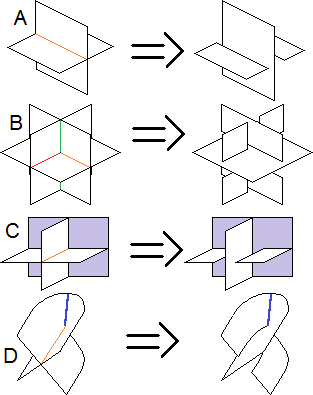}
\caption{How to draw a \gls{genericsurface} and / or \gls{BSD}}
\label{fig:Figure 1}
\end{center}
\end{figure}

The knotted surfaces we regard in this thesis are proper 1-1 PL functions from a compact surface $F$ into $M \times \R$. Such functions have the form $k=(i,h)$ where $i:F \to M$ and $h:F \to \R$. An arbitrarily small perturbation can turn $i$ into a \gls{genericsurface}. This was proven by Izumiya and Marar in \cite{Izu&Mar2} for the case where $F$ is closed and $M$ is boundaryless, and the proof readily extends to the case where $F$ and $M$ have boundaries. Since $k$ will remain 1-1 after a sufficiently small perturbation to $i$, this implies that every knotted surface can be perturbed into a surface $k=(i,h)$ for which $i$ is generic. One may think of such a knotted surface as a \gls{lifting} of the \gls{genericsurface} $i$.

\begin{defn} \label{BSD}
Given a \gls{genericsurface} $i:F \to M$:

1) A \gls{lifting} of $i$ is a \gls{genericsurface} $k$ in $M \times \R$ whose projection into the $M$ component is $i$. Such a \gls{lifting} has the form $k=(i,h)$ for some PL function $h:F \to \R$. We refer to $h$ as the ``height function" of the lifting $k$.

2) We say that two \gls{lifting}s $k_1=(i,h_1)$ and $k_2=(i,h_2)$ are equivalent if they uphold $\forall p,q \in F,i(p)=i(q): h_1(p)>h_1(q) \Leftrightarrow h_2(p)>h_2(q)$ - the relative height of every two points with the same $i$-image is the same. 

3) In order to draw a \gls{lifting} $(i,h)$, one draws the surface $i$ and, whenever two points $p,q \in F$ have the same $i$ value, indicate which of them is ``lower" (has a lower $h$ value) by ``deleting" the $i$-image of a small neighborhood of the lower point from the drawing. This clearly describes the lifting up to equivalence. This type of drawing of a knotted surface is called a ``\gls{BSD}" of the knotted surface. We believe this notation (\gls{BSD}) was first used by Satoh in \cite{Sat1}.
\end{defn}

Figure~\ref{fig:Figure 1} demonstrates how to draw the different parts of a \gls{genericsurface}, and how each part will look when some of it is deleted in order to draw a \gls{BSD}. On the left of Figure~\ref{fig:Figure 1}A, there are two sheets of the surface $F$ that are embedded in $M$, such that their images intersect transversely. Each sheet has a line on it, these lines are the preimages of the segment of double arc that is formed where the sheets intersect. Each double value on the arc segment has one preimage in each surface, and each of these preimages has a different $h$ value since $k=(i,h)$ is 1-1. Since $h$ is continuous, all the ``higher" preimages come from the same sheet.

\begin{defn} \label{SurfaceStrip}
From now on, we will informally refer to the ``higher" and ``lower" sheet at each such intersection. In a \gls{BSD}, a small neighborhood of the arc segment is deleted from the lower surface sheet. Globally, one can think of a double arc as a place where two long strips of surface intersect. One of these will be the ``lower" strip, and we will delete a small neighborhood of the arc from this strip, as in Figure~\ref{fig:Figure 2}.
\end{defn}

Figures~\ref{fig:Figure 1}C and \ref{fig:Figure 1}D show what the \gls{BSD} looks like at the end of an open double arc - at a DB or branch value. One should keep deleting a part of the ``lower strip" until the end of the arc is reached. This includes a neighborhood of one of the preimages of a DB value (the ``lower" one). A branch value has only one preimage, so the deleted part narrows as we approach a branch value and ends there.

Figure~\ref{fig:Figure 1}B depicts a triple value. It has three preimages, one on each of the intersecting surface sheets. We refer to the sheets as the ``highest", ``middle" and ``lowest" sheet based on the relative height of the preimage this surface contains. Each pair of sheets intersect at one of the three arc-segments that cross the triple value. Due to continuity, the lowest surface will be lower than both the middle and highest sheet along its intersection with each of them. One should thus remove a neighborhood of the union of these segments, which is a thickened ``X" shape, from the lowest sheet. One should also remove from the middle sheet a neighborhood of its intersection with the highest sheet.

\begin{figure}
\begin{center}
\includegraphics{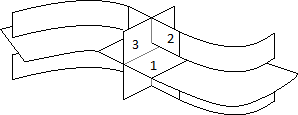}
\caption{Lifting a double arc}
\label{fig:Figure 2}
\end{center}
\end{figure}

\Gls{BSD}s are a higher dimensional analogue to knot diagrams. Unlike the lower-dimensional case, one needs to prove that every \gls{BSD} of a \gls{genericsurface} $i$ really does define a \gls{lifting} of it. For knot diagrams this is trivial - simply take the generic loop in $\R^2 \subseteq \R^3$ and, at every intersection, ``push up" the strand that the diagram tells us is supposed to be higher. The same general idea works for \gls{BSD}s, but the execution is slightly more complicated.

\begin{lem} \label{DTL}
Every \gls{BSD} on a \gls{genericsurface} $i$ defines a \gls{lifting} of $i$.
\end{lem}

\begin{proof}
We need to define a height function $h$ that corresponds to the diagram. At first we will prove that for every value $p \in i(F)$ it is possible to define $h$ locally in the preimage of a neighborhood $U$ of $p$: if $v$ is a regular or RB value, choose such a $U$ that is disjoint from $X(i)$, and define the local $h$ to be constant $0$ there. If $v$ is a double or DB value, $i^{-1}(U)$ will contain the two surface sheets, and the broken surfaces diagram indicates which of them is out to be higher and which is out to be lower. Set $h$ to return $1$ on the former and $0$ on the latter. Do the same for the highest, middle and lowest surface sheets of a triple value with the heights $2$, $1$ and $0$.

The case for branch values is only slightly more complicated. In this case, $i$ is a PL approximation of the function $(x,xy,y^2)$, in some parametrizations of $U$ and $i^{-1}(U)$. The double values in $U$ are $i(0,y)=i(0,-y)=(0,0,y^2)$ for all $y>0$. Set $h(x,y)$ to be equal $y$ or $-y$, making sure we pick the value that makes the right preimages higher as depicted in the \gls{BSD}. The function $(x,y)\mapsto (x,xy,y^2,\pm y)$ is a smooth 1-1 embedding, so its PL approximation will be a PL 1-1 embedding as needed.

We can now create a global $h$ using a common partition of unity trick. Take a PL partition of unity on $M$, $\tau_k:U_k \to \R$ where each $U_k$ is one of the aforementioned neighborhoods. Define the global $h$ as $h(p)=\sum h_k(p)\tau_k(i(p))$ ($h_k$ is the local $h$ on $U_k$). This $h$ corresponds to our \gls{BSD}, since if $i(p)=i(q)$ and the diagram tells us that $p$ is higher than $q$ then for $k$ for which $i(p) \in U_k$ upholds $h_k(p)>h_k(q)$ and this implies that $h(p)>h(q)$.
\end{proof}

\section{The obstruction to liftability} \label{OtL}

A \gls{genericsurface} may not be liftable. One can always ``attempt" to lift the surface by doing the following: Think of a double arc as a place where two long strips of surface intersect, as in Figure~\ref{fig:Figure 2}. In order to lift the surface, choose, at each arc, which strip will be ``lower" than the other. This is analogous to lifting a generic loop in 2-space to a knot by choosing crossing information at each intersection.

This ``\gls{liftingAt}" can be represented in a drawing of the surface. Doing so involves choosing one double value on each arc and ``deleting" a part of the lower strip around it, as one does in a \gls{BSD}. One must then progress along the arc, in both directions, and remove more parts of the lower strip until all the arc has been covered. If one succeeds in doing this for every arc, then the \gls{liftingAt} is successful - it describes a \gls{BSD} and thus a \gls{lifting}. In this section, we will review two ``obstructions" that can cause a \gls{liftingAt} to fail, and prove that these are the only obstructions to the liftability of the surface.

Firstly, notice that a neighborhood of a small segment of double arc looks like a bundle over that segment, the fibres of which are ``X"'s. Each of the two intersecting strips is a sub-bundle whose fibres are one of the two intersecting lines that compose the ``X". In Figure~\ref{fig:Figure 3}A the two strips are colored green and orange. Figure~\ref{fig:Figure 3}B depicts a single fibre of that ``X" bundle. The neighborhood of a long segment of double arc may be an immersed (but not embedded) image of an X-bundle over an interval, since it can intersect itself around triple values.

The neighborhood of a closed double arc will be an immersed image of an X-bundle over $S^1$. There is more than one kind of X-bundle over $S^1$. Each bundle of this kind is created by taking an X-bundle over an interval (which is trivial since intervals are contractable), and gluing the fibres at both ends of the interval together. Up to isotopy, there are 8 ways to do this ``gluing" - 4 rotations and 4 reflections.

\begin{figure}
\begin{center}
\includegraphics{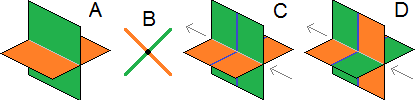}
\caption{X bundles, X fibres, and closing an X-bundle over $S^1$}
\label{fig:Figure 3}
\end{center}
\end{figure}

In 4 of these gluings, as in Figure~\ref{fig:Figure 3}C, the ends of each strip will be glued together, producing two ``closed strips" - immersed annuli / M\"{o}bius bands in $M$. In the other 4 cases, as Figure~\ref{fig:Figure 3}D depicts, each of the ends of one strip will be glued to an end of the other strip, combining them into one big closed strip.

\begin{defn} \label{Triv}
We refer to a closed arc whose neighborhood is composed of two separate closed strips as trivial, and a closed arc whose neighborhood is composed of one big closed strip as non-trivial.
\end{defn}

\begin{lem} \label{TrivLem}
A \gls{genericsurface} that has a non-trivial closed double arc is not liftable.
\end{lem}

\begin{proof}
As seen in Figure~\ref{fig:Figure 4}B, any attempt to lift a non-trivial arc will inevitably fail - what started as the lower strip will end up as the higher strip after going around the arc. Since $h$ is continuous, this implies that somewhere along the way both strips have the same $h$ value, contradicting the fact that $(i,h)$ is 1-1. 
\end{proof}

\begin{figure}
\begin{center}
\includegraphics{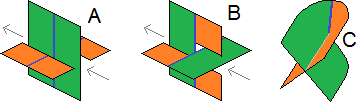}
\caption{A and B - Attempting to lift a trivial and a non trivial closed arc.
C - The ends of the two surface strips at a branch value.}
\label{fig:Figure 4}
\end{center}
\end{figure}

There is no similar obstruction for an open arc or a trivial closed arc. It is possible to progress throughout the whole arc and delete the lower strip until reaching the ends of the arc (Figure~\ref{fig:Figure 1}C or D) or returning to the starting point (see Figure~\ref{fig:Figure 4}A). In other words, an arc that is trivial and either open or closed is the intersection of two ``global" surface strips. In this case, a ``\gls{liftingAt}" becomes a simple choice of ``which of the two strips intersecting at the arc is higher".

\begin{rem}
Note that when an open arc ends in a branch value, the strips meet after the value as per Figure~\ref{fig:Figure 4}C, but this does not pose a problem. One can still lift one of the strips above the other throughout all of the arc up until the branch value(s) at its end(s) - as per Figure~\ref{fig:Figure 1}D.
\end{rem}

\begin{defn} \label{Attempt}
Let $i:F \to M$ be a \gls{genericsurface} with no non-trivial closed double arcs. A ``\gls{liftingAt}" of $i$ is a choice, for each double arc $DA$, of which of the two surface strips that intersect at $DA$ is higher.
\end{defn}

There is a second obstruction that might make an individual \gls{liftingAt} fail. One can draw any \gls{liftingAt} using the above method - draw the surface, and delete a small ``sub-strip" from the lower surface strip at each arc. Lemma~\ref{DTL} says that if this drawing confers to the definition of a \gls{BSD} the \gls{liftingAt} is successful - it describes a genuine lifting of the surface. However, the drawing may fail to be a \gls{BSD}. 

At each triple value three ``surface sheets" intersect. The intersection of any two of them is a double arc segment that goes through the triple value. Each of these segments is a small part of a double arc. Choosing how to lift this double arc determines which of the said two sheets is lower. For example, in Figure~\ref{fig:Figure 2}, the lifting of the double arc implies that the sheet marked ``2" is lower than the sheet marked ``1". There are two other arcs that go through the triple value. Choosing a \gls{lifting} for them would tell us if the sheet marked ``3" is higher or lower than ``1" and/or ``2", but Figure~\ref{fig:Figure 2} does not depict this information.

A full \gls{liftingAt} will dictate, for each pair of intersection sheets at each triple value, which sheet is lower. Since each triple value has 3 pairs of sheets and 2 ways to lift each pair, a \gls{liftingAt} can have one of $2^3=8$ forms around each triple value. We draw 4 of them in Figure~\ref{fig:Figure 5} and the other 4 are the mirror images of those. If at every triple value the \gls{liftingAt} has one of the forms \ref{fig:Figure 5}A-C or their mirror images, then it will fall in line with the definition of a \gls{BSD}, as per Figure~\ref{fig:Figure 1}B, and so the \gls{liftingAt} is successful.

\begin{figure}
\begin{center}
\includegraphics{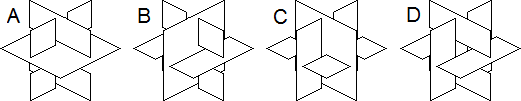}
\caption{Successful and failed ways to lift a triple value}
\label{fig:Figure 5}
\end{center}
\end{figure}

On the other hand, if a triple value $p$ has the form of Figure~\ref{fig:Figure 5}D, then the preimages of $p$ have a ``cyclic height relation". Denoting the preimage in the $k$th surface sheet ($k=1,2,3$) as $p_k$, one can see that $h(p_1)<h(p_2)$, $h(p_2)<h(p_3)$ and $h(p_3)<h(p_1)$ - a contradiction. The mirror image of Figure~\ref{fig:Figure 5}D depicts the reverse cyclic relation, where $h(p_1)>h(p_2)$, $h(p_2)>h(p_3)$ and $h(p_3)>h(p_1)$. This implies that a \gls{liftingAt} that dictates one of these two configurations on any triple value must fail. To summarize:

\begin{thm} \label{LiftThm}
1) A \gls{liftingAt} of a \gls{genericsurface} with no non-trivial closed double arcs is successful iff it does not produce a cyclic height relation at any triple value.

2) A \gls{genericsurface} is liftable iff it has no non-trivial closed double arcs and at least one of its \gls{liftingAt}s does not produce a cyclic height relation at any triple value.
\end{thm}

\begin{rem}
Note also that the ``successful \gls{liftingAt}s" are in 1-1 correspondence with the equivalence classes of the \gls{lifting}s of the surface.
\end{rem}

We end this chapter with a very important note:

\begin{rem} \label{SurfaceSet}
In order to check if a \gls{genericsurface} $i:F \to M$ is liftable, one needs to check if any closed arc is non-trivial, and then if any \gls{liftingAt} is successful. In order to do this, one only needs to examine the image $S \subset M$ of the surface. 

It is also simple to see if a subset $S$ of $M$ is the image of a \gls{genericsurface} - this happens iff every point in $S$ has a neighborhood like one of those in Figure~\ref{fig:Figure 1}.

We will thus abuse the term ``\gls{genericsurface}" for the remainder of the thesis: instead of a function $i:F \to M$ as in Definition~\ref{Generic}, we will use it to mean a set that is the image of such a function. We do this because a subset of $M$ is simpler to define and examine algorithmically than a function.
\end{rem}

\chapter{An Algorithm to Lift a Generic Surface} \label{SecNP}

In this chapter we will explain the technique we use to determine if a \gls{genericsurface} is liftable or not. We will also describe a lifting algorithm - an algorithm that receives a \gls{genericsurface} and determines if it is liftable or not.

The algorithm is composed of three parts. First is the preliminaries, in which the algorithm verifies that the input is valid (a real \gls{genericsurface}) and compiles some information from the surface. The algorithm also checks if the surface has any non-trivial closed double arcs, in which case the surface is not liftable.

If the surface has no non-trivial arcs, then the algorithm proceeds to the second part. In it, the algorithm compiles a \gls{symmetric}, called the ``lifting formula of the surface", such that the surface is liftable iff the formula is satisfiable. The last step involves using any 3-sat solving algorithm to determine if the lifting formula is satisfiable, and thus whether the surface is liftable.

\section{The preliminaries of the algorithm} \label{PremSec}

In the first section, we will list the preliminary steps of the lifting algorithm and the run-time required by each step. The general idea of each step is simple to understand, but the actual realization and run-time computation is often long and technical. In order to preserve the flow of the thesis, we will forgo these technical parts here. We will provide them in the next chapter (named ``technicalities"), which is dedicated specifically to them.

The preliminaries of the algorithm are as follows:

1) The input of the algorithm is a data type that represents a \gls{genericsurface} in a 3-manifold. It is a pair $(M,S)$ where $M$ is an abstract simplicial complex whose geometric realization is a 3-manifold, and $S$ is a subcomplex of $M$, whose geometric realization is a \gls{genericsurface} in $M$. The size of the data is indicated by a parameter called $n$ (the number of 3-simplices in $M$).

The first step of the algorithm is to verify the validity of the input - that $M$ really is a 3-manifold and that $S$ is a \gls{genericsurface}. We will rigorously define this data type, explain how the algorithm verifies that the input is valid, and prove that it can be done in linearithmic ($O(n \cdot \log(n))$) time, in the subsections \ref{DTsec1}, \ref{DTsec2}, and \ref{DTsec3} (see Theorem~\ref{ValidQuad}).

2) The algorithm then identifies the relevant parts of the surface - it indicates what are the double arcs, the triple values, the 3 intersecting surface sheets at each triple value and the 3 intersecting arc-segments at each triple value. It saves them as accessible data. While it identifies these parts, the algorithm also names (or indexes) them. It indexes the double arcs as $DA_0,...,DA_{N-1}$, the triple values as $TV_0,...,TV_{K-1}$, and the 3 intersecting arc-segments at each triple value $TV_k$ as $TV_k^1$, $TV_k^2$ and $TV_k^3$.

Furthermore, each arc segment $TV_k^l$ is a small part of some double arc $DA_j$, and the algorithm will define / calculate an index function $j(k,l)$ that gives us the index $j$ of this double arc (for every $k=0,...,K-1,l=1,2,3$). All of this will take quadratic $O(n^2)$ time, as we will explain in section~\ref{idensec} (see Theorem~\ref{IdenQuad}).

3) The algorithm will also identify the two intersecting surface strips at each double arc in $O(n^2)$ time. It is possible that some of the closed arcs will be non-trivial - their two surface strips will merge into one strip. The algorithm will check if this occurs. If there is such a non-trivial arc, then the surface is not liftable and the algorithm will end. Otherwise, it will name the surface strips of each double arc $DA_j$ to distinguish between them. One of them will be called the ``0 strip" at $DA_j$ and the other will be called the ``1 strip" at $DA_j$.

\begin{rem}
The purpose of distinguishing between the 0 and 1 strips at any double arc is as follows: when working with knots in 3 space, one sometimes uses oriented knots / links. In the diagram of oriented knots / links, one can distinguish between $+$ crossings and $-$ crossings. Furthermore, given an oriented generic loop, the choice of which of its intersections will be $+$ crossings and which will be $-$ crossings determines how to lift said loop.

While there was always a choice between two kinds of \gls{lifting}s at each crossing, without an orientation there is no way to distinguish between the two without drawing the loop. In a way, the difference between them involves the global topology of the loop as a subspace of $\R^2$. But once there is an orientation, one can distinguish between a $+$ crossing and a $-$ crossing at a given intersection point, by looking at a neighborhood of this intersection. Furthermore, if one indexes all the intersection points as $IP_0,...,IP_{N-1}$, then one can describe each lifting of the diagram via an $N$-tuple of $+$'s and $-$'s - the $k$'s intersection point have a $+$ / $-$ crossing iff the $k$'s entry in the vector is $+$ / $-$.

In a \gls{genericsurface}, there may be no way to orient the surface (it may be non-orientable), but one can choose arbitrarily which of the intersecting strips in each double arc is the $0$ strip and which is the $1$ strip. Afterwards, one can describe a \gls{liftingAt} by choosing, for each arc, if the $0$ strip is higher than the $1$ strip or if it is the other way around.
\end{rem}

\begin{defn} \label{01Strip}
Given a \gls{genericsurface} in which all closed double arcs (if there are any) are trivial, name one of the intersecting strips at each double arc ``the $0$ strip" and the other strip ``the $1$ strip". If a \gls{liftingAt} makes the $0$ strip (resp. $1$ strip) at a certain double arc higher than the $1$ strip (resp. $0$ strip), we will say that this \gls{liftingAt} is a ``$0$ lifting" (resp. ``$1$ lifting") at this arc.

Using this, we encode each \gls{liftingAt} of the surface as a vector \\$(x_0,...,x_{N-1}) \in \{0,1\}^N$ ($N$ is the number of the surface's double arcs) where the \gls{liftingAt} is a $0$ / $1$ lifting at the double arc $DA_j$ iff the $j$th entry at the vector is $0$ / $1$.
\end{defn}

\section{The lifting formula of a surface} \label{LiftFormSec}

Having verified that the surface has no non-trivial double arc, our next step, as per Theorem~\ref{LiftThm}, is to see whether either of the potential \gls{liftingAt}s of the surface is legitimate. In this section, we will define a 3-sat formula in $N$ variables, called the lifting formula of the surface, such that a \gls{liftingAt} is legitimate iff its corresponding vector $(x_0,...,x_{N-1}) \in \{0,1\}^N$ satisfies the formula. In particular, the surface will be liftable iff the formula is satisfiable.

By definition, a \gls{liftingAt} is legitimate iff it does not produce a cyclic height relation at any triple value. In order for the formula to capture this information, the parameters of the formula should represent information about the neighborhood of the triple value. Specifically:

\begin{defn} \label{3Discs}
At each triple value $TV_k$, every two of the three surface sheets at $TV_k$ intersect at one of the arc segments. In particular, each of the sheets contains a unique two of the arc segments. We therefore denote the sheets $D_k^{\{1,2\}}$, $D_k^{\{1,3\}}$ and $D_k^{\{2,3\}}$ - where each sheet is named after the two arc segments it contains.
\end{defn}

Up to homeomorphism, the neighborhood of $TV_k$ looks like Figure~\ref{fig:Figure 6}. A \gls{liftingAt} will produce a cyclic height relation in $TV_k$ iff one of the following two situations occur:

I) The sheet $D_k^{\{1,3\}}$ is higher than the sheet $D_k^{\{1,2\}}$ along their intersection at $TV_k^1$, the sheet $D_k^{\{1,2\}}$ is higher than the sheet $D_k^{\{2,3\}}$ along their intersection at $TV_k^2$, and the sheet $D_k^{\{2,3\}}$ is higher than the sheet $D_k^{\{1,3\}}$ along their intersection at $TV_k^3$. 

II) The exact opposite situation. The sheet $D_k^{\{1,2\}}$ is higher than the sheet $D_k^{\{1,3\}}$ along their intersection at $TV_k^1$, the sheet $D_k^{\{2,3\}}$ is higher than the sheet $D_k^{\{1,2\}}$ along their intersection at $TV_k^2$ and the sheet $D_k^{\{1,3\}}$ is higher than the sheet $D_k^{\{2,3\}}$ along their intersection at $TV_k^3$.

\begin{figure}
\begin{center}
\includegraphics{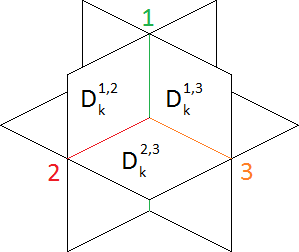}
\caption{The neighborhood of a triple value with the names of the arc segments and the sheets}
\label{fig:Figure 6}
\end{center}
\end{figure}

As the arc $DA_{j(k,l)}$ crosses the triple value $TV_k$ via the arc segment $TV_k^l$, each of the two surface strips that intersect at the arc coincide with one of the two sheets that intersect at $TV_k^l$. The algorithm will find which strip coincides with which sheet. It will do so while it identifies the strips, so no additional computation is required. 

\begin{defn} \label{sParam}
We encode this information using binary parameters $s(k,1),$ $s(k,2),s(k,3) \in \{0,1\}$. Their values are set as follows:

a) $s(k,1)=0$ if the sheet $D_k^{\{1,3\}}$ coincides with the $0$ strip along $TV_k^1$, and the sheet $D_k^{\{1,2\}}$ coincides with the $1$ strip along $TV_k^1$. $s(k,1)=1$ if it is the other way around.

b) $s(k,2)=0$ if the sheet $D_k^{\{1,2\}}$ coincides with the $0$ strip along $TV_k^2$, and the sheet $D_k^{\{2,3\}}$ coincides with the $1$ strip along $TV_k^2$. $s(k,2)=1$ if it is the other way around.

c) $s(k,3)=0$ if the sheet $D_k^{\{2,3\}}$ coincides with the $0$ strip along $TV_k^3$, and the sheet $D_k^{\{1,3\}}$ coincides with the $1$ strip along $TV_k^3$. $s(k,3)=1$ if it is the other way around.
\end{defn}

\begin{rem}
The algorithm will calculate these parameters while it identifies the 0 and 1 strips of the surface, in section \ref{idensec}. It will not add to the run-time of the algorithm.
\end{rem}

The definition insures that:

a) The sheet $D_k^{\{1,3\}}$ coincides with the $s(k,1)$ strip along $TV_k^1$, and the sheet $D_k^{\{1,2\}}$ coincides with the $\neg s(k,1)$ strip along $TV_k^1$.

b) The sheet $D_k^{\{1,2\}}$ coincides with the $s(k,2)$ strip along $TV_k^2$, and the sheet $D_k^{\{2,3\}}$ coincides with the $\neg s(k,2)$ strip along $TV_k^2$.

c) The sheet $D_k^{\{2,3\}}$ coincides with the $s(k,3)$ strip along $TV_k^3$, and the sheet $D_k^{\{1,3\}}$ coincides with the $\neg s(k,3)$ strip along $TV_k^3$.

All this implies that the surface will produce a cyclic height relation in $TV_k$ iff one of the following two situations occur:

I) The $s(k,1)$ strip is higher than the $\neg s(k,1)$ strip along their intersection at $TV_k^1$, the $s(k,2)$ strip is higher than the $\neg s(k,2)$ strip along their intersection at $TV_k^2$, and the $s(k,3)$ strip is higher than the $\neg s(k,3)$ strip along their intersection at $TV_k^3$.

II) The exact opposite situation. The $\neg s(k,1)$ strip is higher than the $s(k,1)$ strip along their intersection at $TV_k^1$, the $\neg s(k,2)$ strip is higher than the $s(k,2)$ strip along their intersection at $TV_k^2$, and the $\neg s(k,3)$ strip is higher than the $s(k,3)$ strip along their intersection at $TV_k^3$.

Let the vector $x=(x_0,....,x_{N-1})$ represent a \gls{liftingAt}. Recall that each $x_j$ is either $0$ or $1$ and $x_j=0$ iff the 0 strip is higher than the 1 strip at the double arc $DA_j$. The above is equivalent to saying that the \gls{liftingAt} will produce a cyclic height relation in $TV_k$ iff either $x_{j(k,1)} \equiv s(k,1)$, $x_{j(k,2)} \equiv s(k,2)$ and $x_{j(k,3)} \equiv s(k,3)$, or $x_{j(k,1)} \equiv \neg s(k,1)$, $x_{j(k,2)} \equiv \neg s(k,2)$ and $x_{j(k,3)} \equiv \neg s(k,3)$.

In other words, the \gls{liftingAt} will \textbf{not} produce a cyclic height relation in $TV_k$ iff the vector $x=(x_0,....,x_{N-1})$ solves the formula $(((x_{j(k,1)} \leftrightarrow s(k,1)) \vee (x_{j(k,2)} \leftrightarrow s(k,2)) \vee (x_{j(k,3)} \leftrightarrow s(k,3))) \wedge ((x_{j(k,1)} \leftrightarrow \neg s(k,1)) \vee (x_{j(k,2)} \leftrightarrow \neg s(k,2)) \vee (x_{j(k,3)} \leftrightarrow \neg s(k,3))))$.

In particular, the \gls{liftingAt} is successful iff it upholds this formula for all $k=0,...,K-1$. This means that: 

\begin{thm} \label{Thm1}
A \gls{liftingAt} is successful iff the vector $x=(x_0,....,x_{N-1})$ solves the following $3$-sat formula, and in particular the surface is liftable iff this formula is satisfiable. 
\end{thm} \begin{multline} \label{LiftForm}
\bigwedge_{k=0}^{K-1} (((x_{j(k,1)} \leftrightarrow s(k,1)) \vee (x_{j(k,2)} \leftrightarrow s(k,2)) \vee (x_{j(k,3)} \leftrightarrow s(k,3))) \wedge \\
((x_{j(k,1)} \leftrightarrow \neg s(k,1)) \vee (x_{j(k,2)} \leftrightarrow \neg s(k,2)) \vee (x_{j(k,3)} \leftrightarrow \neg s(k,3)))).
\end{multline}

\begin{defn} \label{LiftFormDef}
Formula (\ref{LiftForm}) is called the lifting formula of the surface.
\end{defn}

\begin{rem} 
1) This 3-sat formula has $N$ variables (the number of double arcs the surface has), and $2K$ clauses (twice the number of triple values the surface has). After calculating the $j(k,l)$s and $s(k,l)$s, writing the full formula takes $O(K) \leq O(n)$ time.

2) The lifting formula is clearly symmetric.

3) The index function $j(k,l)$ and parameters $s(k,l)$ do not exactly match the Definitions~\ref{3Sat}(4,5) of the index function and parameters of a 3-sat formula. This is because each $j(k,l)$ and $s(k,l)$ is used in two clauses of the formula, instead of just one.

In the notation of Definition~\ref{3Sat}(4,5), $j(k,l)$ is the index function of the $(2k,l)$th and $(2k+1,l)$th literals, $s(k,l)$ is the parameter of the $(2k,l)$th literal and $\neg s(k,l)$ is the parameter of the $(2k+1,l)$th literal.
%
\end{rem}

\section{The complexity of the lifting problem} \label{CompOfLift}

The final step of the algorithm is to use any 3-sat algorithm to solve the lifting formula. As per Remark~\ref{reduce}, this takes $O(c^N+K \cdot \log(K))$ time, with $c$ depending on the algorithm one uses. Add to this the run-time of the preliminaries, $O(n^2)$, and the total runtime of the algorithm is thus $O(c^N+K \cdot \log(K)+n^2)$.

$N$ and $K$ - the numbers of double arcs and triple values of the surface, respectively, are clearly both $O(n)$ - there are some $a,b,c,d \in \R$ such that $N \leq an+b$ and $K \leq cn+d$. Using this, the $K$-dependent element of the complexity is observed in the $n$-dependent part and the run-time becomes $O(c^N+n^2)$. One may use the bound on $N$ to represent the complexity entirely in terms of $n$, as $O((c^a)^n)$. We will not calculate a formal bound to $a$, but one can intuitively expect it to be small. $M$ has at most $6n$ 1-simplices, but only a few of these are likely to be in the intersection set $X(i)$, and these will compose an even smaller number of double arcs (each arc is usually made of many 1-simplices).

That being said, we prefer to continue representing the complexity of the algorithm by both $n$ and $N$, as $O(c^N+n^2)$. Our claim is that the parameters $n$ and $N$ capture fundamentally different aspects of the topology of the surface, and they should both be represented. $n$ measures the difficulty of encoding the surface as a data type, of describing it to a computer. There are several ways one may encode a \gls{genericsurface} as a data type. Each way will have its own ``$n$" - a parameter that measures the size of the data type. Any lifting algorithm will have a ``preliminaries" part, whose complexity will depend on $n$. In it, the algorithm will calculate the relevant data required to determine if the surface is liftable. If the surface is represented efficiently, the ``preliminaries" part should take polynomial $O(n^\alpha)$ time. 

After the preliminaries, the algorithm will use said ``relevant data" to determine if the surface is liftable. The previous works on the lifting formula suggest that the complexity of this part depends on the parameter $N$.

In their articles, Carter and Saito (\cite{Car&Sai1}) and Satoh (\cite{Sat1}) each equated a \gls{liftingAt} with a kind of combinatorial structure on the intersection graph of the surface, and found that a \gls{liftingAt} is successful iff the matching structure upholds some condition. This is similar to the way we equated a \gls{liftingAt} with a vector in $\{0,1\}^N$, and showed that the attempt is successful iff the vector satisfies that lifting formula.

In each case, a close inspection reveals that there are $2^N$ possible structures. This complies with the fact that a surface has $2^N$ potential \gls{liftingAt}s. The ``trivial" way to check if a surface is liftable is to check all structures, and see if any of them are successful. This takes $O(2^N)$ time. It may be possible to devise a more efficient check, as we did using 3-sat algorithms, but there is no known way to check this in less than exponential $O(c^N)$ time. 
%

Theorem~\ref{Thm1} can also be used to prove the following:

\begin{thm} \label{Thm2}
The lifting problem is NP.
\end{thm}

Recall that a problem is NP if there is a polynomial-time algorithm that receives an input and a certificate and returns ``yes" if the certificate solves the problem for the given input and ``no" otherwise. In our case, the input is a \gls{genericsurface} $(M,S)$ and the certificate is a form of data that represents a \gls{liftingAt} of the surface. The algorithm should return ``yes" if this \gls{liftingAt} is legitimate, and ``no" otherwise.

All parts of the lifting algorithm, except solving the lifting formula, take polynomial time. This includes checking if the surface has any non-trivial closed arcs, identifying the relevant parts of the surface and calculating the lifting formula. This can be used as a foundation for the ``certificate verifying algorithm". Next, if the surface has no non-trivial arcs, the algorithm should check if the \gls{liftingAt} described by the certificate is legitimate. Intuitively, this should take polynomial time. The only problem is deciding how to encode this \gls{liftingAt}, what would be a valid certificate for the lifting problem.

The simplest way to depict a \gls{liftingAt} it to provide the vector $(x_0,...,$ $x_{N-1})$ that corresponds to it. To check if this \gls{liftingAt} is legitimate, one needs only to insert these values into the lifting formula and see if it returns $1$ or $0$. This clearly takes linear time. If the reader is willing to except this as a certificate, then we now have a complete polynomial-time certificate verifying algorithm, and Theorem~\ref{Thm2} follows.

However, one could argue that this is not a valid certificate. For one, the exact formulation of the lifting formula depends on arbitrary choices made by the algorithm - the way it chooses to order the double arcs determines which variable $x_j$ corresponds to which arc, and the way it chooses which of the surface strips at the $j$th arc is the 0 strip and which is the 1 strip effects the variable $s(k,l)$ of all literals for which $j(k,l)=j$. These choices also determine which vectors in $\{0,1\}^N$ correspond to which \gls{liftingAt}. Before this choice is made, a vector does not correspond to a \gls{liftingAt}.

There is a better way to describe a \gls{liftingAt}. Recall that one of the things that the algorithm identifies in the ``preliminaries" stage is the pair of intersecting surface strips at each double arc, and names one of them the 0 strip and the other the 1 strip. The exact data that the algorithm uses to represent a surface strip is slightly complicated. We use what we call ``a continuous designation" - see Definition~\ref{Des} in section~\ref{idensec}.

A good certificate for the lifting problem will be a similar data type, but instead of identifying the 0 and 1 strips along each double arc, it will identify the higher and lower strips along the arc.

The algorithm needs to verify that this information is valid - that it really describes the two surface strips at each arc. Then, after the algorithm chooses which surface strip at any arc $DA_j$ is the 0 strip and which is the 1 strip, it will compare this information with the certificate and see which strip is higher at each arc. It will then determine the value of the variable $x_j$ per Definition~\ref{01Strip} - if the 0 strip is the higher strip then $x_j=0$, otherwise $x_j=1$. Lastly, it will insert the values of the $x_j$s into the formula and check if they satisfy it.

Intuitively, this will take polynomial time $O(p(n))$ - verifying that the certificate contains a real description of the surface strips of all arcs is a simpler task than identifying the surface strips yourself, and we know that the latter takes polynomial time. Nonetheless, we will give a formal proof that this process takes $O(n \cdot \log(n))$ in the last section of the technicalities chapter - section~\ref{NPsec}.

After the technicalities, we will set about proving that the lifting problem is NP-hard. We will do this by reducing the symmetric 3-sat problem, proved to be NP-complete in section~\ref{Sym3Sec}, to the lifting problem in polynomial time. In order to do this, we need to reverse the process we used in this chapter. For every \gls{symmetric}, we need to construct a \gls{genericsurface} $(M,S)$ such that the lifting formula of the surface is equivalent to the given formula.

%
\chapter{Technicalities} \label{SecTech}

As stated throughout the previous chapter, this chapter is dedicated to the technicalities of the lifting algorithm. In it, we explain the following: how to depict a ``sub-simplicial complex of a triangulated 3-manifold" to a computer as a data type, how the algorithm verifies that the input is valid, how it identifies the relevant parts of the surface - the triple values, double arcs, intersecting surface strips at each arc, etc, how it checks if the surface has any non-trivial closed double arcs, and how it deduces the index function $j(k,l)$ and parameters $s(k,l)$ of the lifting formula. We will also examine the complexity of each of these tasks, and in particular prove that they all take polynomial time.

Additionally, in the last section, we will explain how to encode a \gls{liftingAt} of the surface as a certificate, and how to check if this \gls{liftingAt} is legitimate in polynomial time.

\section{A \gls{genericsurface} as a data type - the basics} \label{DTsec1}

%
In this first section, we will explain how to encode a simplicial complex as a data type a computer can use, and in particular how to encode a \gls{genericsurface} within a 3-manifold to a computer. We also explain the first steps needed to verify that an input of this sort is valid.

\begin{defn} \label{GSIM}
1) We represent a finite 3-dimensional simplicial complex via a data type that contains the following entries: a number $\#V$ which indicates how many vertices the complex has; a list $M_1$ of pairs of numbers representing edges - a pair $(s,q)$ means that there is an edge between the $s$'th vertex and the $q$'th vertex (we index the vertices between $0$ and $\#V-1$); a similar list $M_2$ of triples representing triangles and a list $M_3$ of quadruples representing tetrahedron.

Each $d$-simplex is supposed to be a set of $d+1$ elements. We use $d+1$-tuples instead of sets since it is easier for a computer to define and work with them. In particular, the $d+1$ elements of each tuple must all be distinct and, since the order of the elements is irrelevant, we assume that the numbers in each tuple appear in ascending order. For instance, the triangle with vertices $1$, $4$ and $6$ will be written as $(1,4,6)$ and not $(4,1,6)$.

$M_d$ is also supposed to represent the set of all $d$-simplexes, and we use a list instead of a set for similar reasons. Due to this, no $d$-simplex should appear in $M_d$ twice. Lastly, the faces of every simplex in a complex must also belong to the aforementioned complex. Due to this, every ``sublist" of every $d$-simplex must belong to the appropriate $M_d$. For instance, if the 3-simplex $(0,3,4,9)$ is in $M_3$, then the 2-simplices $(0,4,9)$ and $(0,3,4)$ must be in $M_2$.

2) A subcomplex $S$ contains the following data: a list $S_0$ of all the vertices in this sub-complex (a list of numbers between $0$ and $\#V-1$), and lists $S_1$, $S_2$ and $S_3$ of all the edges, triangles and tetrahedrons in the sub-complex. 

3) An abstract \gls{genericsurface} is a pair $(M,S)$ where $M$ is a 3-dimensional simplicial complex whose geometric realization is a compact 3-manifold (possibly with a boundary), and $S$ is a 2 dimensional sub-complex of $M$ whose geometric realization is a \gls{genericsurface} in the above-mentioned 3-manifold (in the meaning of Remark~\ref{SurfaceSet}).
\end{defn}

We would like to explain some conventions we use, regarding \gls{genericsurface}s and complexity.

\begin{rem} \label{Calculate}
1) The parameter we use to describe the size of a \gls{genericsurface} $(M,S)$ is the number $n$ of 3-simplices in $M$. A 3-manifold is a pure simplicial complex - every $0$, $1$ or $2$ simplex is in the boundary of some $3$-simplex. A $3$-simplex has a 4 vertices, 6 edges and 4 faces, and so the manifold can have no more then $4n$ 0- or 2-simplices and $6n$ 1-simplices. $n$ thus linearly bounds the length of all the lists $M_i$ and the vectors $S_i$, and is therefore a good representation of the size of $(M,S)$.

2) The way one calculates the runtime of an algorithm depends on the kind of actions one considers to be trivial - to take $O(1)$ time. For instance, it is common to ``cheat" and consider the addition of two integers $k_1$ and $k_2$ to take $O(1)$ time, but the amount of time it actually takes depends on the number of digits in each $k_i$, and is thus proportional to the logarithm of these numbers - $O(\log(max\{k_1,k_2\}))$. This is why, for instance, each of the different algorithms for the multiplication of $n \times n$ matrices is commonly considered to take $O(n^a)$ time for some $2 < a \leq 3$, disregarding the logarithmic component that depends on the number of digits of the entries in the matrix.
 
This ``cheating" reflects an assumption that there is a common bound on the number of digits of all the numeric values that appear in the input. This really is the case when working with a computer, where every type of variable that represents a number (Int, Double, Float, Long, etc) can only contain numbers of a given size. However, an abstract algorithm (or alternatively a Turing machine) has no such limitation.

We will often make the same assumption. For instance, we consider the size of a \gls{genericsurface} to be $O(n)$ since the length of every list in it (the $M_i$'s) is $(O(n))$, even though technically it is $O(n \cdot \log(n))$, since every value the entries of the list has $O(\log(n))$ digits. Similarly, we will sometimes consider the time some action takes to be $O(1)$ instead of $O(\log(n))$ or $O(\log^2(n))$. This assumption can only change the runtime of the algorithm by removing a logarithmic multiplier - for instance, an algorithm where at most $O(n^2)$ such actions are made will take $O(n^2)$ time instead of $O(n^2 \cdot \log(n))$. In any case, this change between these two perspectives is too small to let the same algorithm have a polynomial runtime from one perspective but not from the other. 

3) Our definition of a simplicial complex does not allow for two simplices with the exact same vertices. Even had our definition allowed same-vertex simplices, a single barycentric subdivision on any complex would still have created an equivalent complex with no same-vertex simplices. The total number of 3-simplices in the complex will be multiplied by a constant (24), and so polynomial time algorithms will have the same runtime on the subdivided complex as they did on the original one. Due to this, we may assume WLOG that all complexes have no same-vertex simplices.
\end{rem}

Verifying that $(M,S)$ is indeed a \gls{genericsurface} is a long and technical process. In the following few subsections, and the remainder of the current one, we will explain all the steps of this process, and prove that, collectively, they all take linearithmic $O(n \cdot \log(n))$ time. The first conditions one should check is that $M$ is a valid 3 dimensional simplicial complex (as per Definition~\ref{GSIM}), that $S$ is a 2 dimensional sub-complex, and that they are both pure.

\begin{lem} \label{PureComp}
All these checks take ($O(n \cdot \log(n))$) time. 
\end{lem}

\begin{proof}
Finding $n$ (the number of 4-tuples, or 3-simplices), and making sure there are no more than $4n$ triangles and $6n$ edges, and that $\#V \leq 4n$ takes $O(1)$ time. If this does not hold $M$ cannot be pure. One may now assume that the lengths of the lists of $d$-dimensional simplices in $M$ are linearly bounded by $n$ ($O(n)$).

Next, checking that the number of $j$-simplices of $S$ is smaller or equal to that of $M$, and that it has no 3-simplices, also takes $O(1)$ time. One may now assume that the lengths of the lists of $d$-dimensional simplices in $S$ are $O(n)$ as well.

Definition~\ref{GSIM} demands that one verifies that every tuple contains integers from the correct domain ($0,...,\#V-1$) in increasing order. Verifying this clearly takes $O(n)$ time.

One may then sort the lists of $d$-simplices of $M$ (for $d=1,2,3$) and of $S$ (for $d=0,1,2$) according to the lexicographical order, so that, for instance, the 2-simplex $(1,3,7)$ will appear before $(2,3,5)$. This will simplify some of the following steps of the algorithm. This takes $O(n \cdot \log(n))$ time using the merge sort algorithm.

Next, one must write the list of all the 2-faces of all the 3-simplices in $M$, with multiplicities. For every 3-simplex $(a_0,...,a_3)$ in $M$, the list will have 4 entries - the faces $(a_0,a_1,a_2)$, $(a_0,a_1,a_3)$ etc. It will take $O(n \cdot \log(n))$ time to sort this list, and $O(n)$ time to delete all the repeating instances of 2-simplices. After the sorting and deleting, this list should be equal to the list of 2-simplices of $M$ in order for $M$ to be a valid and pure simplicial complex. Since both of these lists are sorted and of length $O(n)$, comparing them takes $O(n)$ time.

Checking that $M$ is a valid and pure complex requires that one also compares the list of 1-faces produced from the 2-simplices of $M$ with the list of 1-simplices of $M$, and the list of 0-faces produced from the 1-simplices of $M$ with the list $0,1,...,V-1$. This will again take $O(n \cdot \log(n))$ time, as will checking that $S$ is a valid and pure complex. Lastly, one must check that the list of $d$-simplices of $S$ ($d=1,2$) is contained in the list of $d$-simplices of $M$. Since all these lists are sorts, this takes $O(n)$ time.
\end{proof}

\begin{rem} \label{SortSimp}
As part of the previous proof, we ordered the lists of $d$-simplices of $M$ and $S$ in lexicographical order. From now on, one may assume that these lists are ordered. This will simplify some of the following calculations.
\end{rem}

The remaining steps in proving that the input is valid are verifying that $M$ is a 3-manifold and that $S$ is a \gls{genericsurface} within $M$. The former involves examining the stars and/or links of the simplices of $M$. Recall that the star $St(\sigma)$ of a simplex $\sigma$ in a simplicial complex $M$ is the sub-complex of $M$ that contains the simplices that contain $\sigma$ and their boundaries, and that the link $Lk(\sigma)$ is the sub-complex of $St(\sigma)$ that contains only those simplices that are disjoint from $\sigma$. Similarly, verifying that $S$ is a \gls{genericsurface} requires the examination of the $S$-stars or $S$-links of every simplex $\sigma$ in $S$.

\begin{defn} \label{St'}
The ``$S$-star" of $\sigma$, $St'(\sigma)$ is the sub-complex containing all the simplices \textbf{within $S$} that contain $\sigma$, and their boundaries. The ``$S$-link" $Lk'(\sigma)$ is the sub-complex of simplices from $St'(\sigma)$ that are disjoint from $\sigma$.
\end{defn}

For every $d$-simplex $\sigma$ of $M$ ($d=0,1,2$), we will identify and store the subcomplex $St(\sigma)$ in such a way that it can be readily accessed in $O(1)$. Using a computer program, this can be achieved using pointers. This complex is made of 4 lists listing $d$-simplices of $St(\sigma)$. Searching these lists will still take linear time. We will similarly define the sub-complexes $Lk(\sigma)$, $St'(\sigma)$ and $Lk'(\sigma)$ for the appropriate $\sigma$'s.

\begin{lem} \label{St'Lem}
This can be done in linear $O(n)$ time.
\end{lem}

\begin{proof}
We begin by defining each $St(\sigma)$ and $Lk(\sigma)$ as an ``empty" complex - it will contain 4 empty lists, one for the simplices of each dimension. We will later add the appropriate simplices to each list.

A $d$-simplex is a set of $d+1$ integers represented as an increasing sequence $(a_0,...,a_d)$. Calculating the intersection, union or difference of two such sets takes $O(1)$ time (since $d$ is bounded by $3$). For instance, the union of $(5,9)$ and $(1,5,7)$ is $(1,5,7,9)$, their intersection is $5$ and the difference is $9$.

A simplex $\nu$ is in the star of another simplex $\sigma \neq \nu$ iff their union is also a simplex in $M$. For every simplex $\tau$ that contains $\sigma$ and every simplex $\mu$ that is contained in $\sigma$ (including the ``empty simplex") the simplex $\nu=(\tau \setminus \sigma) \cup \mu$ is in $St(\sigma)$. $\mu$ is the intersection of $\nu$ and $\sigma$ while $\tau$ is their union. Two different simplices $\nu_1,\nu_2$ in $St(\sigma)$ cannot have the same union and intersection with $\sigma$, and we can use $\tau$ and $\mu$ as a way of listing all of the simplices in $St(\sigma)$ with no repetitions. Notice that such a $\nu$ is in $Lk(\sigma)$ iff it is disjoint from $\sigma$ iff $\mu$ is empty.

As per the above, we will go over every simplex $\tau$ of dimension $d>0$ and every non-empty sub-complex $\sigma$ of $\tau$ (for instance, if $\tau=(1,3,4)$ these will be $1$, $3$, $4$, $(1,3)$, $(1,4)$ and $(3,4)$). We also go over every subset $\mu$ of $\sigma$, this time allowing the empty set and $\sigma$ itself, and add $\nu=(\tau \setminus \sigma) \cup \mu$ to $St(\sigma)$. If $\mu$ is empty, we add $\nu$ to $Lk(\sigma)$ as well. Notice that when $\tau=\sigma$ and $\mu$ is empty then $\nu$ is empty, in which case we ignore it and move on.

Since the number of different $\tau$'s to go over is bounded by $O(n)$, and the number of $\sigma$'s and $\mu$'s per a given $\tau$ is bounded, this process takes $O(n)$ time. Computing $St'$ and $Lk'$ is done similarly.
\end{proof}

\begin{rem} \label{AllStIsn}
1) It follows from the proof that there is a number $a$ such that every simplex can appear at the stars of at most $a$ simplices. This implies that the sum of the sizes of all stars, $\sum_\sigma \#St(\sigma)$, is bounded by $a$ times the size of $M$, and is thus $O(n)$. The same applies to the sizes of $Lk$, $St'$ and $Lk'$.

2) As in Remark~\ref{SortSimp}, we can sort the lists of $d$-simplices in every star $St(\sigma)$ according to lexicographical order. For a single star, this takes $O(\#St(\sigma) \cdot \log(\#St(\sigma)))$ time. For all stars, this takes $\sum O(\#St(\sigma) \cdot \log(\#St(\sigma))) \leq \sum O(\#St(\sigma) \cdot \log(an+b)=O(n \cdot \log(n))$ time. We can similarly sort the simplices in every $Lk(\sigma)$, $St'(\sigma)$ and $Lk'(\sigma)$ in $O(n \cdot \log(n))$ time.
\end{rem}

\section{A \gls{genericsurface} as a data type - manifolds} \label{DTsec2}

In this section we explain the complexity of verifying that the total space $|M|$ of $M$ is a 3-manifold. We also comment on the complexity of verifying that a simplicial complex of dimension $m$ is an $m$-manifold. Later, we study the complexity of the first few checks needed to verify that $S$ is a \gls{genericsurface} in $M$.

Any triangulation of a 3-manifold is combinatorial (see \citep[p.165-168, theorems 23.1 and 23.6]{MoiB1}). There are several equivalent definitions for a combinatorial manifold. The most common definition is that an $m$-dimensional complex is a combinatorial manifold if the link of every $d$-dimensional simplex $\sigma$ in $M$ (for $d=0,1,...,n-1$) $Lk(\sigma)$ is an $m-1-d$ dimensional PL-sphere or ball, although it is actually enough for this to hold for $d=0$, which implies this for every other $d$. For a short but encompassing introduction to combinatorial manifolds, see \citep[chapter 5, pp. 20-28]{Hud1}.

\begin{rem} \label{ManBySpheres}
The complexity of verifying that an $m$-dimensional pure complex is a combinatorial manifold is tied to the complexity of determining if a connected combinatorial $m-1$-manifold is PL-homeomorphic to a sphere or a ball. The following known devices demonstrate this:

1) On the one hand, given a connected combinatorial $m-1$-manifold $X$, the suspension $SX$ is an $m$-dimensional pure complex. Its triangulation is as follows: it has two new 0-simplices $a$ and $b$ and, for every simplex $\sigma$ in $X$, $SX$ contains the simplices $\sigma$, $\sigma \cup a$ and $\sigma \cup b$. For every $\sigma$ in $X$, the links of $\sigma \cup a$ and $\sigma \cup b$ in $SX$ are equal to the link of $\sigma$ in $X$, and the link of $\sigma$ in $SX$ is equal to the suspension of the link of $\sigma$ in $X$. In particular, all these links are spheres or balls.

The only remaining simplices of $SX$ are $a$ and $b$, and each of their links is clearly equal to $X$. It follows that $SX$ is a combinatorial manifold iff $X$ is PL sphere or ball. This implies that checking if an $m$ complex is a manifold is at least as hard as checking that an $m-1$ manifold is a sphere or a ball.

2) On the other hand, checking that a pure $m$-dimensional simplicial complex $X$ in a manifold involves proving that the link $Lk(\sigma)$ of any $d$-simplex $\sigma$ for any $d \leq m-1$ is an $m-d-1$ dimensional sphere or disc. $Lk(\sigma)$ is clearly a pure $m-1-d$-dimensional sub-complex of $X$. The following remains to be checked:

a) For $d<m-1$, an $m-d-1>0$ dimensional sphere or disc is connected. A simplicial complex is connected iff its 1-skeleton, a graph, is connected. Checking that a graph with $\#V$ vertices and $\#E$ edges is connected takes $O(\#E+\#V)$ time, as is proven in \cite{Hop&Tar1}. In particular, checking that $Lk(\sigma)$ is connected takes $O(\#Lk(\sigma))$ time. Similarly to Remark~\ref{AllStIsn}(1), the sum of the sizes of all links in $X$ is $O(\#X)$. It thus takes linear $O(\#X)$ time to verify that they are all connected.

b) For $d=m-1$, the link of an $m-1$ simplex is a finite set of 0-simplices. One should only check that there are either 1 or 2 of them (A 0-dimensional ball/sphere is just 1/2 points). Again, this takes $O(\#X)$ time.

c) For lower dimensions, induction can be used. The induction uses the facts that if $\sigma$ is a $c$-dimensional simplex in $M$ and $\tau$ is a $d$ dimensional simplex in $Lk(\sigma)$, then $\sigma \cup \tau$ is a $c+d+1$-dimensional simplex in $M$, and that the link of $\tau$ \textbf{inside $Lk(\sigma)$} - its $Lk(\sigma)$-link - is equal to the link of $\sigma \cup \tau$ in $M$. They both have the same definition - the set of simplices in $M$ that are disjoint from $\sigma$ and $\tau$, but whose union with $\sigma \cup \tau$ is in $M$.

Given a dimension $k<m-1$, assume that for every $k<d<m$ the link of a $d$-simplex in $M$ is an $m-1-d$-dimension disk or a sphere. The link of every $c$-simplex inside the link of a $k$ simplex is thus an $m-1-(c+k+1)=(m-1-k)-1-c$-sphere or a ball. This implies that that the link of every $k$-simplex in $M$ is an $m-1-k$-dimensional combinatorial manifold. Step (a) verified that this manifold is connected.

In order to progress to the next phase of the induction, one will have to prove that each of these combinatorial manifolds is a PL ball or sphere. Once one does so, the same argument implies that the links of $k-1$-simplices are connected $m-k$-manifolds, and one will need to prove that these are all PL-spheres or balls and so on.


3) This implies that checking if a pure $m$-complex is a combinatorial manifold is harder than checking if a combinatorial $m-1$-manifold is a sphere/ball, but is easier than checking if a combinatorial $d$-manifold is a sphere/ball for an arbitrary $0 \leq d \leq m-1$. Intuitively, this gets harder as $d$ increases, and reaches maximal difficulty for $d=m-1$, hence the equivalence. But the complexity of this last problem is not fully understood for some $d$s.

In the next paragraphs, we will show that this takes $O(\#X)$ time for $d=0,1$ and $O(\#X \cdot \log(\#X))$ times for $d=2$. For $d=3$, the problem is known to be NP (Schleimer, \cite{Sch1}), and in \cite{Has&kup1} Hass and Kuperberg proved that it is also co-NP, assuming that the generalized Riemann hypothesis holds. These results are formulated for the recognition of 3-spheres but can be easily modified for 3-balls.  For $d=4$, very little is known. It is not even known if they are exotic 4-spheres - PL 4-manifolds that are topologically homeomorphic, but not PL-homeomorphic, to $S^4$.

However, for every $d \geq 5$ this problem is known to be undecidable - no algorithm can solve it, regardless of runtime. The original proof is in Russian, by Novikov, in an appendix of \cite{Vol&Kuz&Fom}. The main idea is that one can check (in polynomial time) if $X$ is a homology $d$-sphere, and such a homology sphere is PL-homeomorphic to a sphere iff it is simply connected. The real problem is looking at a presentation of the fundamental group of $X$, and determining if it is trivial. Novikov proves that this is impossible, using a variation of the Adian-Rabin theorem.

4) Some higher-dimensional manifolds have non-combinatorial triangulations. One can use a similar technique to determine if a simplicial complex is a non-combinatorial manifold, using a theorem of Galewski and Stern. See \cite{Gal&Ste1}
\end{rem}

\begin{lem} \label{MisMan}
It takes $O(n \cdot \log(n))$ time to check that a pure $3$-dimensional complex $M$ is a $3$-manifold.
\end{lem}

\begin{proof}
As per Remark~\ref{ManBySpheres}(1) one must first check that for every 2-simplex $\sigma$, $Lk(\sigma)$ contains only 1 or 2 0-simplices. This clearly takes $O(1)$ time per 2-simplex and $O(n)$ time for all 2-simplices. If this holds then, as per Remark~\ref{ManBySpheres}(1), every 1-simplex $\sigma$ $Lk(\sigma)$ is a compact 1-manifold. The only connected compact 1-manifolds are the 1-sphere (circle) and the 1-disc (interval), and so it is enough to prove that every $Lk(\sigma)$ is connected. As per Remark~\ref{ManBySpheres}(1) this takes $O(n)$ time.

Lastly, we may assume that for every 0-simplex $\sigma$ $Lk(\sigma)$ is a compact 2-manifold. Again, it takes $O(n)$ time to verify that all these links are connected. Next, we calculate the euler characteristic of every link. This takes $O(1)$ time per link and $O(n)$ time for all links. If the euler char of a link is $\neq 2,1$ it cannot be a sphere or a disc, and so $M$ is not a manifold. If it is equal to $2$ then it is a sphere, and we move on to the next link. If it is 1 then it may either be a disc or a projective plane. In this case, $Lk(\sigma)$ is a disc iff it has a boundary iff there is a 1-simplex in $Lk(\sigma)$ which is contained in only 1, as opposed to 2, 2-simplices of $Lk(\sigma)$.

Go over all the 2-simplices in $Lk(\sigma)$ and list all of their 1-faces (a single list for the faces of all the 2-simplices). Order this list. As in Remark~\ref{AllStIsn}(2) this takes $O(\#Lk(\sigma) \cdot \log(\#Lk(\sigma)))$ time per 0-simplex $\sigma$ and $O(n \cdot \log(n))$ time for all such 0-simplices. A non-boundary edge will appear twice in the list. Go over the list and search for two consecutive instances of the same 1-simplex. Delete every such pair, and what remains is the list of 1-simplices of the boundary $\partial Lk(\sigma)$. $Lk(\sigma)$ is a disc iff this list is not empty. 
\end{proof}

\begin{rem} \label{save}
During its run this algorithm determined which of the 0-simplices have a disc / sphere for a link and similar information that will be useful later on. We will store this information. Specifically:

1) During this calculation we can store a list of all the $d$-simplices ($d=0,1,2$) whose link is a disc, and do the same for simplices whose links are spheres. The algorithm did not check which 1-simplex is of which kind, but it can do so by checking the euler char of the link. The lists will be lexicographically ordered. The lists of all $d$-simplices whose links are discs form the sub-complex $\partial M$ of $M$. The lists of simplices whose links are spheres, which we will refer to as ``internal simplices", do not form a sub-complex since internal simplices may have faces in $\partial M$.

2) We can indicate which 0-simplices are internal/boundary in another way - by creating a list of elements ``b" and ``i" with length $\#V$ (the number of 0 simplices of $M$) such that the $r$th element in this list is ``i" iff the $r$th 0-simplex is internal. This allows us to check if said $r$th simplex is internal or external in $O(1)$ time.

3) By going over the list of vertices in $S$, we can create a similar list that will allow us to check if a 0-simplex of $M$ is in $S$ or not in $O(1)$ time. It takes $O(n)$ time to create this list.

4) We can also calculate the intersection of the list of internal / boundary $d$-simplices of $M$ with the list of $d$-simplices in $S$. Since all the lists are sorted, and of length $O(n)$, calculation the intersection takes $O(n)$ time.

5) Lastly, for every boundary 0-simplex $\sigma$ in $M$ we will save the sub-complex $\partial Lk(\sigma)$ of $Lk(\sigma)$ as we saved $St(\sigma)$ and $Lk(\sigma)$. We calculated the list of 1-simplices in $\partial Lk(\sigma)$ during the proof of Lemma~\ref{MisMan}. The 0-simplices of $\partial Lk(\sigma)$ are the 0-faces of the 1-simplices.

Calculating the list of the 0-faces of all 1-simplices of the circle $\partial Lk$ and ordering it takes $O(\#Lk \cdot \log(\#Lk))$ time. Every entry will appear twice in this list, deleting the repeating instances takes $(O(\#Lk))$ time. Doing this for every 0-simlices $\sigma$ in $\partial M$ takes $O(n \cdot \log(n))$ time due to the usual reason (as in Remark~\ref{AllStIsn}(2)). 
\end{rem}

After verifying that $M$ is a 3-manifold and that $S$ is a pure 2-dimensional sub-complex, proving that $S$ is a \gls{genericsurface} requires that one go over all of the $d$-simplices in $S$, for $d=2$, 1 and then 0, and check that the star or link of the simplex has the appropriate shape. Specifically, we need to check that every point in (the geometric realization of) $S$ has a neighborhood of one of the types required in Definition~\ref{Generic}. Every such point is contained in the interior of some simplex, and so we must verify that, for every simplex $\sigma$ in $S$, the internal points of $\sigma$ have a neighborhood as Definition~\ref{Generic} requires.

\begin{lem} \label{Surf2&1}
Doing this for all simplices of dimension $d=2,1$ takes $O(n)$ time.
\end{lem}

\begin{proof}
For $d=2$, a small neighborhood of a point in the interior of a boundary 2-simplex will look like the embedding of the $xy$ plane in the upper half space. In particular, it will not look like any of the neighborhoods of Definition~\ref{Generic}. On the other hand, any point in the interior of a 2-simplex in $S$ that is internal in $M$ will clearly be a regular value - it will have a neighborhood that looks like the embedding of a plane in 3-space. It follows that it is enough to check that there are no 2-simplices in $S \cap \partial M$. This takes $O(1)$ time.

Assuming that every 2-simplex in $S$ is internal, we move on to $d=1$. Figure~\ref{fig:Figure 7} shows that an interior point in an internal 1-simplex $\sigma$ in $S$ will have the neighborhood of a regular/double value if the star in $S$, $St'(\sigma)$, contains two / four 2-simplices. If $St'(\sigma)$ contains any other number of 2-simplices, then its neighborhood does not look right and $S$ is not a \gls{genericsurface}. Figure~\ref{fig:Figure 7}C, for example, depicts the case where $St'(\sigma)$ contains 5 2-simplices. Similarly, a 1-simplex in $S \cap \partial M$ must have only 1 2-simplex in its star, in which case its interior points will be DB values. Figure~\ref{fig:Figure 7}D depicts this (the purple 2-simplices are from $\partial M$ while the white one is from $St'(\sigma)$). Checking that a 1-simplex $\sigma$ has the right number of 2-simplices in $St'(\sigma)$ takes $O(1)$ time, and doing this for every $\sigma$ takes $O(n)$ time.
\end{proof}

\begin{figure}
\begin{center}
\includegraphics{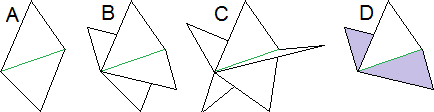}
\caption{The stars of 1-simplices in $S$}
\label{fig:Figure 7}
\end{center}
\end{figure}

\begin{rem} \label{Save3}
While checking that every 1-simplex has the right neighborhood, create a list of all the 1-simplices that are made of double values - the one with 4 triangles in $St'$. If $S$ is indeed a \gls{genericsurface}, then the aforementioned list will be the set of 1-simplices of the intersection graph $X(S)$, and so it will be useful to store this data. 
\end{rem}
%

\section{Homeomorphisms of topological graphs} \label{TGsec}

We move on to the 0-simplices of $S$. The $S$-link of a 0-simplex $\sigma$ in $S$, $Lk'(\sigma)$, is a topological graph. The check whether $\sigma$ has the right sort of neighborhood involves examining the homeomorphism type of this graph. We dedicate this separate section to the graph homeomorphism problem. We will explain what constitutes ``the homeomorphism type" of a graph and the complexity of calculating this homeomorphism type and checking whether two graphs are homeomorphic.

The ``homeomorphism type" of a graph is a slightly different concept than the ``isomorphism type" of the graph. The difference is that the actions of adding a degree 2 vertex into some edge, and thus splitting it into two edges (see Figure~\ref{fig:Figure 8}A) and the opposite action of deleting a degree 2 vertex, both preserve the homeomorphism type of the graph.

\begin{figure}
\begin{center}
\includegraphics{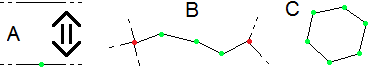}
\caption{A - Adding or deleting a degree 2 vertex, B and C - a long edge and a circle}
\label{fig:Figure 8}
\end{center}
\end{figure}

This is the only difference - homeomorphic graphs are ``isomorphic up to adding or deleting degree 2 vertices". In order to study a graph's homeomorphism type, one may delete all of its degree-2 vertices to produce and examine the isomorphism type of the resulting ``simplified graph". 

\begin{defn} \label{LongEdge}
1) A ``long edge" in a graph is a path $v_0,v_1,...,v_n$ such that for every $k=1,...,n-1$ $v_k$ is a degree-2 vertex, $v_0$ and $v_n$ are vertices of degree $\neq 2$, and each $v_k$ is connected to $v_{k+1}$. This long edge stretches between $v_0$ and $v_n$. It is possible that $v_0=v_n$, in which case the long edge is said to be a long loop. Figure~\ref{fig:Figure 8}B depicts a long edge.

2) Similarly, a circle in a (multi)graph is a path $v_0,v_1,...,v_n=v_0$ such that every $v_k$ is a degree-2 vertex and is connected to $v_{k+1}$. Figure~\ref{fig:Figure 8}C depicts a circle in a graph.

3) The ``simplified graph" produced by deleting all of the degree-2 vertices of a graph clearly has the following structure: it has all the vertices of the original graph whose degree is not 2, and each two of these vertices have one (short) edge between them per every long edge that stretches between them in the original graph. In particular, the simplified graph is a multigraph, possibly with loops, and with no degree-2 vertices. Additionally, the simplified graph will have one disjoint copy of $S^1$ per every circle of the original graph. This may be expressed by adding an integer (that counts the number of circles) to the multigraph structure.
\end{defn}

\begin{rem}
When discussing long edges, we sometimes refer to to the actual edges of the graph as ``short edges" in order to avoid confusion. 
\end{rem}

\begin{lem} \label{SimpHomeo}
1) Two simplified graphs are homeomorphic iff they have the same number of circles and their multigraph portions are graph-isomorphic.

2) Two graphs are homeomorphic iff their simplified graphs are homeomorphic.
\end{lem}

\begin{proof}
(2) is trivial, as is the ``if" part of (1). As for the ``only if" part: for a vertex $a$ of the first graph and its image $f(a)$ to have homeomorphic neighborhoods $f(a)$ must be a vertex of the same degree. This only holds since $deg(a) \neq 2$, otherwise $f(a)$ could have been a point in the middle of an edge. Lastly, an edge between vertices $a$ and $b$ must clearly be sent to a path between $f(a)$ and $f(b)$ that crosses no other vertices. Since there are no degree-2 vertices, this must be an edge between them. This defines an embedding of the first graph portion in the second, and this embedding must be an isomorphism, since it is invertible.

The homeomorphism also sends the rest of the simplified graph, the union of the circles, homeomorphically to the union of circles of the other graph. This implies that the graphs have the same number of circles.
\end{proof}

In order to determine the topological type the graph $Lk'(\sigma)$, we must identify its long edges and circles.

\begin{lem} \label{LongEdgeComp}
Given a $1$-dimensional sub-complex (sub-graph) of $M$ with $\#V$ vertices and $\#E$ edges, identifying the vertices of degree $\neq 2$, long edges, and circles of the sub-graph takes $O(\#V+\#E)$ time.
\end{lem}

\begin{proof}
We begin by calculating, for every vertex $v$, the list $Ad(v)$ of the adjacent vertices to $v$. We do this by setting every such $Ad(v)$ to be an empty list and then, for every edge $(v,u)$ in $G$, add $v$ to $Ad(u)$ and $u$ to $Ad(v)$. This clearly takes $O(\#V+\#E)$ time. The length $\#Ad(v)$ is now the degree $deg(v)$ of $v$. Next, we create a copy of the list of all vertices of the graph, and a list of all vertices of degree $\neq 2$. This takes $O(\#V)$ time.

On to identifying the long edges. We pick a vertex $v_0$ of degree $\neq 2$ and an adjacent vertex $v_1$. There is a long edge that begins at $v_0$, continues into $v_1$ and keeps going until it reaches another vertex $v_r$ of degree $\neq 2$. We find it accordingly - if $v_1$ is of degree 2, we let $v_2$ be its other adjacent vertex - the only vertex in $Ad(v_1)$ except $v_0$. We continue in this way. As long as $\#Ad(j)=2$, we let $v_{j+1}$ be the other vertex in $Ad(v_j)$ (other then $v_{j-1}$). We add every new $v_j$ to a list $v_0,v_1,v_2,...,v_r$, and end this list at the first $r$ for which $\#Ad(v_r) \neq 2$. We then delete $v_1,...,v_{r-1}$ from the copy of the list of the vertices of the graph.

Additionally, weא delete $v_1$ from $Ad(v_0)$, so as not to accidentally try to construct the same long edge again. If $Ad(v_0)$ is now empty, then there are no more long edges that begin or end in it, and we delete it also from both the copy of the list of the vertices of the graph and the list of vertices of degree $\neq 2$. Similarly, we delete $v_{r-1}$ from $Ad(v_r)$ and, if $Ad(v_r)$ becomes empty, we delete $v_r$.

We let $v_0,v_1,v_2,...,v_r$ be the first long edge is a list of all the long edges (which are themselves lists). Next, we pick a new vertex $v_0'$ from the list of vertices of degree $\neq 2$, and a new adjacent vertex $v_1'$, and repeat the same process. We add the new long edge to the list of long edges and continue in this way, until all of the vertices of degree $\neq 2$ have been deleted from the list. If there are any vertices left in the copy of the list of all vertices of the graph they must have degree 2 and cannot be on any long edge, and so they come from the circles of the graph.

We calculate the circles similarly. We pick one of the remaining vertices in the copy list and call it $v_0$, pick an adjacent vertex and call it $v_1$, and for every $j>0$ we pick the vertex in $Ad(v_j)$ that is not $v_{j-1}$ to be $v_{j+1}$. We go over the whole circle until we return to $v_0$, and end the process in the first $r$ for which $v_r=v_0$. The list $v_0,v_1,...,v_r=v_0$ is a circle. It is the first in a list of all circles. We delete $v_0,...,v_{r-1}$ from the copy list, pick a new vertex $v_0'$ and identify the next circle. We finish when there are no more vertices in the copy list.

This process made a bounded number of simple actions for every pair of adjacent vertices, which means that it made a bounded number of actions per every edge of the graph, so it all took $O(\#E)$ time. The need to calculate $Ad(v)$ increases the runtime to $O(\#V+\#E)$. This is a necessary addition in case the graph has isolated vertices.
\end{proof}

\section{A \gls{genericsurface} as a data type - The vertices of $S$} \label{DTsec3}

We return to verifying that the input $(M,S)$ is valid. All that remains at this point is to check that every vertex $s$ of $S$ has one of the types of neighborhood required by Definition~\ref{Generic}. As mentioned above, this is done by examination of the homeomorphism type of the $S$-link $Lk'(\sigma)$. We dedicate this section to proving this, and showing that it can be done in polynomial time.

\begin{thm} \label{GenLink}
1) An internal $0$-simplex of $S$ is:

a) A regular value iff $Lk'(\sigma)$ is a circle.

b) A double value iff $Lk'(\sigma)$ is homeomorphic to a multigraph with $2$ vertices and $4$ edges between them.

c) A triple value iff $Lk'(\sigma)$ is homeomorphic to the $1$-skeleton of a octahedron - a graph with degree $6$ vertices, each of which is connected (via a single edge) to exactly $4$ of the other vertices.

d) A branch value iff $Lk'(\sigma)$ is homeomorphic to an $8$-graph - a multigraph with $1$ vertex and $2$ loops.

2) A vertex in $S \cap \partial M$ is:

a) An RB value iff $Lk'(\sigma)$ is an interval that is properly embedded in the disc $Lk(\sigma)$. ``Proper" means that the boundary of the interval, its degree-1 vertices, are the only part of the interval that is contained in the boundary circle $\partial Lk(\sigma)$.

b) A DB value iff $Lk'(\sigma)$ is a properly embedded ``X" in the disc $Lk(\sigma)$ - two properly embedded intervals that intersect each other transversally.
\end{thm}

Figure~\ref{fig:Figure 9} depicts one graph from each of the 6 homeomorphism types. The set of purple vertices is the intersection of the graph with $\partial Lk(\sigma)$. The green dots along the (long) edges are degree-2 vertices. They are there because of the specific triangulation of $M$ and $S$, and do not effect the homeomorphism type of the graph. 

\begin{figure}
\begin{center}
\includegraphics{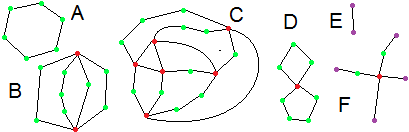}
\caption{The links of 0-simplices $S$}
\label{fig:Figure 9}
\end{center}
\end{figure}

The description in each item defines a unique graph up to homeomorphism. We consider this statement to be trivial except in the case of item (1C), where we will prove it: 

\begin{cor} \label{Oct}
Up to isomorphism, there is a unique graph with $6$ vertices, for which every vertex is connected to exactly $4$ of the other vertices.
\end{cor}

\begin{proof} [Proof of Corollary~\ref{Oct}]
Let us name one of the vertices $0$. It is connected to all other vertices save for one. Call said vertex $1$. $1$ is not connected to $0$ and thus must be connected to each of the other vertices. Choose another vertex and call it $2$. As per the above it is connected to $0$ and $1$ and must be connected to two of the other vertices. Call these two other vertices $4$ and $5$. There is one vertex in the gaph that is not connected to the vertex $2$. Call this vertex $3$. Since $3$ is not connected to $2$ it must be connected to all other vertices. $4$ is now known to be connected to $0$, $1$, $2$ and $3$, and thus cannot be connected to $5$. It follows that our graph must be isomorphic to the full graph on $0,...,5$ minus the edges $\{0,1\}$, $\{2,3\}$ and $\{4,5\}$.
\end{proof}

We will prove the two directions of Theorem~\ref{GenLink} separately, starting with the ``only if" direction. That proof relies on the fact that each of these graphs can be embedded in the sphere or disc in a unique way:

\begin{lem} \label{UnEmb}
1) Each of the graphs in Theorem~\ref{GenLink}(1) can be embedded in the sphere $S^2$ in a unique way up to homeomorphism of the graph and/or sphere.

2) Similarly, each of the graphs in Theorem~\ref{GenLink}(2) has a unique proper embedding in the disc $D^2$ up to homeomorphism.
\end{lem}

\begin{proof}[Proof of Lemma~\ref{UnEmb}]
We prove the most complicated case - the 1-skeleton of a octahedron. The other cases are similar. Figure~\ref{fig:Figure 10} illustrates the proof. We represent the sphere as a plane with a point in infinity. Pick a ``triangle of long edges" of the graph and embed it in the sphere as in Figure~\ref{fig:Figure 10}A - there is clearly a unique way to do this up to homeomorphism. Each of the 3 remaining (degree $\neq 2$) vertices are connected to each other via long edges. This implies that any embedding must either put them all ``inside" or ``outside" the triangle. There is a homeomorphism of the sphere that preserves the triangle and exchanges the inside and outside, so we may assume WLOG that the remaining vertices are outside.

Place one of the remaining vertices outside the triangle, as per Figure~\ref{fig:Figure 10}B. It must be connected to exactly two of the triangle vertices. Connect it to them. Up to homeomorphism, it looks like Figure~\ref{fig:Figure 10}C. The sphere is now tiled with two triangles and a square. All of the following vertices must be in the square, since they must be connected to the left and right vertices. One new vertex must be connected to the left, right and bottom vertices. Add it and its edges to the embedded graph which will look like Figure~\ref{fig:Figure 10}D up to homeomorphism. Now there are 4 triangles and a square. The last vertex must be in the square and be connected to all of its vertices, as in Figure~\ref{fig:Figure 10}E. Due to the determinism of this construction, this is the only embedding up to homeomorphism.
\end{proof}

\begin{figure}
\begin{center}
\includegraphics{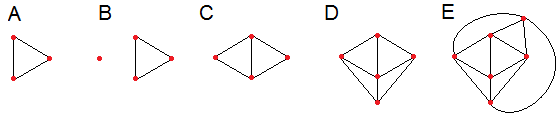}
\caption{Embedding the 1-skeleton of a octahedron in a sphere}
\label{fig:Figure 10}
\end{center}
\end{figure}

\begin{proof} [Proof of the ``only if" direction of Theorem~\ref{GenLink}]
All 6 items have a similar proof. Take the unique (proper) embedding of the graph in the sphere/disc. This must be ``pair-homeomorphic" to the pair of spaces $(Lk(\sigma),Lk'(\sigma))$. Look at the product of this pair of spaces with an interval $I=[0,1]$ - the pair $(Lk(\sigma) \times I,Lk'(\sigma) \times I)$, then quotient it by sending all of $Lk(\sigma) \times \{0\}$ to a single point. The result must be PL homomorphic to $(St(\sigma),St'(\sigma))$. Figures~\ref{fig:Figure 11}A and B depict this for a triple value.

For each one of the 6 types of graphs, one may simply look at this pair of spaces and see that it is PL-homeomorphic to the neighborhood of the matching type of value. In truth, $(St(\sigma),St'(\sigma))$ is not a neighborhood of $\sigma$ in $(M,S)$ - $(St(\sigma),S \cap St(\sigma))$ is. It is possible that $S \cap St(\sigma)$ may contain simplices that are not in $St'(\sigma)$. However, one may fix this by looking at a smaller neighborhood, like the image of $(Lk(\sigma) \times [0,\frac{1}{2}],Lk'(\sigma) \times [0,\frac{1}{2}])$. This will be a neighborhood of $\sigma$ in $(M,S)$. Figure~\ref{fig:Figure 11}C depicts an example in a lower dimension - the star of an intersection point in a generic loop on a surface. The red simplices are in $St'$ and the blue simplices are in $S \cap St$ but not in $St'$.
\end{proof}

\begin{figure}
\begin{center}
\includegraphics{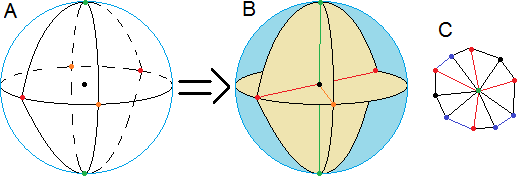}
\caption{A: From link-graph to star, B: $St \cap S$ may contain simplices that are not in $St'$}
\label{fig:Figure 11}
\end{center}
\end{figure}

We move on to the ``if" direction. Assume that $S$ is a triangulated \gls{genericsurface} in $M$, and recall that $S$ is thus the image of a generic and proper PL map $i:F \to M$, where $F$ is a surface - the underlying surface of $S$. Each simplex in $S$ is pulled back into a number of simplices in $S$, depending on what type of values compose the simplex. Definition~\ref{Generic} says that regular, RB and branch values have 1 preimage in $F$, double and DB values have 2 preimages, and triple values have 3 preimages. A 0-simplex $\sigma$ of $S$ is pulled back to 1, 2 or 3 0-simplices of $F$, accordingly. We refer to them as $\sigma_1,...,\sigma_q$ ($1 \leq q \leq 3$). The internal points of a 2-simplex are regular values, and therefore it is pulled back to only one 2-simplex in $F$. The same goes for 1-simplices that consist of regular values, but 1-simplices that consist of double values (1-simplices in $X(S)$) have two preimages in $F$.

The general idea of the proof is that the link $Lk'(\sigma)$ of a 0-simplex $\sigma$ in $S$ is the image of the union of the links of its preimages in $F$, $\bigcup Lk(\sigma_k)$, via the quotient map $i$. If the triangulation is well behaved, $\bigcup Lk(\sigma_k)$ will be a disjoint collection of 1, 2 or 3 circles. $i$ will glue some of the degree-2 vertices of these circles together, in a predictable way, and this will produce the necessary graph. However, the triangulation of $F$ may be slightly pathological. Following are some properties of the triangulation of $F$ and its connection to that of $S$: 

\begin{lem} \label{TriLem}
i) For every $0$-simplex $\sigma$ in $S$, $St'(\sigma)=\bigcup i(St(\sigma_k))$ and $Lk'(\sigma)=\bigcup i(Lk(\sigma_k))$.

ii) The vertices $a$ and $b$ of a $1$-simplex in $F$ are distinct ($a \neq b$), and furthermore $i(a) \neq i(b)$. The same holds for $2$-simplices.

iii) If $(a,b)$ is a $1$-simplex in $X(S)$ and $a$ is a double, triple or DB value, then each of the two pullbacks of $(a,b)$ will contain a different preiamge of $a$. The same holds for $b$.

iv) If two simplices of degree $d>0$ in $F$ have the exact same vertices, then $d=1$ and they are the two pullbacks of the same $1$-simplces in $X(S)$ and both vertices of this edge are branch values.

v) If $p$ is an internal $0$-simplex in $F$ (it is not in $\partial F$), and unless: a) $i(p)$ is a branch value or b) the other vertex of the unique edge in $X(S)$ that contains $i(p)$ is also a branch value, then $(St(p),p)$ is homeomorphic to $(D^2,\{0,0\})$ - $St(p)$ is a disc and $p$ is a point in the interior of this disc.

vi) If $p$ upholds the conditions (a) and (b) that where excluded in (v) then $St(p)$ is the space depicted in Figure~\ref{fig:Figure 12}F. $St(p)$ is a disc with $p$ in its interior, and two vertices on the boundary of the disc coincide/are glued together.

vii) If $p$ is a 0-simplex on the boundary of $F$, then $(St(p),p)$ is homeomorphic to $(D_+^2,\{0,0\})$ - $St(p)$ is a disc and $p$ is a point in the boundary of this disc, as Figure~\ref{fig:Figure 12}G depicts. The purple points are the boundary of the interval $Lk(p)$.
\end{lem}

\begin{proof}
ii) This item follows from the fact that $(i(a),i(b))$ (or $(i(b),i(a))$) is a 1-simplex in $S$, and 1-simplices have distinct vertices (see Remark~\ref{Calculate}(3)). 

ii $\Rightarrow$ i) $St'(\sigma)=\bigcup i(St(\sigma_k))$ since a simplex $\tau \in F$ is in $St(\sigma_{k_0})$ for some ${k_0}$ iff it is contained in a 2-simplex that contains $\sigma_{k_0}$ iff $i(\tau)$ is contained in some 2-simplex of $S$ that contains $\sigma$ iff $i(\tau) \in St'(\sigma)$. 

Similarly, $i(\tau) \in Lk'(\sigma)$ iff $i(\tau) \in St'(\sigma)$ and $i(\tau)$ does not contain $\sigma$ iff $\tau \in \bigcup St(\sigma_k)$ and $\tau$ does not contain any $\sigma_k$ iff for some $k_0$ $\tau \in Lk(\sigma_{k_0})$ and for every other $k$ $\tau$ does not contain $\sigma_k$. (i) implies that the second condition is redundant - if $\sigma_k \subseteq \tau \in Lk(\sigma_{k_0})$, then $\sigma_k \subseteq \tau \in Lk(\sigma_{k_0})$, and there is a 1-simplex that connects $\sigma_{k_0}$ and $\sigma_k$. This implies that $Lk'(\sigma)=\bigcup i(Lk(\sigma_k))$.

iii) $(a,b)$ is an interval that is contained in a double arc of $S$. As Figure~\ref{fig:Figure 12}A depicts, there are two surface strips that intersect along such an interval (marked in light orange and light green). Each of the preimages of $(\sigma,\tau)$ comes from a different strip, and in particular each of the preimages of $\sigma$ / $\tau$ come from a different strip and are thus different. 

Figure~\ref{fig:Figure 12}A depicts the case where $a$ is a double value and $b$ is a triple value, but it would look roughly the same in any case where $a$ and $b$ are each double, DB or triple value. If $b$ is a branch value, as in Figure~\ref{fig:Figure 12}B, the two strips will meet at $b$, but this is the only difference, and the argument still holds. This reflects the fact that $b$ has only one preimage, and both pullbacks of $(a,b)$ will have this preimage as one of their vertices, but their other vertices will be different preimages of $a$.

iv) The $i$ images of these simplices will be $d$-simplices of $S$ with the same vertices. As per Remark~\ref{Calculate}(3), the $i$-images of the two simplices in $F$ are equal, implying that these two simplices are pullbacks of the same simplex in $S$. The only $d$-simplices in $S$ with $d>0$ and more than one pullback are 1-simplices in $X(S)$. Furthermore, as per (iii), if either of the vertices of this 1-simplex is not a branch value, then its preimages in $S$ will not have the exact same vertices.

v) Even in the most pathological triangulation of a surface $F$, every internal 0-simplex $p$ must have a neighborhood like the one depicted in Figure~\ref{fig:Figure 12}C. There are $m$ ``wedges" $w_1,...,w_m$ arranged is a circle around $p$ for some $m$ ($m=5$ in the figure). Each one is one of the three corners of some 2-simplex in $F$, and each of them shares a small line segment with the following one, and the segment is one of the ends of one of the 1-simplices of $F$. In a pathological triangulation, some the wedges may come from the same 2-simplex, or there may be only one wedge. For instance, Figure~\ref{fig:Figure 12}D depicts a triangulation of the sphere with only two 2-simplices.

However, in our triangulation the different 1-simplices that contain $p$ have each a different end point. This implies that each segment is a part of a different 1-simplex, that these 1-simplices have distinct ``other vertices". The argument follows, as Figure~\ref{fig:Figure 12}E depicts.

vi) In this case, two of the segments continue into 1-simplices that end in the same vertex $q$. These are the pullbacks of the 1-simplex of $X(S)$ that contains $i(p)$, and $i(q)$ is the other vertex of this 1-simplex. Every other segment continues into a 1-simplex with a distinct ``other vertex". The argument follows.

vii) Similar to (v).
\end{proof}

\begin{figure}
\begin{center}
\includegraphics{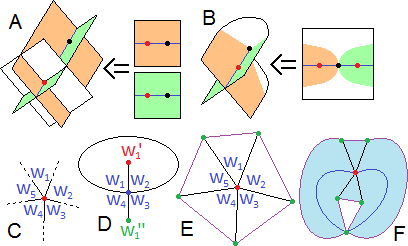}
\caption{A,B - the pullback of an edge made of double values}
\label{fig:Figure 12}
\end{center}
\end{figure}

We will use the properties of the triangulations of $F$ and its connection to that of $S$ to prove the ``if" direction of Theorem~\ref{GenLink}:

\begin{proof}[Proof of the ``if" direction of Theorem~\ref{GenLink}]

We prove the ``if" direction for the hardest cases - triple and branch values. All other cases are similar to, but simpler than, the triple value case.

Let $\sigma$ be a triple value. There are 3 small sheets in $F$ that $i$ embeds in $M$ in such a way that their images intersect transversally in $\sigma$. Each of the sheets contains one of the preimages of $\sigma$. Each preimage $\sigma_i$ has a neighborhood like those mentioned in the proof of Lemma~\ref{TriLem}(v). We draw these neighborhoods side by side in \ref{fig:Figure 13}A. As per Lemma~\ref{TriLem}(v), each such neighborhood extends into the star of one of the $\sigma_i$s, and this star is a topological disc. We draw each of the stars in Figure~\ref{fig:Figure 13}B.

The intersection of the images of each two of the three sheets forms one of the arc segments that intersect at $\sigma$. In each segment, there are two 1-simplices of $X(S)$ that end in $\sigma$. We index the arc segments ``1", ``2" and ``3", we name the two 1-simplices in the $l$th segment $l^+$ and $l^-$ ($1^+$, $2^-$ etc) as depicted in Figure~\ref{fig:Figure 13}C colored in blue. We use the term $\partial l^+$ to refer to the other vertex of the pullback $l^+$, ``other" meaning not a preimage of $\sigma$. We use $\partial l^-$ similarly. In Figure~\ref{fig:Figure 13}, we depict $\partial l^{\pm}$ in green.

One of the 3 stars contains pullbacks of the 1st and 2nd arc segments, one contains pullbacks of the 1st and 3rd arc segments and one contains pullbacks of the 2nd and 3rd arc segments. In Figures~\ref{fig:Figure 13}A,B, we mark the pullbacks of the 1-simplices from $X(S)$ in blue, and we indicate which 1-simplex of $S$ ($1^+$, $2^-$ etc) each blue 1-simplex is a pullbacks of. The other end of each pullback of $l^{\pm}$ ends in a preimage of $\partial l^{\pm}$. We mark these preimages in green.

The link of each preimage of $\sigma$ is thus a circle that contains 4 green vertices - preimages of $\partial l_1^{\pm}$ and $\partial l_2^{\pm}$ for some $l_1,l_2$. Note that the preimages of $\partial l_1^+$ and $\partial l_1^-$ are situated on opposite sides of the circle - a path that connects them must cross either $\partial l_2^+$ or $\partial l_2^-$. For instance, on the first (leftmost) circle in Figure~\ref{fig:Figure 12}B, $\partial 1^+$ and $\partial 1^-$ are at the top and bottom of the circle respectively, while $\partial 2^+$ and $\partial 2^-$ are on the left and right.

Intuitively, $Lk'(\sigma)$ is created by taking the 3 circles and, for each $l$, gluing the two preimages of $\partial l^+$ to each other, and gluing the two preimages of $\partial l^-$ to each other. This clearly produces a graph that is homeomorphic to the 1-skeleton of an octahedron. However, there are two fine points that we must address before this proof can be considered as complete.

a) We must explain why no other parts of the 3 links are glued together. Let us look at the different vertices and edges of each link. Begin with the green vertices. We know that the two preimages of each $\partial l^{\pm}$ are glued together, but can preimages of different $\partial l^{\pm}$s also be glued together? For instance, can preimages of $\partial 2^+$ and $\partial 3^+$ be glued together? The answer is no, since we know that they have different $i$ images - that $\partial 2^+$ and $\partial 3^+$ are different vertices of $S$. Formally, we know this since we know that $2^+$ and $3^+$ are different 1-simplices in $S$, and two different 1-simplices in $S$ cannot connect $2^+=3^+$ to $\sigma$.

Similarly, none of the black vertices of the link can be glued together to any other black vertex or to any green vertex. This is because each black vertex is connected to a preimage of $\sigma$ with a black 1-simplex - the preimage of a 1-simplex in $S$ that is not in $X(S)$. This sort of 1-simplex can only have one preimage in $F$. Had two vertices been glued together and at least one was black, then the two 1-simplices that connect them to preimages of $\sigma$ would be pullbacks of the same 1-simplex is $S$. But as we just explained, this 1-simplex can only have 1 preimage.

Lastly, look at the edges of the links. Two edges can be glued together only if the vertices at their ends are glued together, but this is impossible. Each edge either has at least one black vertex, which cannot be glued to other vertices, or it connects a preimage of some $\partial l_1^{\pm}$ to a preiamge of some $\partial l_2^{\pm}$ for $l_1 \neq l_2$ (We depicted this for $\partial 1^-$ and $\partial 3^+$ and for $\partial 2^+$ and $\partial 3^-$). This edge could only be glued to another edge that connects preimages of $\partial l_1^{\pm}$ and $\partial l_2^{\pm}$, but this is impossible since each star contains pullbacks of a different pair of arc segments.

b) While they are drawn as disjoint sets, the different stars may meet. As per the proof of Lemma~\ref{TriLem}(i), a point in the intersection of two stars must reside on the boundary (link) of each star. Obviously, if some points on the links have different images in $S$, they cannot coincide. Therefore, as per the above, the only points that can coincide are different preimages of the same $\partial l^{\pm}$. Figure~\ref{fig:Figure 13}D demonstrates this for the preimages of $\partial 3^+$ - the two rightmost discs of Figure~\ref{fig:Figure 13}B are glued together at these preimages. In the proof Lemma~\ref{TriLem}(v) we saw exactly under which conditions these two preimages coincide- precisely when the said $\partial l^{\pm}$ is a branch value. In any case, since the preimages of $\partial l^{\pm}$ are glued together by $i$, gluing them together before that, in $F$, will not change the topology of $Lk'(\sigma)$.

For regular, double, DB and RB values, the proof is similar but simpler. Branch values have a slight difference. A branch value $\sigma$ has one preimage $\sigma_1$. Only one 1-simplex in $X(S)$ ends in $\sigma$. We refer to it as $\tau$ and to its other end as $\partial \tau$. The star of $\sigma_1$ will contain both of the preimages of $\tau$. It will usually be a disc and $Lk(\sigma_1)$ will be a circle with two distinct vertices on it - the preimages of $\partial \tau$. These vertices will be glued together by $i$. As in the proof for triple values, $i$ will not glue any other vertices or edges together, and so $Lk'(\sigma)$ will be the graph one gets by taking a circle and gluing two points on it together - an 8-graph. However, if $\partial \tau$ is also a branch value, then the two perimages will already coincide in $F$. In this case, $St(\sigma_1)$ will look like Figure~\ref{fig:Figure 12}F. As in the case where
$\sigma$ is a triple value, this will not affect $Lk'(\sigma)$. In fact, one can see that in this case $Lk(\sigma_1)$ is already an 8-graph. 
\end{proof}

\begin{figure}
\begin{center}
\includegraphics{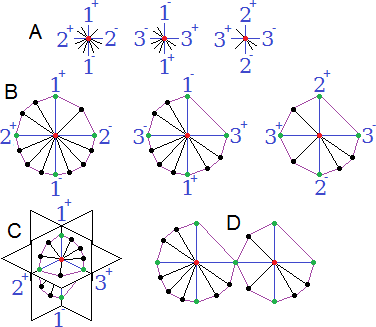}
\caption{The $Lk'$ of a triple value is the 1-skeleton of a octahedron}
\label{fig:Figure 13}
\end{center}
\end{figure}

Now that we know how to verify that $S$ is a \gls{genericsurface} in $M$, we will examine the complexity of this process. As per Theorem~\ref{GenLink} we must verify that each link $Lk'(\sigma)$ is homeomorphic to one of the topological graph types specified in Lemma~\ref{GenLink}. We begin by calculating the intersection of $Lk'(\sigma)$ with $\partial Lk(\sigma)$ for every boundary 0-simplex $\sigma$. Recall that we calculated $\partial Lk$ in Remark~\ref{save}(5). Intersecting list of edges/vertices in $Lk'$ with that of $\partial Lk$ takes $O(\#Lk)$ time since these lists are all ordered. One then verifies that there are 0 such edges and 2 or 4 such vertices. Doing this for every $\sigma$ takes $O(n)$ time.

The next step is verifying that every $Lk'$ has the right number of vertices of every degree $\neq 2$ and, if $\sigma \in \partial M$, that the vertices that are in $\partial Lk$ are exactly those vertices of degree 1:

\begin{lem} \label{Deg}
Verifying this takes $O(n)$ time.
\end{lem}

\begin{proof}
As per Lemma~\ref{LongEdgeComp}, calculating the degree $deg(v)$ of every vertex in $Lk'(\sigma)$, the lists of the degree $\neq 2$ vertices, the long edges, and the circle of $Lk'(\sigma)$ all takes $O(\#Lk'(\sigma))$ time. In order to be one of the graphs listed in Theorem~\ref{GenLink}, the graph can have no more than 6 vertices of degree $\neq 2$. Verifying this takes $O(1)$ time. Verifying that each of them is of degree 1 or 4, and count the numbers $Ver_1$ and $Ver_4$ of vertices of each degree, also takes $O(1)$ time. If $\sigma \in \partial M$ then $(Ver_1,Ver_4)$ should be equal to $(2,0)$ or $(4,1)$. Otherwise, it should be equal to $(0,0)$, $(0,1)$, $(0,2)$ or $(0,6)$. Lastly, if $\sigma \in \partial M$, one should verify that the vertices of $Lk'(\sigma) \cap \partial Lk(\sigma)$ are exactly the vertices of degree 1. Verifying all of this clearly takes $O(1)$ time. Doing this for every $\sigma$ takes $O(n)$ as usual.
\end{proof}

\begin{rem} \label{save2}
1) The value of $(Ver_1,Ver_4)$ implies the kind of value that $\sigma$ should be ($(0,6)$ for a triple value, $(4,1)$ of a DB value, etc). As in Remark~\ref{save}, one can store this information in two ways. Firstly, by creating lists of all the ``potential triple values", all the ``potential DB values", etc. The adjective "potential" is used since we only verified that each graph has the right number or vertices of each degree, but it may still have the wrong homeomorphism type. We will verify this shortly, and then it will be appropriate to forgo the ``potential" adjective.

Secondly, one may create a list of $\#V$ elements such that, if the $r$'th vertex of $M$ is in $S$, the $r$th entry in the list is the pair $(Ver_1,Ver_4)$ of this vertex (if the $r$th vertex is not in $S$ then the $r$th entry is irrelevant). This can be used to check, in $O(1)$ time, what type of value is a given 0-simplex in $S$.

2) In later calculations, we will refer to the $k$'th entry in the list of triple values as $TV_k$, as per the conventions of section~\ref{PremSec}. Formally, $TV_k$ is just a number - the index of the appropriate vertex of $M$. We can create a similar list of $\#V$ entries whose $r$th entry is $k$ iff the $r$th vertex in $M$ is $TV_k$. It takes $O(n)$ time to write this list, and it can be used to check, in $O(1)$ time, what is the matching $k$ of a given triple value.

3) We also save all the information calculated about the topological graph structure of every $Lk'(\sigma)$ - the list of degree $\neq 2$ vertices, the degree $deg(v)$ of each vertex, the lists of long edges and circles, to be used later. As usual, it will take $O(1)$ time to access each such list but it will take $O(\#Lk'(\sigma))$ to search it.
\end{rem}

Lastly, we must make sure that every potential double/triple/etc value is indeed a double/triple/etc value - that $Lk'$ has the required homeomorphism type.

\begin{lem} \label{Surf0}
This takes $O(n)$ time.
\end{lem}

\begin{proof}
For a potential regular value $\sigma$ to be an actual regular value, $Lk'(\sigma)$ should be a circle. It is known that this graph has only degree 2 vertices which implies that it is a collection of circles. All that is left is to check that there is only one such circle. Since we already calculated the lists of circles of $Lk'(\sigma)$, this takes $O(1)$ time to verify. One must similarly verify that for every other  potential value $Lk'$ has no circles, as the formulation of Theorem~\ref{GenLink} requires.

There is nothing more to check for potential RB and branch values - A multigraph with no circles whose only vertices of degree $\neq 2$ are two degree-1 vertices must be a line, and we already verified that that this line is properly embedded in the disc $Lk(\sigma)$, since its intersection with the boundary circle of the line is equal to the set of degree-1 vertices. Similarly, a graph with only 1 degree-4 vertex and no circles must be an 8-graph.

For potential double values one must check that all the long edges begin in one of the degree 4 vertices and end in the other - that none of them are loops. Finding the beginning and ending vertices of a given long edge takes $O(1)$ time and, since there are only 4 such edges (due to vertex-degree constraints), this whole check takes $O(1)$ time. For potential DB values, one must check that every edge connects the degree-4 vertex with one of the degree-1 vertices. This similarly takes $O(1)$ time.

For a potential triple value one must check that every degree 4 vertex is connected via a long edge to 4 of the other vertices - that there are no loops and no two vertices that are connected via two or more long edges. This will again take $O(1)$ time, due to the fact that a potential triple value has a constant number of degree-4 vertices (6), and long edges (12). Since the inspection of every $\sigma$ takes only $O(1)$ time, it takes $O(n)$ time to inspect all $\sigma$s. 
%
%
\end{proof}

This concludes the verification that $(M,S)$ is a valid input. The various results of sections \ref{DTsec1}, \ref{DTsec2}, and \ref{DTsec3} - in particular Lemmas~\ref{PureComp}, \ref{St'Lem}, \ref{MisMan}, \ref{Surf2&1}, \ref{Deg}, and \ref{Surf0} - imply that:

\begin{thm} \label{ValidQuad}
One can check that a \gls{genericsurface} $(M,S)$ is valid in linearithmic $O(n \cdot \log(n))$ time.
\end{thm}

\section{Identifying the relevant parts of the surface} \label{idensec}

As we explained in section~\ref{PremSec}, the algorithm needs to identify some parts of the surface - the double arcs, the two intersecting surface strips at each arc, the three intersecting surface sheets at each triple value, and the three segments of double arc that pass each triple value. In this section we will explain how the algorithm does this, and study the complexity of this process.

1) Begin by identifying the 3 intersecting arc segments $TV_k^1$, $TV_k^2$ and $TV_k^3$ at each triple value $TV_k$. Observe the star of a triple value - Figure~\ref{fig:Figure 11}A depicts the star and link of a triple value. As one can see, each of the intersecting arc segments begins at one of the degree-4 vertices $v$ of $Lk'(TV_k)$, goes into $TV_k$ itself (via the unique 1-simplex that connects them), and then continues into the antipodal vertex $v'$ in $Lk'(TV_k)$ - the only vertex that $v$ is not connected to via a long edge.

\begin{defn} \label{SegEnd}
Name the degree-4 vertices of $Lk'(TV_k)$ $a_1^+$, $a_1^-$, $a_2^+$, $a_2^-$, $a_3^+$ and $a_3^-$ in such a way that every $a_l^+$ is antipodal to $a_l^-$. For every $k,l$, the arc segment $TV_k^l$ will be the path $a_l^-,TV_k,a_l^+$ when $a_l^{\pm}$ are taken from $Lk(TV_k)$.
\end{defn}

\begin{lem} \label{Iden1}
Defining these arc segments takes $O(n)$ time.
\end{lem}

\begin{proof}
Trivial.
\end{proof}

2) We move on to the 3 intersecting sheets at each $TV_k$. For every $1 \leq l_1 < l_2 \leq 3$, look at the path in $Lk(TV_k)$ that begins in $a_{l_1}^+$, continues into $a_{l_2}^+$ via the unique long edge that connects them, continues into $a_{l_1}^-$, into $a_{l_2}^-$ and returns into $a_{l_1}^+$. As Figure~\ref{fig:Figure 11}B depicts, this will be the boundary of one of the intersecting sheets at $TV_k$ - the one that contains $TV_k^{l_1}$ and $TV_k^{l_2}$. We called it $D_k^{\{l_1,l_2\}}$ in Definition~\ref{3Discs}. We consider finding these paths as identifying the sheets, since the simplices of the sheets are just the simplices of the path and the union of each such simplex with $TV_k$.

\begin{lem} \label{Iden2}
Identifying these paths takes $O(n)$ time.
\end{lem}

\begin{proof}
Using Lemma~\ref{LongEdgeComp}, we can identify the long edges of $Lk'(TV_k)$ in \\$O(\#Lk'(TV_k))$ time. There is a unique long edge $v_0,...,v_r$ that either begins in $a_1^+$ and ends in $a_2^+$ or the other way around - the way that Lemma~\ref{LongEdgeComp} identifies each arc chooses the direction of the arc arbitrarily. If the direction is wrong (from $a_2^+$ to $a_1^+$), we can reverse the direction in $O(r)$ time.

Next, we take the long edges from $a_2^+$ to $a_1^-$, from $a_1^-$ to $a_2^-$ and from $a_2^-$ to $a_1^+$ and concatenate them. We might have to reverse the order of each of them. Since, $a_2^+$ appears twice, once in the end of the first path and once in the beginning of the second path, delete one instance from the concatenate list. Do the same with $a_1^-$ and $a_2^-$. We now have the boundary of $D_k^{\{1,2\}}$ as a path in $Lk'$ (or just in $M$). We can find the boundaries of $D^{1,3}$ and $D^{2,3}$ similarly.

This process involved identifying the long edges in $O(\#Lk'(TV_k))$ time, possibly reversing any of them, which again will take at most $O(\#Lk'(TV_k))$ time, and concatenating, which takes $O(1)$ time. Thus, it takes $O(\#Lk'(TV_k))$ time in total for a single $TV_k$, and $O(n)$ time in total for all $TV_k$'s.
\end{proof}

3) We move on to identifying the double arcs. The union of all double arcs is the intersection set $X(S)$ - the sub-complex of $S$ that contains all the edges that are made of double values, and all the vertices that are double, triple, branch and DB values. We identified these while verifying the validity of the surface $(M,S)$ (Remarks~\ref{Save3} and \ref{save2}(1)). The degree-2 vertices of $X(S)$ are the double values, and so the long edges of $X(S)$ connect between different triple, branch and DB values of $S$.

A double arc is a path in $X(S)$. Similarly to a long edge, it goes from one triple, branch or DB value to another. If it enters a DB or branch value, it ends - these are degree 1 vertices so a path that enters one cannot continue. However, when it reaches a triple value, it crosses it via one of the 3 arc segments. The meaning of this is that if at some point the path contains the sequence of vertices $...,v_{r-2},v_{r-1},v_r,v_{r+1},v_{r+2},...$ and $v_r$ is the triple value $TV_k$, then the sequence $v_{r-1},TV_k,v_{r+1}$ must be one of the arc segments $TV_k^l$ that crosses $TV_k$. In particular, one of the vertices $v_{r-1},v_{r+1}$ is the $a_l^+$ from $Lk'(TV_k)$, and the other is the antipode $a_l^-$.

In particular, a double arc will be a path $v_0,...,v_s$ in $X(S)$ such that:

a) There are instances $0<r_1<...<r_{m-1}<s$ such that $v_{r_i}$ is equal to some triple value $TV_{k_i}$ and that there are $l_i=1,2,3$ such that the vertices $v_{r_i-1},v_{r_i+1}$ are the vertices $a_{l_i}^+$ and $a_{l_i}^-$ from $Lk'(TV_{k_i})$.

b) If the arc is open (begin and ends in a degree-1 vertex), then each of the values $v_0$ and $v_s$ is a branch value or a DB value. If the arc is closed (closes into a loop), then $v_0=v_s$ is also a triple value $TV_{k_i}$, and $v_1,v_{s-1}$ are the vertices $a_{l_i}^+$ and $a_{l_i}^-$ from $Lk'(TV_{k_i})$.

c) In any case, the sequences $v_0,...,v_{r_1}$, $v_{r_{m-1}},...,v_s$, and $v_{r_i},...,v_{r_{i+1}}$ for every $i=1,...,m-2$, are long edges of $X(S)$ - they begin and end in a vertex of degree 6 or 1 and all the other vertices in them are of degree 2 (they are double values.)

In order to identify the double arcs, we create a list $DA_0,..,DA_{N-1}$ whose entries are the double arcs of $S$. In particular, each of them is a list of integers (the indices of vertices of $M$). This indexes the double arcs - the $j$th entry in the list, $DA_j$, is the $j$th double arc.

If a vertex $v_{r_i}$ in $DA_j$ is equal to the triple value $TV_{k_i}$, and $v_{r_i-1},v_{r_i+1}$ are the vertices $a_{l_i}^+$ and $a_{l_i}^-$ from $Lk'(TV_{k_i})$, then the double arc $DA_j$ contains the arc segment $TV_{k_i}^{l_i}$. This means that the index function $j(k_i,l_i)$ must be equal to $j$. While identifying the double arcs, we will define the index function as well.

\begin{lem} \label{ArcLin}
Identifying the double arcs and the index function takes $O(n^2)$ time.
\end{lem}

\begin{proof}
We begin by identifying the long edges of $X(S)$. According to Lemma~\ref{LongEdgeComp}, this takes $O(\#X(S))$ time, which is clearly less then $O(n)$. We also store a copy of the lists of all DB and branch values of $S$.

We define an empty list $DA_0$ - we will add vertices to it until it is a full double arc. We identify the open double arcs first, if there are are any. We pick a DB or branch value $v_0$ (if there are any) to be the beginning of the arc and search the list of long edges for the one that begins, or ends, in $v_0$. This takes $O(n)$ time. If this edge ends in $v_0$, we reverse it. It now has the form $v_0,...,v_{r_1}$. We concatenate this into $DA_0$, delete this long edge from the list of long edges, and delete $v_0$ from the copy of the list of branch or DB values.

$DA_0$ now ends in $v_{r_1-1},v_{r_1}$. We search for $v_{r_1}$ in the lists of DB, branch and triple values. It takes $O(n)$ time. If $v_{r_1}$ is a triple value, called $TV_k$, then $v_{r_1-1}$ is equal to one of the vertices $a_l^{\pm}$ of $Lk'(TV_k)$. Finding which one of them it is takes $O(1)$ time. For these $k$ and $l$, the $k,l$th arc segment is part of $DA_0$, so we set $j(k,l)=0$. We define $v_{r_1+1}$ to be the antipodal vertex $a_l^{\mp}$. We than search the list of long edges for the one that begins with $v_{r_1},v_{r_1+1}$, or ends with $v_{r_1+1},v_{r_1}$ - it must exist since $v_{r_1}$ is a vertex of degree 6 and is adjacent to $v_{r_1+1}$.

If the long edge ends in $v_{r_1+1},v_{r_1}$, we reverse it. It now has the form $v_{r_1},...,v_{r_2}$. Concatenate it into $DA_0$, which will then have the form $v_0,...,v_{r_1-1},$ $v_{r_1},v_{r_1+1},v_{r_1+2}...,v_{r_2}$. If $v_{r_2}$ is again a triple value $TV_k$, we repeat the process - we find the $l=1,2,3$ for which $v_{r_2-1}=a_l^{\pm}$ and set $j(k,l)=0$, set $v_{r_2+1}=a_l^{\mp}$, find the long edge that begins with $v_{r_2},v_{r_2+1}$ (we might have to reverse its order) and concatenate it into $DA_0$. Lastly, we delete this long edge from the list of long edges.

We continue in this way until some $v_{r_m}$ is not a triple value, in which case the end of the double arc has been reached. We then delete the long edge from the list of long edges and delete $v_{r_m}$, which is a DB or branch value, from the appropriate copy list. Next, we start a new double arc as an empty list $DA_1$ and repeat the process. This time, for every arc segment $TV_k^l$ we encounter, we set $j(k,l)=1$. We continue identifying open double arcs as long as there are new DB or branch values left (we deleted all of the ones we already used from the list). As soon as they are all done, any remaining long edges in $X(S)$ (if there are any) are parts of closed double arcs.

The next double arc we define, $DA_j$, will be a closed one. We begin by choosing one of the remaining long edges, $v_0,...,v_{r_1}$, and concatenate it into the currently empty list $DA_j$- now $v_{r_1}$ is definitely a triple value $TV_k$. As before, we find the $k$ and $l$ for which $v_{r_1-1}$ is $a_l^{\pm}$, set $v_{r_1+1}$ to be the antipodal vertex $a_l^{\mp}$, set $j(k,l)=j$, and delete the current long edge from the list. Then, we find the next long edge - the one that begins with $v_{r_1},v_{r_1+1}$ (we may have to reverse it).

It may be that $v_{r_1},v_{r_1+1}$ are equal to $v_0,v_1$- if so, then we have come full circle and $v_0,...,v_{r_1}$ is the entire double arc. Otherwise, we concatenate the new long edge into $DA_j$. Now $DA_j$ is equal to $v_0,...,v_{r_2}$. We continue in this way, adding new long edges $v_{r_i},...,v_{r_{i+1}}$ until $(v_{r_{i+1}},v_{r_{i+1}+1})=(v_0,v_1)$. As soon as this happens, we are done with $DA_j$. If there are any long edges left, we begin to construct the next closed double arc $DA_{j+1}$ using the same method. We continue in this way until we have used all the long edges.

In terms of complexity, we have searched for each triple, DB and branch value in either the list of long edges or the lists of DB, branch and triple values (or in all lists). We actually had to search for each triple value three times, once per each time a double arc crossed it. This takes $O(n)$ per vertex and $O(n^2)$ in total. No other action in this process takes as long. For instance, even if we had to reverse all of the long edges it would still take only $O(n)$ time, since the sum of their lengths is $O(n)$.
\end{proof}

4) Lastly, the algorithm must identify the two intersecting strips at each double arc. This includes naming them - calling one of them the 0 strip and the other the 1 strip as per Definition~\ref{01Strip}. This may not always be possible, as the surface may have a non-trivial closed double arc - an arc for which the two surface strips merge into one. The identifying algorithm will also check if and when this happens. If the surface strips are distinct, the algorithm will use this information (which strip is the 0/1 strip) to define the parameters $s(k,l)$ of Definition~\ref{sParam}. The guiding idea here is this:

\begin{rem} \label{WhyDes}
Technically, identifying and naming the surface strips of a double arc is just a way to indicate, at each small interval inside the arc, which of the two surfaces that intersect along this interval should be lifted higher than the other. Every sufficiently small segment of the arc will either be contained in a 1-simplex or will begin and end at two adjacent 1-simplices, at different sides of a 0-simplex.

1) For an interval that is contained in a 1-simplex $\sigma$, it is enough to look at the neighborhood $St(\sigma)$ and the surface $St'(\sigma)$ within it. The 4 triangles of $St'(\sigma)$ form 2 transversely intersecting surfaces, each made of 2 triangles, whose intersection is $\sigma$. In Figure~\ref{fig:Figure 14}A, one surface is colored light blue and the other is beige. As can be seen, 2 of the triangles form one of the surfaces iff they are on opposite sides of $\sigma$. In order to identify the surface strips, we simply need to indicate which of the triangles are parts of the 0 strip and which are parts of the 1 strip.

2) For an interval that crosses a 0-simplex, look at the star of said 0-simplex. Figure~\ref{fig:Figure 14}B depicts the neighborhood of a 0-simplex that is a double value, but the same idea applies for triple values. The ``fan of triangles" marked in pink must all come from the same surface strip. In particular, the two triangles in ends of the fan (blue) must come from the same surface strip. Each of them belong to the star $St'$ of one of the 1-simplices of $X(S)$ that meet at the 0-simplex (red). When we indicate which triangles at the star of each 1-simplex are in the 0/1 strip, we must do so in a ``continuous way" - such that the two triangles at the ends of each fan come from the same strip.

If we do this, we do not need to indicate which strip do the other (pink) triangles of the fan come from - they will come from the same strip as the end (blue) triangles. This means that, in order to identify and name the surface strips, it is enough to go over the stars $St'$ of each 1-simplex of $X(S)$ and indicate which triangles come from which strip, as long as we do so in a continuous way.
\end{rem}

\begin{figure}
\begin{center}
\includegraphics{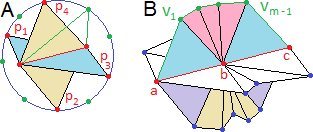}
\caption{The intersecting surface strips in the stars of simplices along the double arc} 
\label{fig:Figure 14}
\end{center}
\end{figure}

This leads us to define:

\begin{defn} \label{Des}
1) A designation on a 1-simplex $\sigma$ in $X(S)$ is a choice of number, either 0 or 1, for each of the 4 triangles in $St'(\sigma)$, such that two triangles have the same number iff they are on opposite sides of $\sigma$.

2) A designation of a double arc is a choice of designation for every 1-simplex in it. A designation can represent a choice of names for the surface strips, where all of the triangles that come from the 0 strip have the designation 0.

3) Let $a,b,c$ be a sub-sequence of consecutive vertices on the double arc $u_0,u_1,...$, $u_s$. $b$ must be either a double value or a triple value. If $b$ is a double value, then Figure~\ref{fig:Figure 14}B depicts $St(b)$. $St'(b)$ can be seen in the figure, composed of triangles. The red line is the local segment of the double arc made of the vertices $a,b,c$ and the edges $\{a,b\},\{b,c\}$. $a$ and $c$ are the two degree-4 vertices in $Lk'(b)$. In green, we indicate one of the 4 long edges in $Lk'(b)$ - a path of the form $a=v_0,v_1,...,v_r=c$. By definition $\{a,v_1,b\}$ and $\{b,v_{r-1},c\}$, indicated in blue in the figure, are 2-simplices in (respectively) $St'(\{a,b\})$ and $St'(\{b,c\})$. 

a) We define the 2-simplex ``following $\{a,v_1,b\}$" to be $\{b,v_{r-1},c\}$. This definition applies for each of the 4 long edges of $Lk'(b)$, and so each 2-simplex in $St'(\{a,b\})$ has a unique ``following 2-simplex" in $St'(\{b,c\})$.

b) We say that a designation on the arc is continuous at $a,b,c$ if every 2-simplex in $St'(\{a,b\})$ has the same designation (0 or 1) as its following 2-simplex in $St'(\{b,c\})$.

4) There are similar definitions when $b$ is a triple value. In this case, there is a $k=0,...,K-1$ for which $b=TV_k$ and there is an $l=1,2,3$ such that $a$ and $c$ are the antipodal vertices $a_l^+$ and $a_l^-$ in $Lk'(TV_k)$. Instead of looking at all of $St'(b)$, we focus on the union of the two surface sheets that contain $TV_k^l$ (for instance, if $l=1$ it will be the union of $D_k^{\{1,2\}}$ and $D_k^{\{1,3\}}$). This union will again look like Figure~\ref{fig:Figure 14}B. Its boundary, the union of the boundaries of the two discs, is again a graph with two degree-4 vertices, $a$ and $c$, and 4 long edges between them.

Other than the fact that we use this graph instead of all $Lk'(b)$, the definitions for ``following 2-simplex" and ``continuous designation" at $a,b,c$ are identical to those given in item 3 for double values.

5) If the arc is closed, then $u_0=u_s$ is a triple value too. One can make the same definitions as in item 4 for $a=u_{s-1}$, $b=u_s=u_0$ and $c=u_1$.

6) A designation on a double arc $u_0,u_1,...,u_s$ is said to be continuous if it is continuous of every triple $u_{i-1},u_i,u_{i+1}$ ($i=1,...,{s-1}$) and, if the arc is closed, for $u_{s-1},u_s=u_0,u_1$.
\end{defn}

\begin{rem} \label{DesRem}
1) As per Remark~\ref{WhyDes}, a designation represents a naming of the surface strips iff it is continuous.

2) The correspondence that sends a 2-simplex in $St'(\{a,b\})$ to its following 2-simplex in $St'(\{b,c\})$ is clearly 1-1.

3) Additionally, as one can deduce from the blue and purple triangles in Figure~\ref{fig:Figure 14}B, if two triangles in $St'(\{a,b\})$ are on opposite sides of $\{a,b\}$, then their following 2-simplices are on opposite sides of $St'(\{b,c\})$. This implies that if a single 2-simplex in $St'(\{a,b\})$ has the same designation as its following 2-simplex in $St'(\{b,c\})$, then the same is true for all simplices in $St'(\{a,b\})$ and the designation is continuous in $a,b,c$.

4) The parameters $s(k,l)$ can be deduced from the designations. For instance, if a double arc crosses a triple value $TV_k$ through the arc segment $TV_k^1$, then a continuous designation will give the same designation to all the four 2-simplices that come from the sheet $D_k^{1,2}$ (two from $St'(\{a,b\})$ and two from $St'(\{b,c\})$). This designation, 0 or 1, will be the surface strip that $D_k^{1,2}$ belongs to.

As per Definition~\ref{sParam}, if $D_k^{1,2}$ is in the $1$ strip then the parameter $s(k,1)$ should be equal to $0$. Otherwise, $D_k^{1,2}$ is in the $0$ strip and the parameter $s(k,1)$ should be equal to $1$. Similar indications can tell us the values of $s(k,2)$ and $s(k,3)$. This can be used to calculate the parameters $s(k,l)$ while we identify and name the surface strips.
\end{rem}

\begin{lem} \label{DesLem}
1) Given a double arc $DA_j=u_0,...,u_q$, one can identify and name the surface strips of $DA_j$, calculate the values of the parameters $s(k,l)$ for every arc segment $TV_k^l$ for which $j(k,l)=j$, and, in case $DA_j$ is a closed double arc, check if it is non-trivial. This can be done in at most $\sum_{i=1}^q O(\#Lk'(u_i)) + O(\#Lk'(\{u_0,u_1\}))$ time.

2) One can identify and name the surface strips of every double arc, check if any of the closed arcs are non-trivial, and calculate the values of the parameters $S(k,l)$ for every $k$ and $l$ in $O(n)$ time.
\end{lem}

\begin{proof}
(1 $\Rightarrow$ 2): Each 1-simplex in $X(S)$ appears exactly once, as the 1-simplex $\{u_{i-1},u_i\}$ of some double arc. In particular, the sum of the expressions $\\ O(\#Lk'(\{u_{i-1},u_i\}))$ from all double arcs is smaller than the sum $\sum O(\#Lk'(\sigma))$ that goes over every 1-simplex in $X(S)$.

Each 0-simplex in $X(S)$ appears at most 3 times as a 0-simplex $u_i$, $i>0$, in some double arc (since triple values are crossed by double arcs 3 times). We excluded $i=0$, since in closed arcs $v_0=v_s$, and we do not want to count this vertex twice. It follows that the sum of the expressions $\sum_{i=1}^q O(\#Lk'(u_i))$ from all double arcs is smaller than 3 times the sum $\sum O(\#Lk'(\sigma))$ that goes over every 0-simplex in $X(S)$.

Together, they are smaller than 3 times the sum $\sum O(\#Lk'(\sigma))$, going over every simplex $\sigma$ in $S$, which is bounded by $O(n)$ per Remark~\ref{AllStIsn}(1).

(1): The linear time algorithm referred to in (1) has 3 steps:

a) Defining a designation on the 1-simplex $\{u_0,u_1\}$: an algorithm that does this must first find which of the four 2-simplices in $St'(\{u_0,u_1\})$ are on opposite sides of the 1-simplex. $Lk(\{u_0,u_1\})$ is a circle. Using the algorithm of Lemma~\ref{LongEdgeComp} on this circle orders its vertices - it produces a list $v_0,v_1,...,v_r=v_0$ such that each $v_i$ is adjacent to $v_{i+1}$ in the circle. This takes $O(\#Lk(\{u_0,u_1\}))$ time. 

We go over the circle in order, from $v_0$ to $v_{s-1}$, and write down which vertices are in $Lk'(u_0 \cup u_1)$. $Lk'$ is a set of four points - the ``other vertex" of each of the four 2-simplices in $St'(u_0 \cup u_1)$. It follows that checking if a vertex $v_i$ is in $Lk'$ takes $O(1)$ time, and doing so for very vertex again takes $O(\#Lk(\{u_0,u_1\}))$ time. This check involves defining variables $p_1,...,p_4$ and setting $p_1=v_i$ the first time $v_i$ is in $Lk'$, setting $p_2=v_i$ the second time this happens (it will be a different $i$ by then), and so on.

The vertices $p_1$ and $p_3$ will then be on ``opposite sides" of the circle $Lk(\{u_0,u_1\})$ and, as Figure~\ref{fig:Figure 14}A depicts, the 2-simplices $\{u_0,u_1,p_1\}$ and $\{u_0,u_1,p_3\}$ will be on opposite sides of $\{u_0,u_1\}$. The same holds for $p_2$ and $p_4$. Define a designation $d$ of $St'(\{u_0,u_1\})$ by setting $d(\{u_0,u_1,p_1\})=d(\{u_0,u_1,p_3\})=0$ and $d(\{u_0,u_1,p_2\})=d(\{u_0,u_1,p_4\})=1$. This takes $O(1)$ time, and so step (a) takes $O(\#Lk(\{u_0,u_1\}))$ time in total.

b) Continuing this designation in a continuous fashion to every 1-simplex of the form $\{u_i,u_{i+1}\}$, and defining the parameters $s(k,l)$ meanwhile:

Assume via induction that you already defined the designation $d$ on every every 1-simplex of the form $\{u_i,u_{i+1}\}$ for $i<m$, and that it is continuous for every triple $u_{i-1},u_i,u_{i+1}$, $i=1,...,m-1$.  You need to define $d$ on $\{u_m,u_{m+1}\}$ in such a way that it will be continuous on $u_{m-1},u_m,u_{m+1}$.

The first thing to do is to check if $u_m$ is a double value or a triple value. As per Remark~\ref{save2}(5) this takes $O(1)$ time. If $u_m$ is a double value, look at the graph $Lk'(u_m)$ and calculate its long edges. This takes $O(\#Lk'(u_m))$ time. For every long edge $u_{m-1}=v_0,v_1,...,v_r=u_{m+1}$, set $d(\{u_m,u_{m+1},v_{r-1}\})$ to be equal to $d(\{u_{m-1},u_m,v_1\})$ (which was already defined in the previous step of the induction). This takes $O(1)$ time.

As per Remark~\ref{DesRem}(3), this will give 2-simplices in $St'(\{u_m,u_{m+1}\})$ that are on different sides of $\{u_m,u_{m+1}\}$ the same $d$ value, and so $d$ really is a designation on the 1-simplex $\{u_m,u_{m+1}\}$. The designation $d$ is also continuous at $u_{m-i},u_m,u_{m+1}$.

In case $u_m$ is a triple value, it takes $O(1)$ time to find the $k$ for which $u_m=TV_k$. It also takes $O(1)$ time to see which of the 3 arc segments $TV_k^l$ (identified in the beginning of this chapter) contains $u_{m-1}$. Look at the union of the boundaries of the two sheets at $TV_k$ that contains this arc segment (also identified earlier in this chapter) and compute their union. This is a sub-complex of $Lk'(u_m)$, and so calculating the union and finding its long edges takes $O(\#Lk'(u_m))$ time. Now you can proceed as in the case where $u_m$ is a double value.

In case $u_m$ is a triple value, you can also compute the appropriate parameter $s(k,l)$ in $O(\#Lk(u_m))$ time. In order to do this, look at the boundary $v_0,...,v_s$ of one of the sheets $D_k^{\{l,l_2\}}$ that contains $TV_k^l$. It must contain $u_{m-1}$. Search for the $0 \leq i \leq s-1$ for which $v_i=u_{m-1}$. This takes $O(s)$ time. Since the boundary of the said disc is a sub-complex of $Lk'(u_m)$, this is bounded by $O(\#Lk(u_m))$.

Look at the 2-simplex $\{u_{m-1},u_m,v_{i+1}\}$. It is contained in $St'(\{u_{m-1},u_m\})$ and in the disc $D_k^{\{l,l_2\}}$. The disc $D_k^{\{l,l_2\}}$ is thus a part of the $d(\{u_{m-1},u_m,v_{i+1}\})$ surface strip of the arc. You can deduce the parameter $s(k,l)$ as per Remark~\ref{DesRem}(4), in $O(1)$ time.

%
In summation, step (b) takes $\sum_{i=1}^{q-1} O(\#Lk'(u_i))$ time. It produces a designation on the whole double arc that is continuous at $u_{i-1},u_i,u_{i+1}$ for every $i=1,...,q-1$ (recall that the double arc is the sequence $u_0,...,u_q$), and computes the parameter $s(k,l)$ for almost every arc segment $TV_k^l$ on the double arc. The only exception is that for closed arcs, $u_q=u_0$ is a triple value and the sequences $u_{q-1},u_q,u_1$ is also an arc segment.

c) If the arc is open, then you are done - the designation is continuous by definition, and thus identifies and names the surface strips, and you have calculated all of the parameters $s(k,l)$ of this arc. If the arc is closed, you must check if it is trivial. Observe the sequence $u_{q-1},u_q,u_1$. $u_q=u_0$ is a triple value, and the sequence is some arc segment $TV-k^l$. As in (b), calculate the long edges of the union of the boundaries of the two appropriate sheets at $TV_k$. This takes $O(\#Lk'(u_q))$ time. Look at a long edge $u_{q-1}=v_0,v_1,...,v_s=u_1$.

Compare the designations of the 2-simplices $\{u_{q-1},u_q,v_1\}$ and $\{u_0 \cup u_1 \cup v_{s-1}\}$. If they have the same designation, then the designation $d$ is continuous at $u_{q-1},u_q,u_1$ (Remark~\ref{DesRem}(3)), and since it is also continuous everywhere else it identifies and names two distinct surface strips along the double arc. In particular, the arc is trivial. You can also calculate the parameter $s(k,l)$ in $O(\#Lk'(u_q))$ time using the method from (b).

If the said following 2-simplices have different designations, then the designation $d$ is not continuous at $u_{q-1},u_q,u_1$. This implies that there is no continuous designation on the whole double arc. If there was such a designation $d'$, then you could assume WLOG that it would agree with $d$ on the 2-simplices of $St'(u_0 \cap u_1)$ - otherwise you could change the $d'$-value of every 2-simplex in the arc and still have a continuous designation. Because both $d$ and $d'$ are continuous at every triple $u_{i-1},u_i,u_{i+1}$, then, by induction, they must agree on the the 2-simplices of every $St'(\{u_i,u_{i+1}\})$ ($i=0,..,s-1$). It follows that $d'$ is equal to $d$ and cannot be continuous at $u_{q-1},u_q,u_1$.

Since the double arc has no continuous designation, there is no way to identify and name the two surface strips. This implies that they merge into one strip, meaning that the arc is non-trivial. There is no reason to calculate $s(k,l)$ in this case, since having a non-trivial arc implies the surface is non-liftable.

Step (c) takes $O(\#Lk'(u_q))$ time in any case. Summing the runtime of the 3 steps shows that the algorithm takes $\sum_{i=1}^q O(\#Lk'(u_i)) + O(\#Lk'(u_0 \cup u_1))$ time in total.
\end{proof}

Lemmas~\ref{Iden1}, \ref{Iden2}, \ref{ArcLin}, and \ref{DesRem} combine into the following:

\begin{thm} \label{IdenQuad}
Identifying the relevant parts of the surface takes quadratic $O(n^2)$ time. 
\end{thm}
%
%

\section{A certificate for the lifting problem} \label{NPsec}
As was explained in section~\ref{CompOfLift}, in order to prove Theorem~\ref{Thm2} (the lifting problem is NP), we need to define a type of certificate that represents a \gls{liftingAt} of the \gls{genericsurface} $(M,S)$, and to devise a polynomial time ``certificate verifying algorithm" that checks whether this is really a legitimate lifting of the surface.

%

Recall that the $S$-star $St'(\sigma)$ of a 1-simplex $\sigma$ in a double arc $DA_j$ contains four 2-simplices. Two of these are parts of the 0 strip of $DA_j$ and the other two are parts of the 1 strip. A \gls{liftingAt} of the surface will make one of the strips into ``the higher strip" and in particular will make two of these 2-simplices ``high" and the other two ``low".

\begin{defn} \label{Cer}
A certificate for the lifting problem of a \gls{genericsurface} $(M,S)$ contains the following information:

a) A binary number $Triv$ that contains the value $0$ if the surface has a non-trivial closed double arc, and $1$ otherwise.

b) If $Triv=1$, it contains additional information reminiscent of defintion \ref{Des}. For every 1-simplex $\sigma$ in $X(S)$ it contains a choice of value, either ``H" or ``L", such that two 2-simplices have the value ``H" and the other two have the value ``L". This represents the \gls{liftingAt} in which the 2-simplices marked ``H" belong to the higher surface strip.
%
\end{defn}
%

\begin{lem} \label{CerIsLeg}
Given a \gls{genericsurface} $(M,S)$ and a certificate, there is a quadratic ($O(n^2)$) time algorithm that verifies that the certificate defines a legitimate lifting.
\end{lem}

\begin{rem}
Lemma~\ref{CerIsLeg}, along with the explanation given after Theorem~\ref{Thm2}, imply Theorem~\ref{Thm2}.
\end{rem}

\begin{proof}
Begin by running all of the algorithms of subsections \ref{DTsec1}-\ref{idensec} on $(M,S)$ - verifying that it is a valid \gls{genericsurface}, and identifying all of its relevant parts. This will take $O(n^2)$ time. In particular, it will discover if $S$ has a non-trivial double arc, and compare this information with the value of $triv$.

The algorithm will also compose a list of all the 1-simplices of $X(S)$. The certificate also contains a list of that should be the 1-simplices of $X(S)$. First, make sure that the latter list has at most $6n$ entries - $M$ cannot contain more 1-simplices than that. Next, compare the two lists. This takes $O(n \cdot \log(n))$ time. The lists must contain the exact same elements in order for the certificate to be legitimate.

Next, for every 1-simplex $\sigma=(v,u)$ in this list, the certificate contains a list of what should be the four 2-simplices in $St'(\sigma)$. The algorithm calculates this as well. You must compare the list of 2-simplices compiled by the algorithm with the one provided by the certification.

If the certificate passed all the checks so far, define a vector $(x_0,...,x_{N-1})$ of binary variables that will be used to represent the \gls{liftingAt} as per Definition~\ref{01Strip}.

Now, for every $j=0,...,N-1$, observe the double arc $DA_j=v_0,v_1,...,v_s$. In order to identify the surface strips, the algorithm produces continuous designations on every double arc (see Definition~\ref{Des}). Pick one of the two higher 2-simplices in $St'((v_0,v_1))$, the ones that the certificate gave the value ``H", and check what value the designation gave it. If it is $1$, then the 1-strip must be the higher strip according to this certificate. To represent this, set $x_j=1$. Otherwise, the 0 strip must be the higher strip, so set $x_j=0$.

The certificate must be consistent - if $x_j=1$, then every 2-simplex that belongs to the 1-strip along $DA_j$ (= that the designation of $DA_j$ gives the value $1$) must belong to the higher strip (= must be given the value ``H" by the certificate). Additionally, every 2-simplex that belongs to the 0 strip along $DA_j$ must belong to the lower strip. If $x_j=0$, it is the other way around. For every $t=1,...,s$, go over the four 2-simplices in $St'((v_{t-1},v_t))$ and verify that this occurs. Collectively, this check takes $O(n)$ time for all double arcs.

Now that you know that the certificate describes a real \gls{liftingAt}, and you have encoded this \gls{liftingAt} via a vector $(x_0,...,x_{N-1})$, check if the \gls{liftingAt} is legitimate by placing the values of the $x_j$s in the lifting formula and see if they satisfy it. This takes $O(K) \leq o(n)$ time.
%
%
\end{proof}

\chapter{The Main Theorem} \label{SecMain}

The remainder of the thesis revolves around proving that the lifting problem is NP-hard. In this short chapter we will explain our strategy for proving this. In section~\ref{Sym3Sec}, we proved that the proper symmetric 3-sat problem is NP-complete. We can thus prove the the lifting problem is NP-complete by reducing the proper symmetric 3-sat problem to the lifting problem in polynomial time. In particular, we would like to match every formula of this sort with a \gls{genericsurface} that ``realizes" it - has the given formula as its lifting formula. There are two points we must sharpen in this regard.

Firstly, as explained in section~\ref{PremSec}, in order to define the lifting formula of a \gls{genericsurface}, one must choose:

a) The order of the triple values, from $TV_0$ to $TV_{K-1}$ - they affect the order of the clauses of the formula.

b) The order of the intersecting surface strips at any triple value - which of them is $TV_k^1$, $TV_k^2$ and $TV_k^3$. This affects the order of the literals in any clause.

c) The order of the double arcs, from $DA_0$ to $DA_{N-1}$ - this determines the index of every variable. For instance, if the names of the arcs $DA_2$ and $DA_5$ are switched, then the names of the variables will switch with them - every instance of $x_2$ in the formula will be replaced with $x_5$ and vice versa. In particular, the order of the double arcs affects the index function $j(k,l)$.

d) The choice of which of the two intersecting surface strips at each double arc will be the 0-strip and which will be the 1-strip. This choice affects the values of the parameters $s(k,l)$.

This all leads to the following definition:

\begin{defn} \label{Realize}
A \gls{genericsurface} $(M,S)$ realizes a \gls{symmetric} $F$ if, for some choice of indexing for the double arcs, triple values, arc segments and surface strips at every double arc, the resulting lifting formula will be equal to $F$ up to a change in the order of the clauses.
\end{defn}

We will usually construct a surface that does not realize the given formula per se, but realizes an almost equal formula.

\begin{defn} \label{AlmostRealize}
1) Two symmetric 3-sat formulas are said to be ``almost equal", if one can be changed into the other by reordering the clauses and the literals within each clause, and add or remove the opposite pairs of clauses of the form $(x_j \vee x_j \vee \neg x_j)$ and $(x_j \vee \neg x_j \vee \neg x_j)$ for different variables $j$.

2) A \gls{genericsurface} $(M,S)$ almost realizes a \gls{symmetric} $F$, if it realizes an almost equal formula to $F$.
\end{defn}

Two almost equal formulas are equivalent, since both the clauses $(x_j \vee x_j \vee \neg x_j)$ and $(x_j \vee \neg x_j \vee \neg x_j)$ are tautologies. If a surface $(M,S)$ almost realizes a formula $F$, then, should you choose the correct indexing for the double arcs and surface strips, the lifting formula $G$ will be almost equal to $F$. This implies that $(M,S)$ is liftable iff $G$ is solvable iff $F$ is solvable. In particular, it is possible to prove that the lifting problem is NP-hard by devising a polynomial time algorithm that receives a \gls{sym&pro} and produces a \gls{genericsurface} $(M,S)$ that almost realizes the formula.
%

We will prove a slightly stronger theorem:

\begin{thm} \label{Thm3}
There is a polynomial time algorithm that receives a proper symmetric $3$-sat formula, and produces a \gls{genericsurface} $(M,S)$ that almost realizes the formula. Additionally, the $3$-manifold $M$ will always be homeomorphic to the closed ball $D^3$, $S$ will contain no branch values, and the underlying surface will always be closed and orientable. 
\end{thm}

This proves that a limited variant of the lifting problem, the lifting problem of orientable closed immersed surfaces in $D^3$, is still NP-hard (and thus NP-complete). One can also replace $D^3$ with any compact 3-manifold $X$, as we will show at the very end of section~\ref{SBsec}.

The main difficulty in creating such an algorithm is finding a way to manufacture a closed \gls{genericsurface} with a given lifting formula. The additional requirement that the surface must be orientable adds to the difficulty, but it also provides us with new tools to work with. For an oriented surface, there is a correlation between the lifting formula and the orientation on the surface that we will utilize to create surfaces with given formulas.

\chapter{The Lifting Formula of Oriented Surfaces} \label{SecGraph}

In this chapter, we will explain the connection between the orientation of an oriented \gls{genericsurface} and its lifting formula. In the following chapter, we will use this to create surfaces that (almost) realize given formulas, and thus to prove Theorem~\ref{Thm3}.

%
\section{Thrice-oriented surfaces and their 0 and 1 strips}

In order to define the lifting formula of a \gls{genericsurface}, one needs to choose, among other things, which of the two intersecting surface strips at each double arc is the 0-strip and which is the 1-strip. If the surface is oriented, resides inside an oriented 3-manifold, and all of its double arcs are oriented (have a chosen direction of progress), then there is a canonical way to choose the 0- and 1-strips so that they correspond to all these orientations. In this chapter, we will explain these canonical 0- and 1-strips.

Recall that a \gls{genericsurface} is a simplicial map $i:F \to M$ that complies with the demands of Definition~\ref{Generic}, though we usually use the term to refer to the image $S$ of such a map. When referring to a \gls{genericsurface} as oriented, we mean that the underlying surface $F$ is oriented. That being said, there is an alternative definition for an orientation on $S$ that does not necessitate dealing with $F$:

\begin{defn} \label{Ori}
1) An orientation on a \gls{genericsurface} $S \subseteq M$ is a choice of orientation for every triangle $\sigma$ in $S$. Such an orientation is said to be continuous at an 1-simplex $e$ of $S$ when:

If $e$ is made of regular values, then $St'(e)$ is a disc, the orientation is said to be continuous at $e$ if it restricts to a continuous orientation on the disc $St'(e)$. In particular, $St'(e)$ is made of two 2-simplices. Each of the triangles induces an orientation on its boundary and thus on $e$. These two induced orientations must disagree in order for the orientation of $S$ to be continuous at $e$.

If $e$ is made of double values, then $St(e)\cap S$ is made of 2 intersecting discs. The orientation is said to be continuous at $e$ if it restricts to a continuous orientation on the each of the two discs.

A continuous orientation on $S$ is an orientation that is continuous at every 1-simplex.

2) A \gls{genericsurface} that has a continuous orientation is said to be orientable. 
\end{defn}

\begin{rem}
Definition~\ref{Ori} generalizes the concept of orientation on an embedded surface $S$ in $M$.
\end{rem}

\begin{lem}
If a \gls{genericsurface} is treated as the image of a function $i:F \to M$, as it was before Remark~\ref{SurfaceSet}, then an orientation on $i(F)=S$ is an equivalent notion to an orientation of $F$. In particular, $S=i(F)$ is orientable iff the underlying surface $F$ is orientable.
\end{lem}

\begin{proof}
$i$ serves as a 1-1 correspondence between the triangles of $F$ and $S$. There is also a 1-1  correspondence between orientations on $F$ and orientations on $S$. For every 2-simplex $\{v_1,v_2,v_3\}$ in $F$, the trio $(v_1,v_2,v_3)$ agrees with the orientation on $F$ iff the trio $(i(v_1),i(v_2),i(v_3))$ agrees with the orientation on $S$.

The star $St(e)$ of every edge $e$ in $S$ is sent homeomorphically into either the star $St(i(e)) \cap S$ (if $i(e)$ is made of regular values), or one of the two intersecting surfaces in $St(i(e)) \cap S$ (if $i(e)$ is made of double values). It is clear that an orientation on $F$ is continuous on a 1-simplex $e$ of $F$ iff the corresponding orientation on $i(St(e))$ is continuous. It follows that an orientation on $F$ is continuous iff the corresponding orientation on $S$ is continuous.
\end{proof}

\begin{defn} \label{Trice}
A \gls{Thrice} \gls{genericsurface} is a \gls{genericsurface} $(M,S)$ with orientations on $M$, $S$, and every double arc of $S$.
\end{defn}

\begin{rem} \label{OriFig}
1) In \citep{Car&Sai1}, Carter and Saito encoded the different \gls{lifting}s of the surface as different orientations of the double arcs. We only need one, arbitrarily chosen, orientation on each arc. We use them to deduce what are the 0 and 1 strips at every double arc, and thus to deduce the lifting formula. Subsequently, the different \gls{lifting}s will be encoded via the solutions of the formula.

2) We would like to depict the orientations of the surface in illustrations. In all illustrations, we follow the convention that the orientation of the 3-manifold that contains the surface coincides with the usual right hand orientation.

3) The orientation of a double arc will usually be represented by an arrow on this arc, indicating the direction of progress.

4) We depict the orientation of the surface in the same way one depicts an orientation of an embedded surface in a 3-manifold. Let $p$ be a point on the embedded surface $S \subseteq M$. $v \in T_pM \setminus T_pS$ has two connected components, each containing the vectors that point toward one side of the surface. All vectors $v$ from one of these components will uphold the property that, for every pair of independent tangent vectors $v_1,v_2 \in T_pF$, the pair $(v_1,v_2)$ represents the orientation on $S$ iff the trio $(v_1,v_2,v)$ represents the orientation on $M$. This connected component is called ``the \gls{preferred} side of the surface at $p$" and vectors from it are said to be \gls{preferred}. Vectors from the other component do not uphold said property and are said to be not-\gls{preferred}.

In illustrations, one usually draws little arrows, originating from various points all over the surface, that represent \gls{preferred} vectors. Using the convention that the 3-manifold has the right-hand orientation, one knows when a trio of vectors $(v_1,v_2,v)$ represents the orientation on $M$, and using a \gls{preferred} vector $v$ one can tell when a pair $(v_1,v_2)$ represents the orientation on $S$. Since the orientation is continuous, the ``\gls{preferred} side" of the surface varies continuously as you move around the surface. This only means that all the little arrows in the illustration point towards the same side of the surface.

Our surface is technically triangulated and not smooth, but each triangle of the surface is linearly, and thus smoothly, embedded in the 3-manifold. We can use arrows to depict the orientation on each triangle, and it will always be clear that the orientation will be continuous in terms of Definition~\ref{Ori} - that arrows on triangles from both sides of any edge will point in the same direction.
\end{rem}

\begin{defn} \label{Ori01}
Given a \gls{Thrice} surface, the canonical way to define the 0 and 1 strip around a double arc is as follows:

Let $p$ be a double value on the arc. Figure~\ref{fig:Figure 15}A depicts the neighborhood of this value. Let $N$ be a vector originating in $p$ and pointing in the direction of progress. In particular, $N$ is tangent to both of the surface strips that intersect at the arc. In Figure~\ref{fig:Figure 15}A, $N$ is colored in red.

Choose two vectors $h_0$ and $h_1$ that originate from $p$, such that each $h_i$ is tangent to a different one of the two surfaces strips, but neither of them is tangent to the double arc itself. Make sure that for each $i=0,1$, the pair $(h_i,N)$ represents the orientation of the surface strip it is on, and that the trio $(h_0,h_1,N)$ represents the orientation of the surface strip. In Figure~\ref{fig:Figure 15}A, $h_0$ is orange and $h_1$ is green.

Among the two surface strips that intersect at the arc, the canonical 0-strip (resp. 1-strip) will be the surface strip that the vector $h_0$ (resp. $h_1$) is tangent to. Figure~\ref{fig:Figure 15}B depicts the 0-strip in green, and the 1-strip in orange.
\end{defn}

\begin{figure}
\begin{center}
\includegraphics{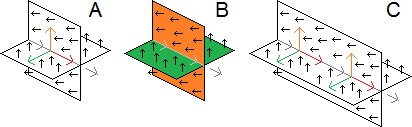}
\caption{The surface strips of a \gls{Thrice} surface} 
\label{fig:Figure 15}
\end{center}
\end{figure}

\begin{rem} \label{OriNonTriv}
1) This is clearly well defined - independent of the choice of $p$, $N$, $h_0$ and $h_1$. $N$ is clearly unique up to multiples with a positive constant, and each of $h_0$ and $h_1$ is unique up to a combination of multiplication and addition of multiplicities of $N$, which implies that $h_0$ will always be on the same surface strip regardless of how you choose it. As for the choice of $p$, Figure~\ref{fig:Figure 15}C shows that choosing a different point farther on the arc will result in the same 0 and 1 strips.

2) In particular, this implies that every closed double arc in a thrice-orientable \gls{genericsurface} is trivial (as there are distinct 0 and 1 surface strips for any such arc.)

3) This definition is given in the language of smooth surfaces but it can be easily modified for use with triangulated surfaces too.
%
\end{rem}

\section{Digital arrowed daisy graph} \label{DADGsec}
%
%

We are attempting to prove Theorem~\ref{Thm3} by devising an algorithm that realizes a \gls{symmetric} via a \gls{genericsurface}. Recall that in section~\ref{idensec}, we deduced the lifting formula of a surface from 3 ``invariants" of the surface. The first two were: a) the simplified graph structure of the intersection graph $X(S)$ of the surface and b) the 3 intersecting arc segments at each triple value. These provide indication as to the double arcs of the surface and which arc segments belonged to each arc - the information encoded in the index function $j(k,l)$.

The third invariant was more complicated. Recall that at each triple value there are 3 intersecting surfaces, and that each arc segment is the intersection of two of these surfaces. Each of the two surfaces is a part of one of the two long surface strips that intersects at the double arc, and the invariant is the information ``which strip contains which sheet" at each arc segment.

Intuitively, the first part of the algorithm should be to decide what the 3 above invariants of the surface should be. We will then create a \gls{genericsurface} with these properties, which will thus have to realize the given formula. Unfortunately, the third invariant is difficult to use - this applies to both deciding what it should be and creating a surface with this invariant. Luckily, \gls{Thrice} \gls{genericsurface}s have an alternative invariant that one can use to deduce the lifting formula - the ``\gls{preferred} direction" at each arc segment.

In this section, we will explain this alternative invariant and how it is encoded, along with the first two invariants, in an ``enriched graph structure". It will be similar to the ``arrowed daisy graphs" Li defined in \cite{Li1}. For algorithmic purposes, this ``enriched graph structure" will be represented by a formal data type that a computer can handle. We will thus call it a Digital Arrowed Daisy Graph, or \gls{DADG} for short.

In the next section, we will prove that the \gls{DADG} structure of a \gls{genericsurface} really does determine the surface's lifting formula.

\begin{defn} \label{Pref}
In addition to the simplified graph structure of the intersection graph (see Definition~\ref{LongEdge}), the \gls{DADG} structure of a \gls{Thrice} \gls{genericsurface} will encode the following information:

1) The intersection graph of a \gls{Thrice} \gls{genericsurface} is a \textbf{directed} multigraph - the direction of progress on any double arc induces a direction on the (long) edges that compose it. The \gls{DADG} will encode that directed multigraph structure.

2) We consider every edge of the \gls{DADG} as being made of 3 parts: two small segments at the ends of the edge, which we refer to as the ``beginning" and ``ending" of the edge, and the rest of the edge, which we refer to as the ``bulk" or ``length" of the edge. If the edge is parametrised as the interval $[0,1]$, where the direction of progress points from $0$ to $1$, then the beginning, ending, and bulk of the edge will respectively be the subsets $[0,\epsilon]$, $[1-\epsilon,1]$, and $[\epsilon,1-\epsilon]$ for some $\epsilon$.

We use the term ``\gls{EoE}" to refer to either the beginning or ending of some edge. We may refer to the set of some, or all, ``ends of edges" - the set containing the beginnings and endings of all edges. In order to indicate the directed multigraph structure of the intersection graph, it is sufficient to indicate which edges begin, and which edges end, at each vertex. This is equivalent to indicating which ends of edges lie on which vertices. The \gls{DADG} will contain this information.

3) When a (directed) double arc passes a triple value $TV_k$ via some arc segment $TV_k^l$, one edge of the arc ends at the triple value and the other begins on the other side of it. The arc segment is thus made of two of the ends of edges that lie on the triple value - the ending of the first edge and the beginning of the second edge. Figures~\ref{fig:Figure 16}A and B depict this. The ``edge-beginning" is dubbed ``b" and the ending is dubbed ``a". The ending is supposed to be hidden behind some surfaces, and so we depicted it as a dashed line. The direction of progress is indicated by the red arrows along the arc.

\begin{figure}
\begin{center}
\includegraphics{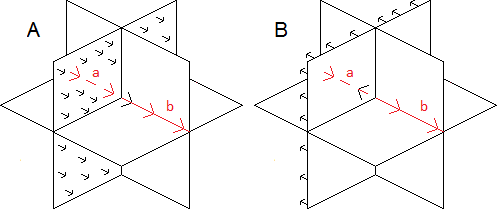}
\caption{The \gls{preferred} ``\gls{EoE}" of an arc segment}
\label{fig:Figure 16}
\end{center}
\end{figure}

For every triple value, the \gls{DADG} will indicate which two of the six ends of edges of the triple value compose each arc segment. In \cite{Li1}, Li described the same information in his ``daisy graphs". His long edges were not oriented, so there was no distinction between the beginning and ending of an edge, but he did indicate which two ends of edges composed each arc segment.

4) In addition to the arc segment $TV_k^l$ and its direction of progress, Figures~\ref{fig:Figure 16}A and B indicate the orientation (as per Remark~\ref{OriFig}(2-4)) of one of the 3 surface sheets that intersect at the triple value - the one that intersects $TV_k^l$ transversally (see Definition~\ref{3Discs}). 

The \gls{preferred} direction of the sheet points towards one of the two ends of edges that composed $TV_k^l$. We refer to this \gls{EoE} as the ``\gls{preferred}" \gls{EoE} of $TV_k^l$. In Figure~\ref{fig:Figure 16}A the \gls{preferred} \gls{EoE} is the beginning of the edge that begins in $TV_k^l$ and in Figure~\ref{fig:Figure 16}B it is the ending of the edge that ends in $TV_k^l$. For every triple value and arc segment, the \gls{DADG} will indicate which of the segment's two ends of edges is the \gls{preferred} one. In \cite{Li1}, Li conveyed the same information in his arrowed daisy graphs.
%
\end{defn}

\begin{rem} \label{IndDep}
In order to describe a graph via its vertices and edges, one must first name each the vertices. Naming the vertices in different ways will create graphs that are isomorphic, but technically not identical. For example, the graphs $(\{a,b\},\{\{a,b\}\})$ and $(\{o,3\},\{\{o,3\}\})$ are isomorphic but not identical. Conceptually, isomorphic graphs are graphs that would have been identical, had their vertices been similarly named.

The \gls{DADG} structure of a surface also depends on how one names the vertices - the DB, branch and triple values of the graph. We name them by indexing them. We will index each type of value separately, and give the name $BV_k$ / $DV_k$ / $TV_k$ to the $k$ branch / DB /triple value respectively. The \gls{DADG} of the surface will depend on the indexing of the values. For instance, switching the indexes of the $3$rd and $5$th vertices will change the \gls{DADG} of the surface. While we will not formally define morphisms of \gls{DADG}, it is clear that \gls{DADG}s produced by different indexings of the same surface should be considered isomorphic.

In order to describe the \gls{DADG} structure of a surface, we also need to index the (long) edges of the intersection graph, and the 3 arc segments at each triple value. Again, the \gls{DADG} will depend on this choice, but only up to isomorphism. As in the previous chapter, we will denote the $l$th arc segment at the $k$th triple value as $TV_k^l$. 

We will use the indexing of the edges to name the ends of edges of the graph. The beginning of the $r$th edge will be dubbed the ``the $(r,0)$th \gls{EoE}", and the ending of this edge will be dubbed ``the $(r,1)$th \gls{EoE}".

In order to encode the information of Definition~\ref{Pref}, the \gls{DADG} will indicate which ends of edges lie on each triple value, which two ends of edges form each arc segment, and which of these two is the \gls{preferred} \gls{EoE}.
\end{rem}

We need to define both the \gls{DADG} data type - what sort of information fields a \gls{DADG} contains, and the \gls{DADG} structure of a \gls{Thrice} \gls{genericsurface} - what values do these fields contain in a \gls{DADG} that describes the surface.

\begin{defn} \label{SurfDADG}
1) The \gls{DADG} structure of a \gls{Thrice} \gls{genericsurface} contains the following:

i) Non negative integers $C$, $B$, $D$ and $T$ that are respectively the numbers of disjoint circles, branch values, DB values and triple values of the surface.

ii) A list $BV$ of $B$ entries, whose $k$th entry is the index of the \gls{EoE} that lies on the $k$th branch value (there is only one such \gls{EoE} since branch values are degree-1 vertices of the intersection graph). The index of the \gls{EoE} is as per Remark~\ref{IndDep} - if the $r$th edge of the intersection graph begins (respectively, ends) at the $k$th branch value, then the $k$th entry in $BV$ is $(r,0)$ (resp. $(r,1)$). 

iii) A list $DV$ of $D$ entries that similarly indicate what \gls{EoE} lies on which DB value of the surface. 

iv) A list of $T$ entries of the following form: each entry $TV_k$ is made of three sub-entries $TV_k=(TV_k^1,TV_k^2,TV_k^3)$. Each of the sub-entries is the pair of the indexes $((r_{2l-1},s_{2l-1}),(r_{2l},s_{2l}))$ of the two ends of edges that compose the arc segments $TV_k^l$, and the second index $(r_{2l},s_{2l})$ is the index of the \gls{preferred} \gls{EoE} of this arc segment.

In particular, $TV_k=(((r_1,s_1),(r_2,s_2)),((r_3,s_3),(r_4,s_4)),((r_5,s_5),(r_6,s_6)))$ and the $(r_i,s_i)$-s are the indices of the 6 ends of edge that lie on the triple value.
\end{defn}

Note that, since every arc segment $TV_k^l$ is composed of the beginning of one edge and the ending of another, one of the binary variables $s_{2l-1}$ and $s_{2l}$ is equal to $0$ and the other equals $1$. Note also that, due to the vertex degree formula, the \gls{DADG} has $3T+\frac{1}{2}(B+D)$ edges. We use this to define a ``\gls{DADG} data type" in general - without deriving it from a \gls{genericsurface}. 

\begin{defn} \label{DADG}
1) A \gls{DADG} is a data type that contains the following information:

i) Non negative integers $C$, $B$, $D$ and $T$.

ii) A list $BV$ of $B$ pairs of the form $(r,s)$, where $r$ is an integer between $0$ and $3T+\frac{1}{2}(B+D)-1$, and $s$ is binary.

iii) A similar list $DV$ of $D$ such pairs. 

iv) A list of $T$ entries of the form $(((r_1,s_1),(r_2,s_2)),((r_3,s_3),(r_4,s_4)),((r_5,s_5)$, $(r_6,s_6)))$ where each $(r_i,s_i)$ is as in (ii), and for each $l=1,2,3$ one of the binary variables $s_{2l-1}$ and $s_{2l}$ is equal to $0$ and the other is equal to $1$.

For a \gls{DADG} to be valid, each of the $6T+B+D$ possible values of $(r,s)$ must appear exactly once. Due to the pigeon coop principle, it is enough to demand that it appear at least once or at most once.

2) A \gls{Thrice} \gls{genericsurface} realizes a \gls{DADG} if, for some choice of indexing for the values of the surface, the long edges of the intersection graph, and the 3 arc segment at each triple value, the \gls{DADG} of the surface (as in (1)) will be equal to the given \gls{DADG}.

A \gls{DADG} $G$ is said to be realizable in an orientable 3-manifold $M$ if there is a \gls{Thrice} \gls{genericsurface} in $M$ that realizes $G$.
\end{defn}

Definitions~\ref{SurfDADG} and \ref{DADG} are clearly compatible - the \gls{DADG} of a \gls{Thrice} surface is indeed a \gls{DADG}. As you will see in section~\ref{InfHom}, the other direction is also true, since every \gls{DADG} can be realized by a \gls{genericsurface} in some 3-manifold. See Theorem~\ref{ADGThm2}.

\begin{rem}
1) The ``size" of a \gls{DADG} is the total number of vertices it has - $D+B+T$. This linearly bounds the number of edges of the \gls{DADG} ($3T+\frac{1}{2}(B+D)$), and so it bounds the size of the \gls{DADG} in the meaning of Remark~\ref{Calculate}(2) - it works if you assume that there is a boundary on the size of all integers you may use. Otherwise, the size of the \gls{DADG} is $O((D+B+T) \cdot \log(D+B+T)+log(C))$.

2) We treat any \gls{DADG} as a directed multigraph with additional structure, and we use the same notation for a general \gls{DADG} as we would for a \gls{DADG} that represents a surface. A \gls{DADG} $G$ has one vertex per every entry in each of the 3 lists $BV$, $DV$ and $TV$. The vertex that matches the $k$th entry in $BV$ / $DV$ / $TV$ is called the $k$th branch / DB / triple value of $G$. Branch and triple values are degree-1 vertices and triple values are degree 6 vertices, and if the the pair $(r,0)$ (resp. $(r,1)$) is in the $k$th entry in the list $BV$ / $DV$ / $TV$, then the $r$th edge begins (resp. ends) in the $k$th branch / DB / triple value.

Additionally, we refer to the beginnings and endings of edges as ``ends of edges", and we refer to the beginning (resp. end) of the $r$th edge as the $(r,0)$ (resp. $(r,1)$) \gls{EoE}. We also refer to each of the 3 pairs of ends of edges in a triple value $TV_k$ as the arc segments of the triple values; and for each $k=0,...,T-1$ and $l=1,2,3$, we name the arc segment $((r_{2l-1},s_{2l-1}),(r_{2l},s_{2l}))$ in the triple value $TV_k$ $TV_k^l$. Lastly, we say that the latter \gls{EoE}, $(r_{2l},s_{2l})$, is the \gls{preferred} \gls{EoE} of $TV_k^l$.

3) We can depict the additional ``enhanced" graph structure in a diagram of a \gls{DADG}: it is a graph diagram where every edge has an (colored) arrow on it that indicates the direction of progress. It is not a planar graph - edges may go over/under one another. We only mark the indexes of vertices, edges, and arc segments when required. We indicate DB values with purple dots to distinguish them from branch values. We draw each triple value as the intersection of 3 lines - the arc segments, and the edges before and after the triple value have the same direction of progress.

We also indicate the \gls{preferred} direction on each arc segment with a small black arrow based at the triple value - there are thus 3 such arrows on each triple value. The \gls{EoE} that the arrow points towards - either the beginning of the edge that begins in this arc segment, or the \gls{EoE} that ends there - is the \gls{preferred} \gls{EoE} of the arc segment. For example, in Figure~\ref{fig:Figure 17}, the edge $1$ ends in some arc segment and the edge $2$ begins there. The ending of $1$ is the \gls{preferred} \gls{EoE} at this arc segment. 
\end{rem}

\begin{figure}
\begin{center}
\includegraphics{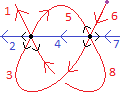}
\caption{A digram of a \gls{DADG}} 
\label{fig:Figure 17}
\end{center}
\end{figure}

\section{The graph lifting formula}

In this section, we will prove that the \gls{DADG} structure of a \gls{genericsurface} determines its lifting formula. Firstly, we will show how to deduce the index function $j(k,l)$ from the \gls{DADG}. Toward this end, one needs to identify the double arcs of the surface. 

\begin{defn} \label{consec}
We say that two ends of edges in a \gls{DADG} are ``\gls{consecutive}", if they form one of the arc segments of the \gls{DADG}, and we say that two edges are \gls{consecutive} if some end of one of them is \gls{consecutive} to some end of the other one.

We say that two edges $e$ and $f$ in a \gls{DADG} are on the same continuation (or just in continuation), if there are edges $e=e_0,e_1,...,e_q=f$ such that $e_i$ and $e_{i+1}$ are \gls{consecutive} for any $i$. ``Being in continuation" is clearly an equivalence relation on edges.
\end{defn}

As we explained in section~\ref{idensec}, when we identified the double arcs (in the proof of Lemma~\ref{ArcLin}), each double arc is made of several (long) edges of the intersection graph. As you progress along a double arc and cross a triple value via an arc segment, you move from one edge in the arc - the one that ends in the arc segment, into the next edge of the arc - the one that begins in the arc segment. Long edges are a part of the same double arc iff you can progress like this from one into the other, and each double arc is either made of such a collection of long edges, or is a long circle of the intersection graph.

In the language of \gls{DADG}s, each long edge of the intersection graph corresponds to an edge of the \gls{DADG}, and per Definition~\ref{Pref}(3), \ref{SurfDADG}(iv), and \ref{consec}, when you cross a triple value, you move between \gls{consecutive} edges. It follows that two edges of the \gls{DADG} come from the same double arc iff they are in continuation. In other words, each double arc of the surface corresponds to a unique equivalence class of the ``same continuation relation", or a disjoint circle of the \gls{DADG}.

In order to calculate the index function $j(k,l)$, you must first index the double arcs as $DA_0,...,DA_{N-1}$. $j(k,l)$ is the index $j$ of the double arc $DA_j$ that contains the arc segment $TV_k^l$. The arc segment is represented in the \gls{DADG} via the ends of edges $((r_{2l-1},s_{2l-1}),(r_{2l},s_{2l}))$ that compose it. It is the same long edge that $r_{2l-1}$ and $r_{2l}$ belong to - you only need to check which edge it is. 

One may use this to define the double arcs of a general \gls{DADG} (one that does not come from a \gls{genericsurface}), as well as the index function. 

\begin{defn} \label{DeduceInd}
1) A ``double arc" of a \gls{DADG} is either a disjoint circle or an equivalence class of the same continuation relation. As per the above, the double arcs of the \gls{DADG} structure of a \gls{genericsurface} correspond to the double arcs of the surface.

2) If you index the different double arcs as $DA_0,...,DA_{N-1}$, then you can determine the value of the index function $j(k,l)$. Observe the \gls{consecutive} ends of edges that compose the arc segment $TV_k^l=((r_{2l-1},s_{2l-1}),(r_{2l},s_{2l}))$. $j(k,l)$ is the index $j$ of the double arc which contains the edges $r_{2l-1}$ and $r_{2l}$.
\end{defn}

For instance, in Figure~\ref{fig:Figure 17} the following pairs of edges are \gls{consecutive} - 1 and 8, 2 and 4, 3 and 5, 3 and 6, 4 and 7, 5 and 8. It follows that $\{1,3,5,6,8\}$ and $\{2,4,7\}$ are equivalence classes of the same continuation relation, and are thus the double arcs of the \gls{DADG}. We coloured the former in red and the latter in blue. One can clearly see which arc segments belong to which double arcs. 

Unfortunately, the \gls{DADG} structure of a \gls{genericsurface} does not indicate the parameters $s(j,k)$ of the surface. However, it does indicate similar parameters which we call $s'(k,l)$. Recall that every arc segment $TV_k^l$ is made of 2 ends of edges, one ``beginning" and one ``ending", and that one of these ends of edges is \gls{preferred}. The parameter $s'(k,l)$ indicates which of them is \gls{preferred}.

\begin{defn} \label{s'}
Let $G$ be a \gls{DADG}. For every $k=0,...,T-1$ and $l=1,2,3$, look at the arc segment $TV_k^l=((r_{2l-1},s_{2l-1}),(r_{2l},s_{2l}))$. We define the parameter $s'(k,l)$ to be equal to $s_{2l-1}$. By referring to the binary parameters $s'(k,l)$ of a \gls{Thrice} \gls{genericsurface} (whose vertices and arcs segments are indexed), we mean the parameters $s'(k,l)$ of its \gls{DADG}.
\end{defn}

\begin{cor} \label{s'means}
For every $k,l$, $s'(k,l)=1$ iff the \gls{preferred} \gls{EoE} at the arc segment $TV_k^l$ is the beginning.
\end{cor}

\begin{proof}
Observe the arc segment $TV_k^l=((r_{2l-1},s_{2l-1}),(r_{2l},s_{2l}))$. By definition, $(r_{2l},s_{2l})$ is the \gls{preferred} \gls{EoE}. It is the ``beginning of edge" at $TV_k^l$ iff $(r_{2l-1},s_{2l-1})$ is the ``ending of edge" there iff $s'(k,l)=s_{2l-1}=1$.
\end{proof}

The parameters $s'(k,l)$ are not necessarily equal to the parameters $s(k,l)$, but they are close enough to be useful. Specifically:

\begin{lem} \label{ParamLem}
For every $k=0,...,T-1$ one of the following will hold: either $s'(k,1)=s(k,1)$, $s'(k,2)=s(k,2)$ and $s'(k,3)=s(k,3)$, or $s'(k,1)= \neg s(k,1)$, $s'(k,2)= \neg s(k,2)$ and $s'(k,3)= \neg s(k,3)$.
\end{lem}

\begin{proof}
Figures~\ref{fig:Figure 18}A and B depict two ways that the neighborhood of $TV_k$ may look like. In every figure, the arc segments $TV_k^1$, $TV_k^2$ and $TV_k^3$ are respectively coloured in red, blue and purple, and the orientations on the arcs point towards the reader. In both figures, the 3-manifold $M$ is assumed to be oriented, and that orientation coincides with the usual right-hand orientation of the manifold, as in Definition~\ref{OriFig}(2).

\begin{figure}
\begin{center}
\includegraphics{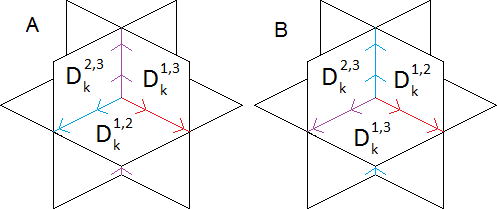}
\caption{Positions of the arcs' segments with regard to orientation} 
\label{fig:Figure 18}
\end{center}
\end{figure}

From a certain perspective, the neighborhood of every triple value will look like this - just spin it until the orientations on all 3 arc segments points towards you. The only difference is, that in Figure~\ref{fig:Figure 18}A the trio of vectors $(v_1,v_2,v_3)$, defined such that each $v_l$ points in the direction of the orientation on $TV_k^l$, agrees with the orientation of $M$, and in Figure~\ref{fig:Figure 18}B it does not.

These figures do not depict the orientations on the 3 intersecting surface sheets at $TV_k$. We added these in Figures~\ref{fig:Figure 19}. These figures depict only some of the possible configurations of orientation that can occur. There are 16 in total - $2^3=8$ possibilities for the orientation of the sheets times $2$ for the configurations in Figures~\ref{fig:Figure 19}A and B. We only depict 5 of these - that is sufficient for our purpose.

\begin{figure}
\begin{center}
\includegraphics{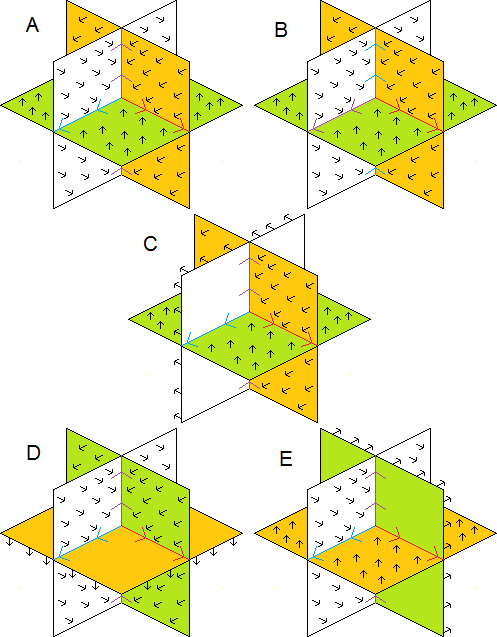}
\caption{The effect of orientation on $s(k,l)$ and $s'(k,l)$} 
\label{fig:Figure 19}
\end{center}
\end{figure}

The arc segment $TV_k^1$ is a part of the double arc $DA_{j(k,1)}$. One of the two surface sheets that intersect at this arc segment is a part of the 0 strip of this arc and the other is a part of the 1 strip. The orientations of the arc, surface and 3-manifolds decide which sheet is a part of which strip, as per Definition~\ref{Ori01}. In the Figures~\ref{fig:Figure 19}A-E, we colored these 0 and 1 strips in green and orange respectively.

In Figure~\ref{fig:Figure 19}A Definition~\ref{sParam} implies that $s(k,1)=1$ (since the surface sheet $D^{1,2}$ with the 1 strip) and Corollary~\ref{s'means} implies that $s'(k,1)=1$ (since the direction of progress on $TV_k^1$ agrees with the \gls{preferred} direction). Symmetry implies that $s(k,2)=s'(k,2)=s(k,3)=s'(k,3)=1$ as well. In particular $s(k,1) \leftrightarrow s'(k,1)=s(k,2) \leftrightarrow s'(k,2)=s(k,3) \leftrightarrow s'(k,3)=1$. Similarly, in Figure~\ref{fig:Figure 19}B $s'(k,1)=s'(k,2)=s'(k,3)=1$ but $s(k,1)=s(k,2)=s(k,3)=0$. And in particular $s(k,1) \leftrightarrow s'(k,1)=s(k,2) \leftrightarrow s'(k,2)=s(k,3) \leftrightarrow s'(k,3)=0$.

In these first two configurations, all 3 expressions ($s(k,1) \leftrightarrow s'(k,1)$, $s(k,2) \leftrightarrow s'(k,2)$ and $s(k,3) \leftrightarrow s'(k,3)$) are indeed equal. You can ``move" from those two configurations to each of the other 16 configurations by changing the orientations of one or more of the 3 intersecting surfaces. It will thus suffice to prove that each of these changes either changes the values of all 3 expressions, or does not change the value of any of them.

The difference between Figure~\ref{fig:Figure 19}A and Figure~\ref{fig:Figure 19}C demonstrates what happens when we change the orientation on the sheet $D_k^{2,3}$, that is perpendicular to $TV_k^1$. According to Corollary~\ref{s'means}, $s'(k,l)$ changes. On the other hand, $s(k,l)$ does not depend on the orientation of $D_k^{2,3}$ and so it remains as is.

The difference between Figure~\ref{fig:Figure 19}A and Figure~\ref{fig:Figure 19}D or \ref{fig:Figure 19}E demonstrates what happens when we change the orientation of any of the other sheets, the ones that contain $TV_k^1$. By Definition~\ref{Ori01}, this changes the names of the surface strips at the double arc $DA_{j(k,1)}$ that contains $TV_k^1$, and by Definition~\ref{sParam}, this changes the value of $s(k,l)$. On the other hand, $s'(k,l)$ does not depend on the naming of the surface strip, and so it remains the same.

In particular, changing the orientation on any of the 3 surfaces will change the value of $s(k,1) \leftrightarrow s'(k,1)$. For reasons of symmetry, the same holds for $s(k,2) \leftrightarrow s'(k,2)$ and $s(k,3) \leftrightarrow s'(k,3)$ - changing the orientation of any of the surfaces will change their value. The lemma follows.
\end{proof}

Lemma~\ref{ParamLem} implies

\begin{thm} \label{ParamThm}
The lifting formula of a \gls{Thrice} \gls{genericsurface} has the exact same clauses as the following formula, but they may appear in a slightly different order.
\end{thm}

\begin{multline} \label{GraphLiftForm}
\bigwedge_{k=0}^{K-1} (((x_{j(k,1)} \leftrightarrow s'(k,1)) \vee (x_{j(k,2)} \leftrightarrow s'(k,2)) \vee (x_{j(k,3)} \leftrightarrow s'(k,3))) \wedge \\
((x_{j(k,1)} \leftrightarrow \neg s'(k,1)) \vee (x_{j(k,2)} \leftrightarrow \neg s'(k,2)) \vee (x_{j(k,3)} \leftrightarrow \neg s'(k,3)))).
\end{multline}

\begin{proof}
Lemma~\ref{ParamLem} implies that for every $k=0,...,K-1$:

1) The 3-clause $F_k \equiv (x_{j(k,1)} \leftrightarrow s(k,1)) \vee (x_{j(k,2)} \leftrightarrow s(k,2)) \vee (x_{j(k,3)} \leftrightarrow s(k,3))$ is equal to one of the following two 3-clauses, either $(x_{j(k,1)} \leftrightarrow s'(k,1)) \vee (x_{j(k,2)} \leftrightarrow s'(k,2)) \vee (x_{j(k,3)} \leftrightarrow s'(k,3))$, or its \gls{mirror} $(x_{j(k,1)} \leftrightarrow \neg s'(k,1)) \vee (x_{j(k,2)} \leftrightarrow \neg s'(k,2)) \vee (x_{j(k,3)} \leftrightarrow \neg s'(k,3))$.

2) The \gls{mirror} $F_k \equiv (x_{j(k,1)} \leftrightarrow \neg s(k,1)) \vee (x_{j(k,2)} \leftrightarrow \neg s(k,2)) \vee (x_{j(k,3)} \leftrightarrow \neg s(k,3))$ is equal to the other one of the two clauses.

The theorem follows.
\end{proof}

\begin{defn} \label{GraphLift}
1) Formula (\ref{GraphLiftForm}) is called the graph lifting formula. It is defined for every \gls{DADG} whose double arcs are indexed. The graph lifting formula of a \gls{Thrice} \gls{genericsurface} whose long edges, branch, DB, and triple values and arc segments are indexed, is the the graph lifting formula of its \gls{DADG} structure.

2) A \gls{DADG} $G$ realizes a \gls{symmetric} $F$ if, for some indexing of its double arcs, the graph lifting formula of the \gls{DADG} is equal to $F$ up to a change in the order of the clauses. Similarly, $G$ ``almost realizes" $F$ if it realizes a formula $F'$ that is almost equal to $F$, in the sense of Definition~\ref{Realize}.
\end{defn}

\chapter{The Lifting Problem is NP-hard} \label{SecNPHrad}

This chapter is dedicated to proving Theorem~\ref{Thm3}, and thus that the lifting problem is NP-hard. Notice that Theorem~\ref{ParamThm} implies that a \gls{genericsurface} realizes/almost realizes a proper \gls{symmetric} iff it realizes a \gls{DADG} that realizes/almost realizes this 3-sat formula. This allows us to divide the algorithm into two parts.

In the first part, we create a \gls{DADG} that almost realizes the given formula. This is done in polynomial time, and the size of the \gls{DADG} is linearly bounded by the size of the given formula (if the formula has $2K$ clauses, the \gls{DADG} has $O(K)$ vertices). 

The second part receives the \gls{DADG} and creates a \gls{genericsurface} $(M,S)$ that realizes the \gls{DADG}, and thus the formula. The second part is done in polynomial time with regards to the size of the \gls{DADG}, and thus to the size of the formula. Additionally, as Theorem~\ref{Thm3} demands, $M$ will always be homeomorphic to the closed ball $D^3$, $S$ will be orientable and contain no branch values, and it will also be a closed \gls{genericsurface} - it will be contained in the interior of $M$ and will have no RB or DB values. It will also have no disjoint circles since those are unnecessary - they do not affect the lifting formula.

In order for the surface to uphold these properties, the \gls{DADG} will also contain no DB or branch values - only triple values. Furthermore, there are \gls{DADG}s that simply cannot be realized with a \gls{genericsurface}. Not only will all the \gls{DADG}s produced by the algorithm be realizable, but they will also uphold a very special property - they will all be what we refer to as ``\gls{Height1}" \gls{DADG}s. In the first section of this chapter, we will explain what \gls{DADG}s are realizable, what are \gls{Height1} \gls{DADG}s, and why they should be used.

\section{Gradable and height-1 DADGs}

\begin{defn} \label{Grading}
A grading of a \gls{DADG} is a choice of a number $g(e)$ (called ``the grade of $e$") for every edge $e$ of the \gls{DADG}, that upholds the following: 1) All the edges that contain \gls{preferred} ends of edges of $TV_k$ have the same grade. 2) All the edges that contain non-\gls{preferred} ends of edges of $TV_k$ have the same grade. 3) The grade of the \gls{preferred} edges at $TV_k$ is greater by 1 than the grade of the non-\gls{preferred} edges.

Using formal terminology, a grading is a function that assigns any number $r \in \{0,...,3K-1\}$ (representing the $r$th edge) an integer $g(r)$ such that, for any triple value $TV_k=(((r_1,s_1),(r_2,s_2)),((r_3,s_3),(r_4,s_4)),((r_5,s_5),(r_6,s_6)))$ $g(r_2)=g(r_4)=g(r_6)=g(r_1)+1=g(r_3)+1=g(r_5)+1$.
\end{defn}

In chapter \ref{BHrep}, we will prove that a \gls{DADG} is realizable via a \gls{genericsurface} in $D^3$, or any 3 manifold $M$ for which $H_1(M;\Z)$ is periodic (all of its elements have a finite order) iff it has a grading. We will also show that if $M$ is compact and $H_1(M;\Z)$ is infinite, then any \gls{DADG} can be realized there. See Theorems~\ref{ADGThm} and \ref{ADGThm2} for details.

Figure~\ref{fig:Figure 20}B depicts an example of a graded \gls{DADG}. On the other hand, the \gls{DADG} in Figure~\ref{fig:Figure 20}A is not \gls{gradable}. If it had a grading, then the red and green edges would have the same grade, since they both have non-\gls{preferred} ends at the upper triple value. In contradiction of this, the green edge has a \gls{preferred} end at the bottom triple value, and the red edge has a non-\gls{preferred} end there, and so they must have different gradings.

\begin{figure}
\begin{center}
\includegraphics{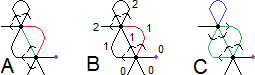}
\caption{A: a non-\gls{gradable} \gls{DADG}, A: a graded \gls{DADG} that is not \gls{Height1}, c: a \gls{Height1} \gls{DADG}} 
\label{fig:Figure 20}
\end{center}
\end{figure}

As per the above, the algorithm must realize every \gls{sym&pro} with a \gls{DADG} that is itself realizable via a \gls{genericsurface} in $D^3$. That means that the \gls{DADG} is \gls{gradable}. Furthermore, the \gls{DADG}s produced by the algorithm will uphold a property that is stronger then gradability - they will all be ``\gls{Height1}" \gls{DADG}s:

\begin{defn} \label{H1}
1) An \gls{EoE} in a \gls{DADG} is said to be \gls{preferred} (resp. non-\gls{preferred}) if it resides on a triple value (not a branch or DB value), and if it is the \gls{preferred} (resp. not the \gls{preferred}) \gls{EoE} of the arc segment that contains it.

2) An edge is said to be \gls{preferred}/non-\gls{preferred}/\gls{indecisive} iff both of its ends are \gls{preferred}/both are non-\gls{preferred}/it has one of each. If at least one of the edge's ends resides on a DB or branch value, we say that it is irrelevant.

3) A \gls{Height1} \gls{DADG} is a \gls{DADG} that has at least one triple value and no \gls{indecisive} edges - all of its edges are either \gls{preferred}, non-\gls{preferred} or irrelevant.
\end{defn}

\begin{lem} \label{H1L}
\Gls{Height1} \gls{DADG}s are \gls{gradable}.
\end{lem}

\begin{proof}
You can give a \gls{Height1} \gls{DADG} the following grading: give every \gls{preferred} edge the grade 1 and every non-\gls{preferred} or irrelevant edge the grade 0. For any triple value $TV_k=(((r_1,s_1)$, $(r_2,s_2)),((r_3,s_3),(r_4,s_4)),((r_5,s_5),(r_6,s_6)))$ in a \gls{Height1} \gls{DADG} $r_2$, $r_4$ and $r_6$ are \gls{preferred} edges while $r_1$, $r_3$ and $r_5$ are non-\gls{preferred} edges. It follows that $g(r_2)=g(r_4)=g(r_6)=1$, and $g(r_1)=g(r_3)=g(r_5)=0$, which is in accordance with the definition of a grading.
\end{proof}

For example, the \gls{DADG} in Figure~\ref{fig:Figure 20}B is not \gls{Height1}, since the red edge is \gls{indecisive}. The \gls{DADG} in Figure~\ref{fig:Figure 20}C is \gls{Height1} - the green edges are all \gls{preferred}, the blue edge is non \gls{preferred} and the black edges are all irrelevant. This \gls{DADG} can be graded by giving all the green edges the grade 1, and all the blue or black edges the grade 0, as the proof of Lemma~\ref{H1L} suggests.

The reason we use \gls{Height1} \gls{DADG}s, instead of general \gls{gradable} \gls{DADG}s, is as follows: while experimenting with \gls{DADG}s, we noticed that the ``height" of a \gls{DADG} is a good indicator as to how complex a surface that realizes the graph must be.

\begin{defn}
The ``height" of a connected \gls{DADG} (connected as a multigraph) is the difference between the maximum grade of any edge and the minimum grade of any edge, $max\{g(e)|e\} - min\{g(e)|e\}$. The height of a non-connected \gls{DADG} is the maximum height of all the components.
\end{defn}

\begin{rem}
One can easily prove that the height of a \gls{DADG} is independent of the grading, and that \gls{Height1} \gls{DADG}s are precisely those \gls{gradable} \gls{DADG}s whose height is 1. Indeed, the height of the grading we defined in the proof of Lemma~\ref{H1L} is clearly 1.
\end{rem}

In practice, ensuring that the first part of the algorithm produces only \gls{Height1} \gls{DADG}s makes the second part (the one that realizes the \gls{DADG} with a \gls{genericsurface}) considerably simpler and faster. Furthermore, this additional demand does not seem to make the first part of the algorithm noticeably longer or more complex.

\section{Realizing a symmetric 3-sat formula with a height-1 DADG}

In this section we prove the following:

\begin{thm} \label{Thm4}
There is a quadratic ($O(K^2)$) time algorithm whose input is a proper symmetric $3$-sat formula of length $K$ (it has $2K$ clauses) and whose output is a height-$1$ \gls{DADG}, with no disjoint circles, DB or branch values, which almost realizes said formula. The size of the \gls{DADG} is also polynomially bounded by the length of the input formula. Specifically, it will have no more than $4K$ triple values.
\end{thm}

The algorithm consists of 5 steps. We will explain each step, and then prove that this step takes at most quadratic time. This will be done in a separate lemma for each step - Lemma~\ref{Step1}, \ref{Step2}, \ref{Step3}, \ref{Step4}, and \ref{Step5}. Step 5 actually takes cubic time (see Lemma~\ref{Step5}.) When we explain step 3, we will show that the \gls{DADG} will have at most $4K$ triple values (see Remark~\ref{UpDown}(5)), and the theorem will follow.

1) The first step is verifying the validity of the input.

\begin{lem} \label{Step1}
This takes no more then quadratic ($O(K^2)$) time.
\end{lem}

\begin{proof}
The input is a string of letters and symbols. The algorithm checks that they are all variable names, brackets, and logical quantifiers, and that they are arranged as a 3-sat formula. The time this takes is linear in regard to the length of the string. If this string is a \gls{symmetric} with $K$ pairs of variables, then this length is $O(K)$, so this check takes $O(K)$ time.

The ``variable names" we mentioned are all of the form $x_i$, where $i$ is a non-negative integer. In order for the formula to be valid, it should have all the variables $x_0,...,x_{N-1}$ for some $N$. Finding $N$ (which is the maximum index of the variables plus 1) takes $O(K)$ time, as there are only $3K$ literals to check. Once you find it, you need to check that each of the variables $x_0,...,x_{N-1}$ appears at least once. This can be done in $O(K)$ time.

The algorithm then checks that the formula is proper: making sure that the literals in each clause are ordered (as per Definition~\ref{reduced}), and that no variable appears in a clause more than once takes $O(K)$ time. Checking that the clauses themselves are ordered also takes $O(K)$ time. If the clauses are ordered, then checking that no clause appears more than once takes $O(K)$ time as well.

Checking that the formula is symmetric - that every clause has a \gls{mirror} - takes no more than $O(K^2)$ time, even when done naively.
\end{proof}

2) The second step is determining the index functions $j(k,n)$ and the parameters $s'(k,l)$ that the \gls{DADG} ought to have. Allow us to explain what this means: assume you have a given \gls{symmetric}, and you want to create a \gls{DADG} that realizes the formula. At the moment, we require the \gls{DADG} to actually realize the formula, not just ``almost realize" it, in the sense of Definition~\ref{AlmostRealize} - the graph lifting formula of the \gls{DADG} ought to be equal to the given formula. The graph lifting formula of the \gls{DADG} is determined by the index function $j(k,l)$ and the parameters $s'(k,l)$ of the \gls{DADG}. It will have the form:

\begin{multline*}
\bigwedge_{k=0}^{K-1}(((x_{j(k,1)} \leftrightarrow s'(k,1)) \vee (x_{j(k,2)} \leftrightarrow s'(k,2)) \vee (x_{j(k,3)} \leftrightarrow s'(k,3))) \wedge \\ ((x_{j(k,1)} \leftrightarrow \neg s'(k,1)) \vee (x_{j(k,2)} \leftrightarrow \neg s'(k,2)) \vee (x_{j(k,3)} \leftrightarrow \neg s'(k,3)))).
\end{multline*}

Since the given formula will have the same clauses as the graph lifting formula, it seems that you can use the given formula to deduce what the index function $j(k,n)$ and the parameters $s'(k,l)$ of the \gls{DADG} ought to be. This will give (a part of) a blueprint for the \gls{DADG} - some idea about what this \gls{DADG} should look like. 

The problem is that the given formula does not tell you this explicitly. Firstly, because it is only a list of clauses made of literals. For instance, if the first ($0$th) clause is $x_0 \vee \neq x_1 \vee \neg x_2$ - you have to deduce that $j(0,1)=0$, $j(0,2)=1$ and $j(0,3)=2$ because of the index of the variables, and that $s'(0,1)=1$ and $s'(0,2)=s'(0,3)=0$ because the first literal has no negation symbol ($\neg$) and the other two do.

Secondly, and more importantly, the above formulation selects one ``prime clause" from every pair of mirror clauses, so that each pair consists of the prime clause $(x_{j(k,1)} \leftrightarrow s'(k,1)) \vee (x_{j(k,2)} \leftrightarrow s'(k,2)) \vee (x_{j(k,3)} \leftrightarrow s'(k,3))$, and its \gls{mirror} $(x_{j(k,1)} \leftrightarrow \neg s'(k,1)) \vee (x_{j(k,2)} \leftrightarrow \neg s'(k,2)) \vee (x_{j(k,3)} \leftrightarrow \neg s'(k,3))$. If for a certain $k$ the values of $s'(k,1)$, $s'(k,2)$ and $s'(k,3)$ were all changed, then the pairs will switch places - the ``prime clause" from before the change will be equal to the ``\gls{mirror}" after the change, and vice versa. Essentially, these two choices for the parameters describe the same formula.

The given formula (the input) does not indicate which clauses should be prime, and so you need to choose one prime clause from every pair of mirror clauses. For instance, the formula $(x_0 \vee x_1 \vee \neg x_2) \wedge (\neg x_0 \vee x_1 \vee x_2)$ is symmetric. It has one pair of mirror clauses, so $K=1$ and the index function and parameters are only defined for $k=0$. As before $j(0,1)=0$, $j(0,2)=1$ and $j(0,3)=2$. However, the values of the parameters are determined by the choice of prime clause. If the first clause is prime, then the parameters describe it - $s'(0,1)=1$ and $s'(0,2)=s'(0,3)=0$. But if the second clause is prime, then the parameters describe it and thus all have the opposite values - $s'(0,1)=0$ and $s'(0,2)=s'(0,3)=1$.

Determining the index function and the parameters includes the choice of a prime clause from every pair. As will soon be apparent, a wise choice of prime clause can shorten the runtime of the algorithm, and decrease the size of the \gls{DADG} it produces. We will, however, choose the prime clauses naively. The runtime of the algorithm and the size of the \gls{DADG} will remain polynomial in regard to $K$ and thus are sufficient for our purposes.

\begin{lem} \label{Step2}
Determining the index function and the parameters naively can be accomplished in quadratic ($O(K^2)$) time.
\end{lem}

\begin{proof}
First, you need to calculate $K$ itself to know how many parameters you need to define. $K$ is the number of clauses. It can be deduced immediately ($O(1)$ time) from the length of the input. 

Next, set $k=0$ and look at the first clause. This clause and its opposite will be the ``first pair of mirror clauses", and it (the first clause) will be the ``prime clause" of this pair. Set $j(0,1)$ to be index of the first literal. Set $s'(0,1)$ to equal 0 if this literal is negative (has $\neg$), and 1 if it is positive (has no $\neg$). Use the second and third literal to set $j(0,2)$, $s'(0,2)$, $j(0,3)$ and $s'(0,3)$ similarly. This takes $O(1)$ time.

Find the \gls{mirror} of the first clause in the formula - this takes $O(K)$ time. The parameters we currently have describe this first pair of mirror clauses in full, and therefore the said clauses are no longer needed. Delete both the first clause and its opposite from the formula. You now have a \gls{symmetric} with $K-1$ pairs of clauses. The first of the remaining clauses will be the ``prime clause" of the second pair of mirror clauses.

Increase $k$ by 1 and repeat - set $j(k,1)$ ($j(k,1)$ in general), $s'(k,1)$, $j(k,2)$, $s'(k,2)$, $j(k,3)$ and $s'(k,3)$ in the same way for the first (remaining) clause. Find its \gls{mirror} and delete them both. Increase $k$ by 1 and continue until there are no clauses left (this happens when $k=K$.) All this takes $O(K^2)$ time.
\end{proof}

3) Let us assume that you have a \gls{symmetric} and you have already determined the index function $j(k,l)$ and the parameters $s'(k,l)$. You would like to realize it with a \gls{DADG}. This \gls{DADG} ought to have one triple value per each pair of mirror clauses and the same parameters $s'(k,l)$ as the formula, and for some indexing of the double arcs the arc segment $TV_k^l$ should belong to the $j(k,l)$th arc (for every $k$ and $l$). However, it may not be possible to create a \gls{Height1} \gls{DADG} (or even a \gls{gradable} \gls{DADG}) with these properties.

It will be possible to create a \gls{Height1} \gls{DADG} that almost realizes the formula (see definition~\ref{GraphLift}(2). This \gls{DADG} will realize an ``enlarged" \gls{symmetric} that is almost equal to the input formula - it will have additional pairs of clauses. The third step of the algorithm is to determine how many pairs of clauses to add, and what will be the index function $j(k,l)$ and parameters $s'(k,l)$ for those new clauses.

%
%
The reason we must add new clauses to the input formula is as follows:

\begin{defn} \label{C}
1) Let $G$ be a \gls{DADG} whose double arcs are indexed as $DA_0,...$, $DA_{N-1}$. For every $j=0,...,N-1$, define $C_j^+$ / $C_j^-$ to respectively be the number arc segments on the $j$th double arc for which the parameter $s'(k,l)$ is equal $1$ / $0$.

2) Given a \gls{symmetric} for which each pair of mirror clauses has a chosen prime clause, define $C_j^+$ / $C_j^-$ similarly - the number of $(k,l)$s for which $j(k,l)=j$ and $s'(k,l)$ is equal $1$ / $0$.

3) In both cases, define also $E_j = 2max\{C_j^+,C_j^-\} - min\{C_j^+,C_j^-\}$.
\end{defn}

\begin{lem} \label{H1Sym}
If $G$ is a \gls{Height1} \gls{DADG}, then for every $j$, $C_j^+=C_j^-$. Additionally, this number will be equal to both the number of \gls{preferred} edges on the $j$th arc and the number of non-\gls{preferred} edges on this arc. 
\end{lem}

\begin{proof}
Let the arc segment $TV_k^l$ reside on the double arc $DA_j$ ($j(k,l)=j$). As per Corollary~\ref{s'means}, if $s'(k,l)=1$, then there is a non-\gls{preferred} edge that ends at $TV_k^l$ and a \gls{preferred} edge that begins at $TV_k^l$. Similarly, if $s'(k,l)=0$ then there is a \gls{preferred} edge that ends at $TV_k^l$ and a non-\gls{preferred} edge that begins at $TV_k^l$. In both cases, both edges belong to the double arc $DA_{j(k,l)}$.

This implies that, for every $j$, there is a 1-1 correspondence between the \gls{preferred} edges and the arc segments of the $j$th arc, sending every edge to the arc segment where it begins. This proves that the number of \gls{preferred} edges is equal to the number of arc segments with $s'(k,l)=1$. These are equal to the numbers of non-\gls{preferred} edges and arc segments with $s'(k,l)=0$, for similar reasons.
\end{proof}

A \gls{DADG} that actually realizes the input formula will clearly have the same $C_j^+$s and $C_j^-$s as the formula. If for any of the $j$s $C_j^+ \neq C_j^-$ then Lemma~\ref{UpDown} implies that this \gls{DADG} cannot be \gls{Height1}. Adding pairs of clauses to the formula can fix this.

Specifically, if $C_j^+ > C_j^-$, then adding a new pair of mirror clauses for which $j(k,1)=j(k,2)=j(k,3)=j$, $s'(k,1)=1$, and $s'(k,2)=s'(k,3)=0$ would increase $C_j^+$ by 1 and $C_j^-$ by 2 - decreasing the difference between them by 1. Adding $C_j^+ - C_j^-$ pairs like this would nullify this difference. Similarly, if $C_j^+ < C_j^-$ you can fix this by adding $C_j^- - C_j^+$ pairs of clauses for which $j(k,1)=j(k,2)=j(k,3)=j$, $s'(k,1)=s'(k,2)=1$ and $s'(k,3)=0$ would nullify this difference. In both cases, this translates to adding $|C_j^+ - C_j^-|$ triple values to the \gls{DADG} (for every $j=0,...,N-1$.)

After you added these new clauses, the enlarged, almost equal formula would uphold $C_j^+=C_j^-$ for every $j$. Such a formula can be actually realized by a \gls{Height1} \gls{DADG}, as we will show in parts 4 and 5.

\begin{rem} \label{UpDown}
1) Notice that in both cases, up to the order of the literals, you will add the same pair of clauses to the formula - $(x_j \vee \neg x_j \vee \neg x_j) \wedge (x_j \vee x_j \vee \neg x_j)$. The difference between the cases is which of the two clauses will be the prime clause. This implies that the graph lifting formula of the \gls{DADG} will be almost equal to the given formula (the input), and in fact this is what motivated the definition of almost equal formulas.

2) The variables $C_j^+$ / $C_j^-$ describe the number of arc-segments with $s'(k,l)=1$ / $0$ in the input formula. But how may variables of each kind will the extended, almost equal formula have? If $C_j^+ > C_j^-$, the extended formula will have $C_j^+ + (C_j^+ - C_j^-)= 2C_j^+ - C_j^-=E_j$ arc segments on the $j$th arc with $s'(k,l)=1$. Otherwise, for similar reasons, it will have $2C_j^- - C_j^+=E_j$ arc segments on the $j$th arc with $s'(k,l)=0$. In both cases, Lemma~\ref{H1Sym} implies that the numbers of arc segments with $s'(k,l)=1$, arc segments with $s'(k,l)=0$, \gls{preferred} edges on the $j$th double arc, and non-\gls{preferred} edges on the $j$th double arc, will all be equal to $E_j$.

3) We defined and will calculate $E_j$ out of convenience, so as to not write a cumbersome expression like $2max\{C_j^+,C_j^-\} - min\{C_j^+,C_j^-\}$ many times in our calculations.

4) The total number of arc segments in the \gls{DADG} will be $\sum_{j=0}^{N-1} 2E_j$. The number of triple values will thus be a third of this number, $K' = \frac{2}{3} \sum_{j=0}^{N-1} E_j$.

5) For every $j$, $E_j \leq 2(C_j^+ + C_j^-)$ which is twice the number of pairs $(k,l)$ with $k=0,..,K-1$ and $l=1,2,3$, for which $j(k,l)=j$. In particular, $K' = \frac{2}{3} \sum_{j=0}^{N-1} E_j \leq \frac{4}{3} \sum_{j=0}^{N-1} (C_j^+ + C_j^-) = \frac{4}{3} \ast 3K = 4K$. This means that the \gls{DADG} will have at most $4K$ triple values - twice the number of clauses in the given formula. The enlarged, almost equal formula will thus have at most $8K$ clauses.
\end{rem}

Now that the reader understands how (and why) to enlarge the \gls{DADG}, we will examine this process computationally.

\begin{lem} \label{Step3}
Enlarging the \gls{DADG} - calculating $C_j^+$, $C_j^-$, $E_j$, and $K'$, and defining the additional index functions and parameters - takes linear $O(K)$ time.
\end{lem}

\begin{proof}
In order to calculate $C_j^+$ and $C_j^-$, we begin by setting $C_i^+=C_i^-=0$ for every $j=0,...,N_1$ and, for every $k=0,...,K-1$ and $l=1,2,3$, add 1 to $C_i^+$ if $j(k,l)=i$ and $s(k,l)=1$, and add 1 to $C_i^-$ if $j(k,l)=i$ and $s(k,l)=0$. This process takes $O(K^2)$, since $N \leq 3K$ (the original formula has only $3K$ literals). Calculating $E_j = 2\max\{C_j^+,C_j^-\} - \min\{C_j^+,C_j^-\}$ and $K' = \frac{2}{3} \sum_{j=0}^{N-1} E_j$ takes $O(K)$ time.

Next, for every $j=0,...,N-1$, we need to add index function and parameters for $|C_0^+ - C_0^-|$ new triple values for which all $j(k,1)=j(k,2)=j(k,3)=j$, and the parameters $s'(k,l)$ are determined by the relative size of $C_j^+$ and $C_j^-$. Specifically, we begin with $j=0$, and for every $k=K,K+1,...,K+|C_0^+ - C_0^-|-1$ set $j(k,1)=j(k,2)=j(k,3)=0$. If $C_0^+ > C_0^-$, set $s'(k,1)=1$ and $s'(k,2)=s'(k,3)=0$ for all these $k$'s. Otherwise, we set $s'(k,1)=s'(k,2)=1$ and $s'(k,3)=0$.

Next, we do the same for $j=1$. This time you add $|C_1^+ - C_1^-|$ triple of indexes and parameters - for every $k=K+|C_0^+ - C_0^-|,...,K+|C_0^+ - C_0^-|+|C_1^+ - C_1^-|-1$. We set $j(k,1)=j(k,2)=j(k,3)=1$ and, as before, set $s'(k,1)=1$ and $s'(k,2)=s'(k,3)=0$ if $C_1^+ > C_1^-$, and $s'(k,1)=s'(k,2)=1$ and $s'(k,3)=0$ otherwise. We do the same for every $j=2,...,N-1$. Since we will have $K' \leq 4K$ triple values in total, this takes $O(K)$ time.
\end{proof}

4) We have accumulated enough information about the \gls{DADG} to start constructing it. In this step we ``format" the data structure of the \gls{DADG}. As per Definition~\ref{DADG}, a \gls{DADG} has several ``data fields". It has integers that tell the number of disjoint circles, DB values, branch values and triple values that the \gls{DADG} has. As per the above, the \gls{DADG} we aim to construct has no disjoint circles, DB or branch values, and $K'$ triple values. Set these numbers to $0$, $0$, $0$ and $K'$, respectively. The \gls{DADG} then has two lists that indicate which \gls{EoE} resides on which DB/branch value. Since there are no DB or branch values, these lists should be empty.

Lastly, The \gls{DADG} has a list that indicates which ends of edges reside on which triple value, which arc segment of the triple value, and which ends of edges are \gls{preferred} or non \gls{preferred}. At the end of the algorithm, the $k$th element of the list is supposed to have the form $((ee_1,ee_2),(ee_3,ee_4),(ee_5,ee_6))$, where for each $l=1,2,3$, $ee_{2l}$ is the \gls{preferred} \gls{EoE} of the arc segment $TV_k^l$, and $ee_{2l-1}$ is the non-\gls{preferred} \gls{EoE} of this arc segment.

Each of the 6 $ee_i$'s, in each of the $K'$ triple values, is suppose to be a pair $(r,s)$ where $r$ is an integer and $s$ is binary. For now, the algorithm allocates space for these variables and set them both to $0$. We will set them with the correct values in the next step.

\begin{lem} \label{Step4}
This takes linear $O(K)$ time.
\end{lem}

\begin{proof}
This one is trivial.
%
\end{proof}

5) Lastly, we need to decide which edge begins/ends at which arc segment, and correct the value of the entries in the \gls{DADG} to match this decision.

\begin{rem} \label{WhatToDo}
As per Definitions~\ref{DADG}, \ref{consec} and \ref{H1}, four conditions must be met for the \gls{DADG} to be a well-defined, \gls{Height1} \gls{DADG} with the correct lifting formula: 

a) Each edge has one beginning and one ending. This means that for every $r=0,...,3K'-1$, each of the pairs $(r,0)$ and $(r,1)$ appears only once in the list.

b) If $s'(k,l)=0$, then $TV_k^l$'s \gls{preferred} \gls{EoE} is the ending of some edge, and its non-\gls{preferred} \gls{EoE} is the beginning of some edge. If $s'(k,l)=1$, then it is the other way around.

c) The \gls{DADG} needs to be \gls{Height1}. This means that all edges are either \gls{preferred} (both of their ends are \gls{preferred}), or non-\gls{preferred} (both of their ends are non-\gls{preferred}). 

d) The equivalence classes of the ``same continuation" relation are the double arcs of the \gls{DADG}, and these should match the index functions $j(k,l)$. Recall that at every arc segment one edge ends and one edge begins. These two edges are in the same continuation, by definition. The index function indicates when the edges of different arc segments are on the same continuation.

Specifically, two arc segments $TV_{k_1}^{l_1}$ and $TV_{k_2}^{l_2}$ have the same index function ($j(k_1,l_1)=j(k_2,l_2)$) iff the edges that start/end in one segment must be in the same continuation with the edges that start/end in the other segment.
\end{rem}

\begin{defn} \label{Intend}
For every $j=0,...,N$, we define $P_j=2 \sum_{i=0}^{j-1}E_j$. This means that $P_0=0$, $P_1=2E_0$, $P_w=2E_0+2E_1$, etc.
\end{defn}

Remark~\ref{UpDown}(2) implies that for each $j=0,...N-1$ the $j$th double arc must have $E_j$ arc segments with $s'(k,l)=0$, $E_j$ arc segments with $s'(k,l)=1$, $E_j$ \gls{preferred} edges, and $E_j$ non-\gls{preferred} edges. In particular, it will have $2E_j=P_{j+1}-P_j$ edges in total. In order to account for this, we will ``construct" the $j$th double arc from the edge $P_j,p_j+1,...,P_j+2E_j-1=P_{j+1}-1$. For different $j$'s the corresponding lists of edges are disjoint. This prevents us from accidentally trying to use the same edge while constructing 1 arc.

In order to comply with the demands of Remark~\ref{WhatToDo}, we will define the \gls{DADG} as follows: For every $j$, we will go over the arc segments of the \gls{DADG}, and search for those arc segments $TV_k^l$ for which $j(k,l)=j$ and $s'(k,l)=1$. There are $E_j$ such arc segments. The \gls{preferred} \gls{EoE} of such a segment is supposed to be the beginning of some \gls{preferred} edge of the $j$th arc, and the non-\gls{preferred} \gls{EoE} of such a segment is supposed to be the ending of some non-\gls{preferred} edge of the $j$th arc.

We iterate over those arc segments (the ones with $j(k,l)=j$ and $s'(k,l)=1$), in the order we encounter them. For $a=0,...,E_j-1$, we will set the \gls{preferred} \gls{EoE} of the $a$th arc segment to be $(P_j+a,0)$ - the beginning of the $(P_j+a)$th edge. We will also set the non-\gls{preferred} \gls{EoE} of the $a$th arc segment to be $(P_j+E_j+a,1)$ - the ending of the $(P_j+E_j+a)$th edge.

We will then similarly iterate over the arc segments with $j(k,l)=j$ and $s'(k,l)=0$. For $a=0,...,E_j-1$, we will set the non-\gls{preferred} \gls{EoE} of the $a$th arc segment to be $(P_j+E_j+a,0)$ - the beginning of the $(P_j+E_j+a)$th edge. For $a<E_j-1$, we will set the \gls{preferred} \gls{EoE} to be $(P_j+a+1,1)$ - the ending of the $(P_j+a+1)$th edge. However, for $a=E_j-1$, we will set the \gls{preferred} \gls{EoE} to be $(P_j,1)$ - the ending of the $P_j$th edge.

This method goes over every arc segment exactly once, and assigns values to its two ends of edges. If it assigns the values in compliance with the demands of Remark~\ref{WhatToDo}, then we are done.
 
\begin{lem}
This method of assigning values to the ends of edges complies with the demands of Remark~\ref{WhatToDo}.
\end{lem}

\begin{proof}
a) We need to prove that for every $r=0,...,P_N-1$, the $r$th edge has been assigned a beginning and an ending. The pigeon coop principle implies that each edge has been assigned exactly one beginning and one ending. To encompass all $r$s, go over every $j=0,...,N-1$, and look at the edges $r=P_j,P_j+1,...,P_{j+1}-1=P_j+2E_j-1$.

For $r=P_j$, the construction assigned the beginning of the $r$th edge to the $0$th arc segment with $j(k,l)=j$ and $s'(k,l)=1$, and the ending of this edge to the $(E_j-1)$th arc segment with $j(k,l)=j$ and $s'(k,l)=0$.

For $r=P_j+a$ with $a=1,...,E_j-1$, the construction assigned the beginning of the $r$th edge to the $a$th arc segment with $j(k,l)=j$ and $s'(k,l)=1$, and the ending of this edge to the $(a+1)$th arc segment with $j(k,l)=j$ and $s'(k,l)=0$.

For $r=P_j+E_j+a$ with $a=0,...,E_j-1$, the construction assigned the beginning of the $r$th edge to the $a$th arc segment with $j(k,l)=j$ and $s'(k,l)=0$, and the ending of this edge to the $a$th arc segment with $j(k,l)=j$ and $s'(k,l)=1$.

b) Every time the method assigned a \gls{preferred} \gls{EoE} for an arc segment with $s'(k,l)=0$ it was the ending of some edge, and every time it assigned it a non-\gls{preferred} \gls{EoE} it was the ending of some edge. For arc segments with $s'(k,l)=1$, it was the other way around.

c) For every $j=0,...,N-1$ and $r=P_j,...,P_j+E_j-1$, the beginning and ending of the $r$th edge are \gls{preferred} ends of edge, so these are \gls{preferred} edges. For $r=P_j+E_j,...,P_j+2E_j-1$, the beginning and ending of the $r$th edge are non-\gls{preferred} ends of edge, so these are non-\gls{preferred} edges. This accounts for all edges.

d) For every $j=0,...,N-1$ and $a=0,...,E_j-1$, edge $P_j+a$ is \gls{consecutive} to edge $P_j+E_j+a$ via the $a$th arc segment, with $j(k,l)=j$ and $s'(k,l)=1$. For $a \leq E_j-2$, edge $P_j+E_j+a$ is then \gls{consecutive} to edge $P_j+a+1$ via the $a$th arc segment with $j(k,l)=j$ and $s'(k,l)=0$. This implies that all edges from $P_j$ to $P_j+2E_j=1=P_{j+1}-1$ are on the same continuation. 

On the other hand, if $P_{j_1} \leq r_1 \leq P_{{j_1}+1}-1$ and $P_{j_2} \leq r_2 \leq P_{{j_2}+1}-1$ for $j_1 \neq j_2$, then the $r_1$th and $r_2$th edges cannot be \gls{consecutive}, since the ends of the $r_1$th edge reside on arc segments with $j(k,l)=j_1$ and the ends of the $r_2$th edge reside on arc segments with $j(k,l)=j_2$. 

This implies that edges $P_j$ to $P_j+2E_j=1=P_{j+1}-1$ form an equivalence class of the same continuation relation - a double arc. We enumerate this arc as the $j$th double arc. As per the above, this implies that an arc segment will be a part of the $j$th arc iff an edge between $P_j$ to $P_j+2E_j=1=P_{j+1}-1$ begins and/or ends there iff $j(k,l)=j$.
\end{proof}
%

\begin{lem} \label{Step5}
The above method of assigning beginnings and endings to every edge (step 5 of the algorithm) can be realized by a quadratic $O(K^2)$ time algorithm.
\end{lem}

\begin{proof}
It takes $O(K)$ time to calculate the $P_j$'s of Definition~\ref{Intend}. After this, do the following for every $j=0,...,N-1$:

%
a) Set an integer $a$ to be $0$. For now, $a$ will count the arc segments with $j(k,l)=j$ and $s'(k,l)=1$. Look at the $k,l$th arc segment. If $j(k,l) \neq j$ or $s'(k,l)=0$, then this is not the kind of arc segment we are looking for, and we ignore it. The first time $j(k,l)=j$ and $s'(k,l)=1$, you have found the $a$th such arc segment - $TV_k^l=((0,0),(0,0))$. Per the method, change the \gls{preferred} \gls{EoE} (the second pair) from $(0,0)$ to $(P_j,0)$, and change the non-\gls{preferred} \gls{EoE} (the first pair) from $(0,0)$ to $(P_j+E_j,1)$. Increase $a$ by 1, and search for the next pair of indexes $k,l$, with $j(k,l)=j$ and $s'(k,l)=1$.

In general, the $a$th arc segment with $j(k,l)=j$ and $s'(k,l)=1$ will also have the form $TV_k^l=((0,0),(0,0))$. Change it to $((P_j+E_j+a,1),(P_j+a,0))$, as per the method, increase $a$ by 1, and search for the next pair of indexes $k,l$, with $j(k,l)=j$ and $s'(k,l)=1$. All of this takes $O(K')=O(K)$ time for a specific $j$, and $O(K' \ast N) \leq O(K^2)$ for all $j=0,...,N-1$ (recall that $N \leq 3K$).

Next, you must repeat this process with $a$ now counting the arc segments with $j(k,l)=j$ and $s'(k,l)=0$ (instead of $s'(k,l)=1$). Set $a=0$ and go over every $k,l$. This time search for indexes $k,l$ for which $j(k,l)=j$ and $s'(k,l)=0$. This time, the method dictates that the $a$th such arc segment should be changed from $TV_k^l=((0,0),(0,0))$ to $((P_j+E_j+a,0),(P_j+a+1,1))$, unless $a=E_j-1$ - in which case 
$TV_k^l$ should be changed into $((P_j+E_j+a,0),(P_j,1))$. Afterwards, increase $a$ by 1, and search for the next pair of indexes $k,l$ with $j(k,l)=j$ and $s'(k,l)=0$.

As before, this also takes $O(K')=O(K)$ time for any $j$, $O(K' \cdot N) \leq O(K^2)$ time for all $j$s together. The lemma follows, as does Theorem~\ref{Thm4}.
\end{proof}

\section{The casing and tube construction} \label{CaseTueSec}

In this section, we will explain how to construct a \gls{genericsurface} that realizes a \gls{Height1} \gls{DADG} with no disjoint circles, DB or branch values in $\R^3$. For now, we disregard the computational aspects of the construction. We will examine them in the next section.

Observe Figure~\ref{fig:Figure 21}. It describes a surface in $\R^2 \times [-1,1]$ using a movie - a depiction of the cross-section of the surface with the plane $\{z=a\}$ for several $a$'s between $-1$ and $1$. A cross-section of this form is known as a still. The surface always intersects the plane transversally, and so their intersection will be an immersion of some number of loops into the plane.

\begin{defn}
The surface described in Figure~\ref{fig:Figure 21} is called the ``triple value casing". As mentioned, we will construct a surface that realizes the \gls{DADG} from many pieces, and that includes several copies of the triple value casing. We can thus refer to ``a triple value casing", or several ``triple value casings", of the construction surface.

A triple value casing has 3 double arcs, called the ``casing arcs", each of which intersects each still at one point. For each still we marked the intersection of the first/second/third double arc with the still with a red dot and the number 1/2/3. The arcs intersect once, at a triple value (colored in green), in the still $z=0$. The changes between the stills at $z=-0.2$, $z=0$ and $z=0.25$ reflect the intersection of the arcs at the triple value. 

The only other changes that occur are a) an isotopy of the loops and b) a bridging between two loops or two strands of the same loop. (b) indicates that the surface has a saddle point, which occurs between the stills $z=-0.8$ and $z=-0.6$, $z=-0.6$ and $z=-0.4$, $z=0.25$ and $z=0.5$.

Each loop in each still has an arrow on it describing a normal direction on the loop. This direction is continuous, as we demonstrated on the two rightmost loops in the still $z=-1$. Because of this continuity, we could (and did) depict the direction with only one arrow on any other loop. These directions merge into a \gls{preferred} direction on the surface and so they describe an orientation on the triple value casing. (We assume $\R^3$ has the right hand orientation). All triple value casings are assumed to have this orientation.
\end{defn}

\begin{figure}
\begin{center}
\includegraphics{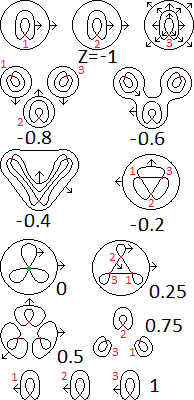}
\caption{A movie description of a triple value casing} 
\label{fig:Figure 21}
\end{center}
\end{figure}

A triple value casing has ``boundaries" that reside on the top and bottom stills ($z= \pm 1$). It is a properly immersed surface in $\R^2 \times [-1,1]$, and in particular a \gls{genericsurface} there, but not in $\R^3$ (where it is not proper). In particular:

\begin{defn}
Each of the 3 double arcs of a triple value casing has two ends - one in the top still ($z=1$), and one in the bottom still ($z=-1$). The one in the top still is surrounded by the shape in Figure~\ref{fig:Figure 22}A, which is called a ``top socket". The one in the bottom still is surrounded by the shape in Figure~\ref{fig:Figure 22}B, which is called a ``bottom socket". A bottom socket looks like a top socket surrounded by a circle. These compose the entire boundary of a triple value casing.
\end{defn}

In order to turn a collection of disjoint triple value casings into a \gls{genericsurface}, we will connect the different bottom and top sockets with tubes.

\begin{defn} \label{Tube}
Figure~\ref{fig:Figure 22}C depicts a ``top tube". It is a bundle over an interval, whose fibres look like top tube sockets. A ``bottom tube" is defined similarly - it looks like a top tube surrounded by a cylinder. The boundary of a top/bottom tube is the union of the two ``end fibres" - the fibres at the ends of the interval. Notice that, like triple values casings, top and bottom tubes come with an orientation.

Each top or bottom tube contains a double arc that goes through its center, we colored it in red in Figure~\ref{fig:Figure 22}C. We called this ``the arc of the tube", or simply the ``tube arc".
\end{defn}

\begin{figure}
\begin{center}
\includegraphics{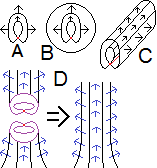}
\caption{Top and bottom sockets, tubes, and tube-gluing} 
\label{fig:Figure 22}
\end{center}
\end{figure}

As in Figure~\ref{fig:Figure 22}D, you may glue two top tubes or two bottom tubes along their end-fibres (one fibre from each tube), and produce a longer tube. You may use the same method to glue an end fibre of a top/bottom tube into a top/bottom socket of some triple value casing. You may glue both end fibres to sockets of the same triple value casing, or each end to a socket of a different casing.

As Figure~\ref{fig:Figure 22}D demonstrates, when you glue the two end fibres / an end fibre and socket together, the points along the gluing area become regular and double values. In particular, if you place an even number of (disjoint) triple value sockets in $\R^3$, arrange their top sockets in pairs, and their bottom sockets too, and connect each pair of top/bottom sockets with a tube (the different tubes do not touch each other, and no part of a tube, other than its end fibres, touches any of the casings), then each point on the resulting surface will be a regular, double or triple value - it will be a closed \gls{genericsurface} in $\R^3$ with no branch values.

\begin{defn}
A surface created by placing $T$ disjoint triple value casings in $\R^3$ (for some even $T$), dividing their $3T$ top sockets into pairs, doing the same for the bottom sockets, and connecting each pair of top/bottom sockets with a top/bottom tube, is called a ``casing and tube construct" or ``casing and tube surface".
\end{defn}

\begin{rem}
The casings and the tubes came with orientations. As Figure~\ref{fig:Figure 22}D demonstrates, 
glue a tube and a casing together, the orientations on the tube and casing ``match" - the combine into a continues orientation on the combined surface. The same happens when you glue two tubes together. It follows that a casing and tube surface inherits a continuous orientation from its different parts.
\end{rem}

\begin{defn} \label{Method}
Given a \gls{Height1} \gls{DADG} $G$ with $T$ triple values and no branch values, DB values or disjoint circles, we will realize it with a \gls{genericsurface} in the following way:

1) Embed $T$ triple value casings into $\R^3$. Index them from $0$ to $T-1$. Each casing contains a unique triple value. Index the triple values such that the triple value in the $k$th casing will be the $k$th triple value $TV_k$, the surface will have no other triple values.

2) Each casing contains 3 casing arcs. In Figure~\ref{fig:Figure 21}, we indexed them from 1 to 3. Name the $l$th casing arc of the $k$th casing $TV_k^l$. As this notion suggests, the casing arcs will serve as the arc segments of the complete surface. 

When you glue a tube to a socket, the casing arc merges with the tube arc. In this way, many casing arcs and tube arcs form one big double arc of the complete surface. In particular, the casing arcs fills the role of the ``arc-segments" of the intersection graph of the full surface - each triple value of the surface is contained in a casing, and is the intersection of its 3 casing arcs. Each casing arc is indeed just a small segment of some double arc in which the double arc crosses the triple value.

Each casing arc is divided into two parts - before the triple value, in stills $z=-1$ to $z=0$, and after the triple value, in stills $z=0$ to $z=1$. These are the two ends of edge that make up the arc segment. As Figure~\ref{fig:Figure 21} depicts, the orientation of the triple value casings always points towards the arc segment in stills $z=0$ to $z=1$, so this will be the \gls{preferred} \gls{EoE} of the arc segment in the complete surface. Notice that the \gls{preferred} \gls{EoE} at each arc segment ends in a top socket and the non-\gls{preferred} \gls{EoE} ends in a bottom socket.

Since the half of the casing arcs will be the end of edges of the surface, you should index them in a way that matches the \gls{DADG}. Each arc segment $TV_k^l$ in the \gls{DADG} is a pair of (indexes of) ends of edges $((r_{2l-1},s_{2l-1}),(r_{2l},s_{2l}))$. As per the above identification, the top half of the casing arc $TV_k^l$ will be the $(r_{2l},s_{2l})$th \gls{EoE} of the surface, and the bottom half of the same casing arc will be the $(r_{2l},s_{2l})$th \gls{EoE} of the surface. Index them as such.

3) Since the \gls{DADG} is \gls{Height1} and have no DB or branch values, each edge is either \gls{preferred} or non-\gls{preferred}. It has $3T$ edges. For every $r=0,...,3T-1$, if the $r$th edge is \gls{preferred}, then both the arc segments $(r,0)$ and $(r,1)$ are \gls{preferred}. Using the identification above, the top halves of two of the casing arcs have been indexed as $(r,0)$ and $(r,1)$. They each end in a top socket. Similarly, if the $r$th edge is non-\gls{preferred}, then the bottom halves of two of the casing arcs have been indexed as $(r,0)$ and $(r,1)$. Connect them with a bottom tube and, as before, call it tube $r$.

This finishes the casing and tube surface. For every $r=0,...,3T-1$, it has a unique tube called $r$ that connects the casing arcs $(r,0)$ and $(r,1)$. There is an edge of the complete surface whose ends of edges are the casing arcs $(r,0)$ and $(r,1)$, and whose bulk is the tube arc of the $r$th tube. This will be the $r$th edge of the complete surface. Give it a direction that points towards the \gls{EoE} $(r,1)$, which will thus be the ending of the edge while $(r,0)$ will be the beginning of the edge. This accounts for all $3T$ edges of the casing and tube surface, which means that we indexed all the edges of the surface as $r=0,...,3T-1$.
\end{defn}

\begin{lem} \label{MethodWorks}
1) This casing and tube surface, with the aforementioned \gls{preferred} direction on the edges, is a \gls{Thrice} \gls{genericsurface} in $\R^3$.

2) Using the aforementioned indexing of the edges, triple values and arc segments, the \gls{DADG} of the surface is equal to the given \gls{DADG}.
\end{lem}

\begin{proof}
1) Every casing and tube surface is oriented. It remains to be proven that the directions of progress on all the edges merge into a continuous direction on the double arcs (everything else is upheld by every casing and tube surface). This means only that whenever you move from one edge of the arc to the next, in an arc segment, one edge ends and the other begins. This follows from the fact that the construction implies that one edge ends and one edge begins at each arc segment.

2) Both the original \gls{DADG} and surface \gls{DADG} have $T$ triple values and no DB values, branch values or disjoint circles. These \gls{DADG}s will be equal iff, for every $k=0,...,T-1$ and $l=1,2,3$, the arc segments $TV_k^l$ of both \gls{DADG}s have the same \gls{preferred} \gls{EoE} $(r_{2l},s_{2l})$ and the same non-\gls{preferred} \gls{EoE} $(r_{2l-1},s_{2l-1})$. The construction clearly ensures this.
\end{proof}

\section{Surface blocks} \label{SBsec}

We now know how to realize every \gls{Height1} \gls{DADG} with no DB or branch values using a \gls{Thrice} \gls{genericsurface} in $\R^3$. If you take a big 3-ball $B$ that contains the surface, and treat the surface as a sub-complex of $B$, then you will have realized the \gls{DADG} with a closed \gls{genericsurface} in $D^3$. Unfortunately, this construction is abstract rather than concrete. The surface and 3-manifold are not triangulated, and so we cannot examine the complexity of the construction as it was given.

In this section, we modify the construction to produce a triangulated surface in a triangulated $D^3$. Instead of taking an already triangulated copy of $D^3$ and try to embed the surface in it in such a way that the surface will be a sub-complex of this triangulation, we will build the surface in $\R^3$ from triangulated pieces. We refer to these pieces as ``surface blocks".

%
%
\begin{defn}
1) A concrete simplicial complex in $\R^3$ is a set of (linearly embedded) simplices in $\R^3$ that adheres to the definition of a simplicial complex - if a simplex is in this set, all of its faces must also be in this set. A concrete simplicial complex is ``rational", if each coordinate of every vertex of every simplex is a rational number. In a computer program, one can represent a point in $\Q$ using a trio of numbers (its coordinates).

The data type of a (rational) concrete simplicial complex $M$ thus contains: a) A list $M_0$ of the vertices of the complex - the $r$th entry in the list is the coordinates of $r$th vertex. We also included an integer $\#V$, whose value is the length of $M_0$.

b) 3 lists - $M_1$, $M_2$ and $M_3$ respectively listing the 1, 2 and 3 simplices. The $r$th entry in the list $M_n$, representing the $r$th $n$-simplex of the complex, is an $n+1$-tuple whose elements are \textbf{indexes} of the vertices of this $n$-simplex. For instance, if the complex contains the 1-simplex $((0,0,0),(1,2,0))$, $(0,0,0)$ is the $6$th entry in the list of vertices ($M_0$), and $(1,2,0)$ is the $4$th entry therein. Therefore, the said 1-simplex will be represented by the pair $(6,4)$ (the list $M_1$ will contain the entry $(6,4)$.)

2) One may choose to ignore the way a concrete simplicial complex is embedded in $\R^3$ and only regard the underlying abstract simplicial complex. Algorithmically, this involves only ``ignoring" the irrelevant data field of the concrete complex, the list $M_0$, and it takes $O(1)$ time. 

3) A triangulated rectanguloid in $\R^3$ is a concrete simplicial complex whose total space is a rectanguloid. We use the notation ``triangulated \underline{   }" for other geometric shapes as well. For instance, a triangulated cube is a cube made of simplices and the faces of a triangulated rectanguloid are triangulated rectangles.

4) A ``surface block" is a pair $(R,S)$ where $R$ is a triangulated rectanguloid, and $S$ is a \gls{genericsurface} in $R$. The size of the surface block is the number of 3-simplices in $R$. In the interest of convenience, we make the surface $S$ disjoint from the vertices and edges of the rectanguloid (not the 1-skeleton of the complex, but just the actual corners of the rectanguloid). A surface block is ``closed" if the \gls{genericsurface} $S$ is closed. This means it has no boundary - no RB or DB values. Equivalently, it is disjoint from the boundary of $R$. When we refer to a ``face of a surface block", it means ``one of the faces of the rectangoloid".
\end{defn}

If you place two triangulated rectanguloids $A$ and $B$ next to each other, such that their intersection $A\cap B$ is both a face of $A$ and a face of $B$, and if this face inherits the same triangulation from $A$ and $B$, then the union $A \cup B$ is a bigger triangulated rectanguloid. We refer to this as ``gluing" two rectanguloids together. One may glue surface blocks in the same way. In order to glue two surface blocks $(A,S_A)$ and $(B,S_B)$ along a joint face $A \cap B$, the intersection of this face with the surfaces in $A$ and $B$ must coincide - $A \cap B \cap S_A=A \cap B \cap S_B$. That way $S_A \cup S_B$ will be a surface of $A \cup B$. A value $p$ in $A \cap B \cap S_A$ can be either an RB or DB value of the surface $S_A$. It will be the same type of value in $S_B$. In the combined surface, $S_A \cup S_B$ $p$ will respectively be a regular or double value as Figure~\ref{fig:Figure 23} demonstrates.

\begin{figure}
\begin{center}
\includegraphics{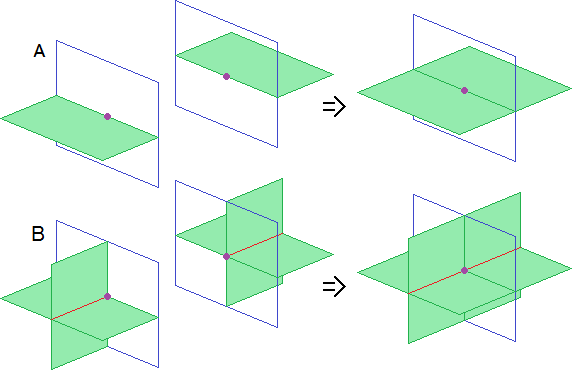}
\caption{Gluing two surface blocks turn an RB/DB value on the glued faces into a regular/double value} 
\label{fig:Figure 23}
\end{center}
\end{figure}

We will construct the casing and tube surface algorithmically, by gluing together many small ``atomic" surface blocks. There are only 8 types of atomic blocks, but the complete surface is made of many copies of each type.
%

\begin{defn}
1) Let $(R,S)$ be a surface block.

a) A ``blank face" is a $1 \times 1$ square face that is disjoint from $S$, and has the triangulation of Figure~\ref{fig:Figure 24}A. Its 9 vertices must be the corners of the square, the center of the square and the center of each edge of the square.

b) A ``top socket face" is a $1 \times 1$ square face, the triangulation of which is as depicted in Figures~\ref{fig:Figure 24}B. The only vertices on the boundary of the square are the corners and the center of each edge. The intersection of the bottom socket face with $S$ is depicted in green, and it is clearly a top tube socket. We also require the triangulation to have a mirror symmetry in regard to the middle axis (the dotted line). Other than these requirements, the reader may choose the exact position of each vertex in the triangulation as required. The only caveat is that all bottom socket faces of all the surface blocks in the construction must have the same triangulation, otherwise it will be impossible to glue different surface blocks along these faces.

When drawing a top socket face, we will usually depict only its intersection with $S$ and the 8 triangles that share an edge with the boundary of the face, as in Figure~\ref{fig:Figure 24}C. A bottom socket face may point in different directions. For instance, in Figure~\ref{fig:Figure 24}C it points rightwards, and in Figure~\ref{fig:Figure 24}D it points upwards.

c) A ``bottom socket face" is defined similarly to a top socket face. Its triangulation is depicted in Figure~\ref{fig:Figure 24}E.

2) A face can be a $n \times m$ lattice of different sockets. The face is divided into $n$ columns and $m$ rows of $1 \times 1$ squares. Figure~\ref{fig:Figure 24}F is a $2 \times 3$ lattice with a downwards pointing top socket face on the $(1,1)$ place, a right-pointing bottom socket face on the $(2,3)$ place, and blank faces everywhere else.

\begin{figure}
\begin{center}
\includegraphics{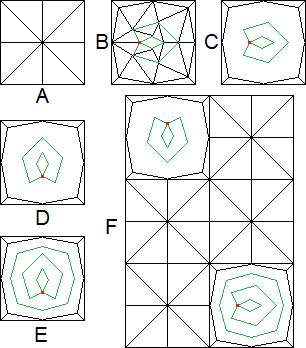}
\caption{A blank face, top and bottom socket faces, and a lattice face} 
\label{fig:Figure 24}
\end{center}
\end{figure}

3) The algorithm will use the following surface blocks:

a) The ``standard empty block" is a triangulated cube, whose total space is $[0,1]^3$ and whose faces are triangulated as blank faces. The reader may choose the exact triangulation of this block. An ``empty block" is any triangulated cube created by transposing the standard empty block - adding a constant vector to all of its vertices.

b) The ``standard straight top tube" is a surface block whose total space is $[0,1]^3$,
and whose boundary is depicted in ``unfolded" Figure~\ref{fig:Figure 25}A. In order to get the actual boundary, you should fold the shape in Figure~\ref{fig:Figure 25}A along each of the orange lines and use a 90-degree fold.

Specifically, two of its antipodal faces are bottom socket faces that point in the same direction and the other four faces are blank. Its interior contains (only) a bottom tube that connects the two faces. Again, the reader may choose the exact triangulation of this block. A ``straight top tube" is any surface block whose faces are parallel to the coordinate planes, and is created by transposing, or rotating and then transposing, the standard straight bottom tube block.

Due to the rotation, the ``socket faces" of the tube may be parallel to any of the coordinate planes. The socket in these faces may point in different directions too. It will always be clear what formation the rotation left the tube in.

c) The ``standard horizontal top tube corner" is a surface block whose total space is $[0,1]^3$
and whose boundary is depicted unfolded in Figure~\ref{fig:Figure 25}B. Specifically, two of its adjacent faces are bottom socket faces, that point in the same direction and the other four faces are blank. Again, the interior contains only a tube connecting the two sockets. Once more, the reader may choose the exact triangulation, and a rotation and/or transposition of this block is called a ``horizontal top tube corner".

d) A ``vertical top tube corner" is defined similarly to a horizontal one, except that the two adjacent bottom socket faces point towards each other, as depicted in unfolded form in Figure~\ref{fig:Figure 25}C. 

e,f,g) A ``straight bottom tube", a ``horizontal bottom tube corner" and a ``vertical bottom tube corner" are defined similarly to their top tube counterparts, but with bottom tubes instead of top tubes. For instance, Figure~\ref{fig:Figure 25}D depicts the boundary of a horizontal bottom tube corner unfolded.

\begin{figure}
\begin{center}
\includegraphics{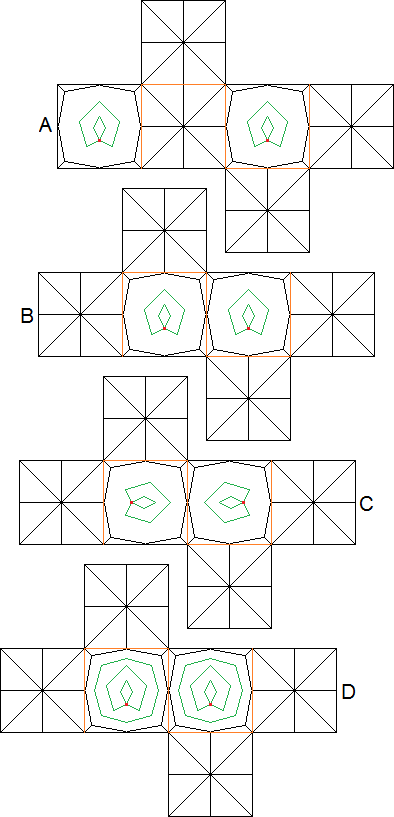}
\caption{The boundaries of the tube blocks} 
\label{fig:Figure 25}
\end{center}
\end{figure}

h) The ``standard casing block" is a surface block whose total space is $[0,2] \times [0,3] \times [0,1]$ and whose boundary is depicted unfolded in Figure~\ref{fig:Figure 26}. Specifically, both $1 \times 2$ faces and both $1 \times 3$ faces are respectively $1 \times 2$ and $1 \times 3$ lattices of blank faces. The top (resp. bottom) $2 \times 3$ face is a lattice with one row of 3 blank faces, and the other row has 3 top (resp. bottom) socket faces pointing towards the line that divided the 2 rows. Additionally, the top sockets of the top face are placed directly above the bottom sockets of the bottom face (and not above the row of blank faces).

The interior of this must contain a triangulated triple value casing, whose top and bottom sockets are those in the socket faces. We mark the sockets with numbers from 1 to 3. The $l$th top and bottom socket must be at the ends of the same arc segment, which we will refer to as the $l$th arc segment of the casing block. The reader may once again choose the exact triangulation, and a ``casing block" is a transposition of the standard one.

\begin{figure}
\begin{center}
\includegraphics{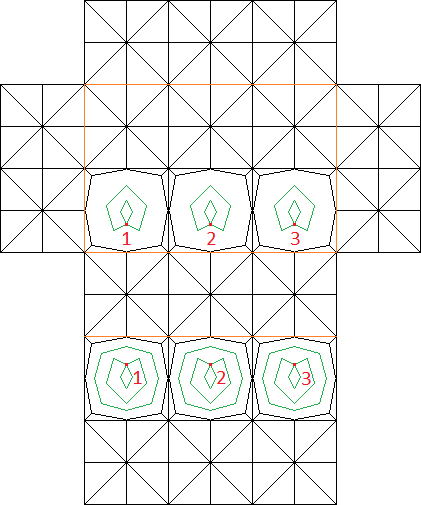}
\caption{The boundary of a casing block} 
\label{fig:Figure 26}
\end{center}
\end{figure}

4) Let $MaxT$ be any integer that is greater or equal to the number of 3-simplices in any one of the 8 types of atomic surface blocks.
\end{defn}

We will construct the casing and tube surface block by gluing together atomic blocks. At each step, we will take a new surface block and glue it to the ``part of the casing and tube surface that we already constructed", which we refer to as the ``current complex". In order to determine the complexity of this algorithm, we must rigorously examine the process of gluing surface blocks. Each step of the algorithm includes two parts:

1. The first part is formally ``writing down" the new atomic block $(M_+,S_+)$. We assume that the algorithm contains a copy $(M_s,S_s)$ of each of the 8 standard atomic blocks. In order to create $(M_+,S_+)$, one must rotate and/or transpose $(M_s,S_s)$. Rotating and/or transposing a surface block means moving every vertex to a new position. For instance, one can rotate a surface block whose total space is $[0,1]^3$ by changing the coordinates of every vertex from $(x,y,z)$ to $(y,1-x,z)$. There is no need to change any other data-field of the surface block - a simplex whose vertices are the $0$th, $3$rd, $5$th and $6$th vertices of the complex will still have these same vertices after the rotation and/or transposition, but the coordinates of vertices may change.

Every atomic block has at most $4MaxT$ vertices, and so a rotation / transposition involves using the same coordinate change $4MaxT$ times. For rotations, each of these changes takes $O(1)$ time since there is only a finite number of ways one can rotate a rectanguloid such that its edge end up parallel to the coordinate axes. Rotation thus takes $O(1)$ time. A transposition involves adding a constant vector to each of the vertices. We will always use an integer vector of the form $(i,j,h)$ where $i=0,1$, $j=0,...,3T-1$ and $h=-\frac{3}{2}T,...,\frac{3}{2}T$. A transposition will therefore take $O(\log(T))$ amount of time. However, in the spirit of Remark~\ref{Calculate}(2), we will treat it as though it takes $O(1)$ time.

2. The second part is merging the new block $(M_+,S_+)$ to the current complex $(M,S)$. $M$ will be a concrete simplicial complex, but not necessarily a rectanguloid. For instance, its total space may be the union of the cubes $[0,1]^3$, $[1,2] \times [0,1]^2$ and $[0,1] \times [1,2] \times [0,1]$. There are two difficulties in merging the new block $(M_+,S_+)$ to the current complex $(M,S)$:

i) Some of simplices of the new block will already be inside the current complex. For instance, if we add a block whose total space is $[1,2]^2 \times [0,1]$ to the above example, then the simplices of the faces $[1,2] \times \{1\} \times [0,1]$ and $\{1\} \times [1,2] \times [0,1]$ appear in both $(M_+,S_+)$ and $(M,S)$. In order to merge these two we must identify these simplices, and add all of the other simplices of $M_+$ into $M$.

ii) Every 1, 2 or 3 simplex in the new block is a list of the indices of the vertices of that simplex. When we merge the complexes, the indices of these vertices change. For instance, assume that $(1,1,1)$ was the $30$'th vertex in the current complex and the $6$th vertex in the new block, that $(1,\frac{3}{2},1)$ was the $12$th vertex of new block, and when we add the vertices of the new block to the current complex, $(1,\frac{3}{2},1)$ becomes the $45$th entry in this list. The edge between these two vertices was represented by the pair $(6,12)$ in the new block. When we ``add it" to the current complex, it will have to be represented by the pair $(30,45)$. We need to add the pair $(30,45)$, rather than $(6,12)$, to $M$. We will need to do something similar with every 1,2 or 3 simplex $M_+$ that is not in $M$, and repeat the process for $S_+$ and $S$.

\begin{lem} \label{CompGlue}
Merging the $k$th atomic block into the current complex can take $O(k)$ time.
\end{lem}

\begin{proof}
At this point, the current complex $M$ is made of $k-1$ atomic blocks, so it contains at most $k \ast MaxT$ 3-simplices. The new atomic block $M_+$ contains at most $MaxT$ 3-simplices. Since $M$ and $M_+$ are pure complexes, the numbers of vertices, edges and 2-simplices they contain are bounded by $4$ or $6$ times the number of 3-simplices. Set $a \leq 4k \ast MaxT$ and $c \leq 6k \ast MaxT$ to respectively be the numbers of vertices and edges in $M$ (before the merge), and $b \leq 4MaxT$ to be the number of vertices in $M_+$. 

For every $i=0,...,b-1$, one of two options will take place. Either the $i$th vertex of $M_+$ is already in $M$, in which case we need to find its index $\sigma(i)$ as a vertex of $M$ as per (ii), or it is not already in $M$, in which case we must add it to $M$. At the start of this step, we define a new variable $r$ and set it to $0$. $r$ will count the ``new" vertices in $M_+$, the ones that are not in $M$. The $r$th new vertex in $M_+$ will need to be assigned a new index in when it is added to $M$ - it will be the $a+r$th vertex of $M$, and we will set $\sigma(i)=a+r$ accordingly.

In order to realize this, do the following for every $i=0,...,b-1$: set $(x,y,z)$ to be the $i$th vertex of $M_+$, and search for $(x,y,z)$ among the $a$ vertices of $M$. This takes $O(a) \leq O(k)$ time. If the $n$th vertex in $M$ is equal to $(x,y,z)$, set $\sigma(i)=n$. If you do not find $(x,y,z)$ in $M$, then $(x,y,z)$ is the $r$th new vertex in $M_+$. You should add $(x,y,z)$ to the list of vertices of $M$, set $\sigma(i)=a+r$ and then increase $r$ by 1. If this $i$th vertex is both a new vertex and is in $S_i$, then it should be added to $S_i$ as well. This means that the $\sigma(i)$th vertex should be added to $S$. Doing so for every $i=0,...,b-1$ takes $O(k \ast b) \leq O(4k \ast MaxT) \leq O(k)$ time, since $b$ is bounded by the constant $4MaxT$.

The variable $\#V$ contains the number of vertices in the complex $M$. At the end of this process, $r$ will be equal to the number of new vertices that were added to $M$, so change the value of $\#V$ to $\#V+r$. This takes $O(1)$ time.

Next, for every edge $(i,j)$ in $M_+$, you should search for $(\sigma(i),\sigma(j))$ among the $c$ 1-simplices of $M$. This takes $O(b) \leq O(k)$ time. If $(i,j)$ is not in $M$ you must add it to $M$. If additionally it is in $S_+$, then you must also add it to $S$. Since $M_+$ has a bounded number of edges, doing this for all edges still takes $O(k)$ time. Doing the same thing for 2 and 3 simplices also takes $O(k)$ time, due to similar reasons.
\end{proof}

\begin{thm}\label{Thm5}
There is a cubic $O(T^3)$ time algorithm that receives a \gls{Height1} \gls{DADG} with $T$ triple values, no disjoint circles, DB or branch values, and produces a surface block version of the casing and tube surface of this \gls{DADG}.
\end{thm}

\begin{rem}
1) As per the explanation in the beginning of the chapter, Theorems~\ref{Thm4} and \ref{Thm5} and Lemma~\ref{MethodWorks} imply Theorem~\ref{Thm3}, and thus imply that the lifting problem is NP-hard.

2) Technically, Theorem~\ref{Thm3} speaks of the algorithm that creates a \gls{genericsurface} in an \textbf{abstract} triangulated 3-manifold, while this algorithm produces a surface block - a \gls{genericsurface} in a \textbf{concrete} triangulated 3-manifold. However this is irrelevant, as one can simply ignore the additional information that makes the complex concrete (the coordinates of each of the vertices).
\end{rem}

\begin{proof}
We begin by embedding a matching casing block for every triple value of the \gls{DADG}. Specifically, we will embed $T$ casing blocks into $\R^3$. The total space of the $k$th block ($k=0,...,T_1$) is $[0,2]  \times [k,k+3] \times [0,1]$. The $k$th casing block will represent the $k$th triple value of the given \gls{DADG}, and its $l$th casing arc will represent the $(k,l)$th arc segment of the given \gls{DADG}. Figure~\ref{fig:Figure 27}A depicts this for a \gls{DADG} with 2 triple values.

Next, we will realize the \gls{preferred} edges of the \gls{DADG} with top tubes. Define a list of $3T$ binary variables $t(1),...,t(3T)$, and set them all to $0$. We'll explain their purpose later on. Additionally, define a variable $b$ that counts the \textbf{\gls{preferred}} edges of the \gls{DADG}. It will have the values 0 to $\frac{3}{2}T-1$. Begin by setting it to 0. Go over the edges of the \gls{DADG} and find the \gls{preferred} ones - for every $r=0,...,3T-1$ go over all $k=0,...,T-1$ and $l=1,2,3$, and check if either of $TV_k^l$'s ends of edges is $(r,0)$ or $(r,1)$.

If you found a non-\gls{preferred} edge that is equal to either $(r,0)$ or $(r,1)$, then the $r$th edge is non-\gls{preferred} and should be ignored - move on to the next $r$. Otherwise, the $r$th edge is also the ``$b$th \gls{preferred} edge". During your search you will find $k_1,k_2$ and $l_1,l_2$ such that $3k_1+l_1 < 3k_2+l_2$, and the \gls{preferred} ends of edges of $TV_{k_1}^{l^1}$ and $TV_{k_2}^{l^2}$ are $(r,0)$ and $(r,1)$. As per the socket and tube construction, we must connect the top sockets at the ends of the $(k_1,l_1)$th and $(k_2,l_2)$th socket arcs. 

We accomplish this as follows: at the first $r$ for which the $r$th edge is \gls{preferred}, when $b=0$, the current complex contains only  the casing blocks. The $(k,l)$th socket arc ends in the $l$th top socket face of the $k$th casing block, which is the face the total space of which is $[0,1] \times [3k+l-1,3k+l] \times \{1\}$. Denote $i_1=3k_1+l_1$ and $i_2=3k_2+l_2$. We must thus connect the sockets at $[0,1] \times [i_1-1,i_1] \times \{1\}$ with that in $[0,1] \times [i_1-1,i_1] \times \{1\}$ using a top tube.

We draw this tube in Figure~\ref{fig:Figure 27}B, where $k_1=0$, $l_1=1$, $k_2=1$, $l_1=2$ and thus $i_1=1$ and $i_2=5$. We start by placing a horizontal top tube corner over the $i_1$th socket, at $[0,1] \times [i_1-1,i_1] \times [1,2]$. The other end of this tube is at the face $\{1\} \times [i_1-1,i_1] \times [1,2]$. We ``turn" the tube to the right by placing a vertical top tube corner at $[1,2] \times [i_1-1,i_1] \times [1,2]$. We then continue forward by placing straight top tubes at $[1,2] \times [i-1,i] \times [1,2]$ for $i_1<i<i_2$. We turn right again using a vertical top tube corner at $[1,2] \times [i_2-1,2_1] \times [1,2]$, and then turn the tube downwards into the $i_2$th socket by placing a horizontal top tube corner at $[0,1] \times [i_2-1,2_1] \times [1,2]$.

Each of the tube pieces we added to the surface is a part of a whole surface block that we added to the current complex. Each of these blocks has two faces which are top tube sockets, but they have already been used in the gluing. You can see this in Figure~\ref{fig:Figure 27}C, which depicts the same tube as in Figure~\ref{fig:Figure 27}B, but emphasizes the ``socket faces" of all the tube blocks. The other faces of each block are blank faces. Figure~\ref{fig:Figure 27}D depicts the current complex after adding the tube blocks. 

\begin{figure}
\begin{center}
\includegraphics{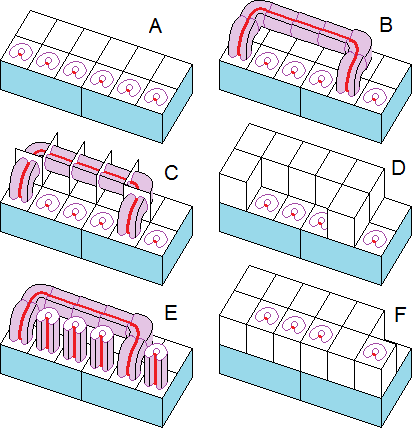}
\caption{Realizing a \gls{DADG} with a casing and tube surface block} 
\label{fig:Figure 27}
\end{center}
\end{figure}

We set $t(i_1)$ and $t(i_2)$ to $1$ to indicate that we already attached the corresponding sockets to each other with a tube. We would like to similarly realize the other \gls{preferred} edges with a tube. In order to avoid intersecting this first tube, the other tubes will go above it, on other ``levels" of the complex. For every $i$ for which the $i$th socket has not yet been connected with a tube - those $i$'s for which $t(i)=0$ - we ``move the $i$th socket" one level upwards by attaching a straight top tube at $[0,1] \times [i-1,i] \times [1,2]$. Figure~\ref{fig:Figure 27}E depicts these tubes, and \ref{fig:Figure 27}F depicts the current complex after we add them. You can think of this as extending the $i$th casing arc upwards, until it reaches the next level - into $[0,1] \times [i-1,i] \times \{2\}$.

We will be able to attach the next tube on this level, at $z \in [2,3]$, to avoid the first tube in $z \in [1,2]$. We will then move another level upwards, and attach the third tube, and so on. Before we move to the $[2,3]$ level, we must ``fill in" every empty spot in the $[1,2]$ level with an empty block, as Figure~\ref{fig:Figure 28}A depicts. We do this since we want the complete ``casing and tube" surface to be a block - to be contained in a full rectanguloid with no holes in it. We must also increase $b$ by $1$, to indicate that we are moving on to the next \gls{preferred} edge.

The first tube is added when $b=0$, in the $z \in [1,2]$ level. The second is added when $b=1$, in the $z \in [2,3]$ level. In general, the $b$th \gls{preferred} tube is placed in the $z \in [1+b,2+b]$ level. For any new $b$, we find the next $r$ for which the $r$th edge is \gls{preferred}, and we find $k_1$, $l_1$, $k_2$ and $l_2$ as before. We define $i_1=3k_1+l_1$ and $i_2=3k_2+l_2$, and we need to connect the sockets at $[0,1] \times [i_1-1,i_1] \times \{b+1\}$ and $[0,1] \times [i_2-1,i_2] \times \{b+1\}$. As before, we do this by placing horizontal top tube corners at $[0,1] \times [i_1-1,i_1] \times [1+b,2+b]$ and $[0,1] \times [i_2-1,i_2] \times [1+b,2+b]$, vertical top tube corners at $[1,2] \times [i_1-1,i_1] \times [1+b,2+b]$ and $[1,2] \times [i_2-1,i_2] \times [1+b,2+b]$, and straight top tubes at $[1,2] \times [i-1,i] \times [1+b,2+b]$ for $i_1<i<i_2$.

\begin{figure}
\begin{center}
\includegraphics{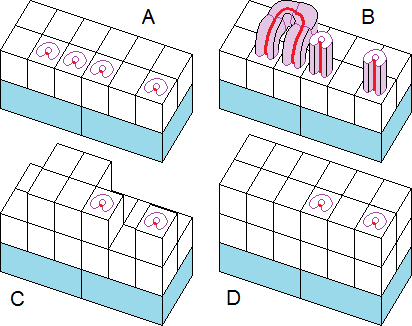}
\caption{Realizing the a \gls{DADG} with a casing and tube surface block 2} 
\label{fig:Figure 28}
\end{center}
\end{figure}

Once more, we set $t(i_1)$ and $t(i_2)$ to $1$, and we lift every socket that has not yet been connected to another socket into the next level by attaching straight top tubes at $[0,1] \times [i-1,i] \times [1,2]$ for every $i$ for which $t(i)=0$. Figures~\ref{fig:Figure 28}B,C depict this for $b=1$. The tube on the previous level connects the $i_1=1$ and $i_2=5$ sockets, and on this level it connects the $i_1=2$ and $i_2=3$ sockets, so at this point $t(1)=t(2)=t(3)=t(5)=1$ and $t(4)=t(6)=0$. Next, we fill the holes in the $[b+1,b+2]$ level with empty blocks. As Figure~\ref{fig:Figure 28}D depicts, these ``holes" occur at two locations: first, at $[0,1] \times [i-1,i] \times [b+1,b+2]$ for every $i$ for which $t(i)=0$ and $i \neq i_1,i_2$; secondly, at $[!,2] \times [i-1,i] \times [b+1,b+2]$ for every $i$ for which $i<i_1$ or $i_2<i$. We than increase $b$ by $1$ and continue to the next \gls{preferred} edge and the next level.

The process continues until $r$ has gone over all edges (reaches $3T$), or $b$ has gone over all \gls{preferred} edges (reaches $\frac{3}{2}T$.) By this time, we have attached a top tube per every \gls{preferred} edge, and there are no top sockets renaming. We than perform a similar process with the non-\gls{preferred} edges and bottom tubes. For every $b=0,...,\frac{3}{2}T-1$, we attach a new level \textbf{below} the current complex, in which we add a bottom tube that realizes the $b$th non-\gls{preferred} edge. We also move every bottom socket that has not yet been connected to a tube downwards into the next floor, so we can connect them with tubes later on.

Now, let us examine the runtime of this construction:

In order to find the $b$th \gls{preferred} edge, the algorithm searches, for every $r=0,...,3T-1$, for the arc segments where the $r$th edge begins and ends. This involves going over every $k=0,...,T-1$ and $l=1,2,3$, and so it takes $O(T)$ time. Repeating this for every $r$ takes $O(T^2)$ time. We later repeat this process, in order to find the non-\gls{preferred} edges. Once more, it takes $O(T^2)$ time.

Aside from this, the algorithm does the following for each $b$ when attaching both the top and the bottom levels: it defines $i_1$ and $i_2$, changes the values of $t(i_1)$ and $t(i_2)$ to 1, and increases $b$ by 1. This takes $O(1)$ time for every $b$ and $O(T)$ time for all $b$s. 

Lastly, the algorithm places $T+18T^2$ blocks - the $T$ triple values casing blocks, and $6T$ blocks at each of the $3T$ levels ($\frac{3}{2}T$ levels above the triple values and the same number below them). As per Lemma~\ref{CompGlue}, this will be accomplished in $\sum_{k=1}^{T+18T^2}O(k)=O(T^3)$ time.
\end{proof}

We would like to finish with two remarks: 

1) As we mentioned before, there is more than one way to define a \gls{genericsurface} as a data type. We used \gls{genericsurface}s in already triangulated 3-manifolds. Another way would be to use triangulated \gls{genericsurface}s in $\R^3$ - the surface is a 2-dimensional concrete simplicial complex in $\R^3$, and there is no need for a triangulated 3-manifold that contains the surface.

The lifting problem for this type of surface is still NP-complete. It is possible to construct lifting formulas for these type of surfaces, and prove that their lifting problem is NP, using a similar process to the one we used in section~\ref{CompOfLift}. As for NP-hardness, the proof of Theorem~\ref{Thm3} still works for this kind of surface. Indeed, the algorithm creates a surface block $(M,S)$ that realizes a given \gls{sym&pro}, and in particular $S$ is a triangulated \gls{genericsurface} in $\R^3$ that realizes the formula.

2) Instead of realizing each formula with a surface in $D^3$, the algorithm can be modified to realize them with surfaces in any chosen 3-manifold $X$, and in particular in $S^3$. The general idea is to take the \gls{genericsurface} $(M,S)$ that the algorithm creates ($M$ is a ball), and a triangulation $M'$ of $\overline{X \setminus D^3}$ ($X$ with an open ball removed from it), and glue the boundary of $M$ to that of $M'$, producing a new \gls{genericsurface} $(\overline{M},S)$, where $\overline{M}$, the union of $M$ and $M$, is homeomorphic to $X$.

The problem is that the triangulations of the gluing-boundaries of $M$ and $M'$ need to match. $M'$ will have a constant triangulation, but that of $M$ depends on the number $T$ of triple values in the \gls{DADG}. Specifically, note that the total space of the constructed surface block is $[0,2] \times [0,3T] \times [-\frac{3}{2}T,1+\frac{3}{2}T]$, that it is a closed surface block, which implies that its boundary is disjoint from the surface within the block, and that it consists of atomic blocks, which implies that all the faces are lattices of blank sockets. The top and bottom faces are $2 \times 3T$ lattices, the front and back faces are $3T \times 3T+1$, and the left and right faces are $2 \times 3T+1$ lattices.

There is a simple trick to change the triangulation of the boundary into a constant one. Given a closed surface $\Sigma$ with a tiling $A$, the tiling may contain not only triangles but other polygons as well. Observe the dual tiling $B$. There is a standard tiling for $\Sigma \times I$, the one boundary of which ($\Sigma \times \{0\}$) is tiled as $A$, and the other boundary is tiled as $B$. For every vertex/edge/face $\sigma$ of $A$ has a matching face/edge/vertex $\sigma'$ in $B$. Each 3-cell in the tiling of $\Sigma \times I$ is the convex hall of $\sigma$ in $\Sigma \times \{0\}$ and $\sigma'$ in $\Sigma \times \{1\}$.

Now, if you subdivide $A$ into a triangulation $A'$ (a tiling made of triangles), you can appropriately subdivide the tiling of $\Sigma \times I$. For every $n=0,1,2$, every $n$-cell $\sigma$ of $A$ and every $n$-cell $\tau$ of $A'$ that is contained in $\sigma$, this refined tiling of $\Sigma \times I$ has one matching 3-cell - the convex hall of $\tau$ in $\Sigma \times \{0\}$ and $\sigma'$ in $\Sigma \times \{1\}$. This is indeed a refinement of the previous tiling - the union of the 3-cells that correspond to all $\tau$'s in a given $\sigma$ is equal to the 3-cell that corresponded to $\sigma$ in the previous tiling. Additionally, if $B$ is a triangulation (and its faces are triangles), then the refined tiling of $\Sigma \times I$ is also a triangulation (all its 3-cells are simplices).

Think of the triangulation of the $\partial M$ as a refinement of the usual cube tiling, where each square has been divided into a lattice of blank faces. The opposite tiling, the octahedron, is indeed a triangulation (all the faces of the octahedron are triangles). Let $M''$ be the triangulation of $S^2 \times I$, where one side is the boundary of $M$ and the other is the octahedron. It has one 3-cell for every triangle in any face of the rectanguloid $M$, one 3-cell for every 1-simplex in any edge of the rectanguloid, and one edge per vertex of the rectanguloid. A blank face has $8$ triangles, the faces of $M'$ contain $2(2 \ast 3T)+2(2 \ast (3T+1))+2(3T \ast (3T+1))=18T^2+30T+4$ triangles. Each edge of a blank face is made of two 1-simplices. The rectanguloid has 4 edges made of $2 \ast 2=4$ 1-simplices, 4 edges made of $2 \ast 3T=6T$ 1-simplices and 4 edges made of $2 \ast (3T+1)=6T+2$ 1-simplices - $48T+24$ edges in total. The rectanguloid also has $8$ vertices, so $M'''$ has $18T^2+30T+4+48T+24+8=18T^2+78T+36$ 3-simplices.

In particular, gluing the the appropriate boundary of $M''$ to the boundary of $M$ will involve adding $O(T^2)$ triangles to it. The gluing is a similar, but simpler process than the one examined in Lemma~\ref{CompGlue}, and will, in particular, take no longer than $O(T^3)$ time. This results in a \gls{genericsurface} contained in the gluing of $M$ and $M''$, which is still homeomorphic to a 3-ball, which realizes the given \gls{DADG}/3-sat formula and whose boundary always has a octahedron triangulation. Taking a triangulation $M'$ of $\overline{X \setminus D^3}$ whose boundary is a octahedron, and gluing it to the ``already glued" $M$ and $M''$, will produce the necessary \gls{genericsurface} in $X$, and will take $O(1)$ time.

\chapter{Realizable DADGs} \label{BHrep}

In this chapter, we will explain which \gls{DADG}s are realizable in which 3-manifolds. This answers an open question left by Li in \cite{Li1}. As we mentioned in section~\ref{DADGsec} (Definition~\ref{Pref}), Li defined two enriched graph structures that describe the intersection graph of \gls{genericsurface}s, ``Daisy Graphs" (DGs) and ``Arrowed Daisy Graphs" (ADGs). The intersection graph of every \gls{genericsurface} has a \gls{DG} structure, which indicates which pairs of ends of edges are \gls{consecutive}. Only the intersection graphs of oriented \gls{genericsurface}s in oriented 3-manifolds have an \gls{ADG} structure, which also indicates which one of the two ends of edges in each pair is the \gls{preferred} one.

Li's work only considered immersions in $S^3$. In \cite{Li1}, Li asked ``which DGs are realizable by immersion into $S^3$?" and ``which DGs are realizable in $S^3$ by an orientable surface?". He created ADGs as a tool to help answer these questions. This led to the question ``which ADGs are realizable in $S^3$ (by an oriented surface)?", which Li left open.

The \gls{DADG}s we defined in \ref{SurfDADG} and \ref{DADG} are a generalization of the ADGs defined by Li. We will answer the question: ``given a 3-manifold $M$, what \gls{DADG}s are realizable in $M$?". The answer depends on the first homology group $H_1(M;\Z)$. In particular, it depends on whether or not $H_1(M;\Z)$ has an element of infinite order. If it does not, then a \gls{DADG} is realizable in $M$ iff it has a grading (Definition~\ref{Grading}).

In the first section of this chapter, we will review the concept of \gls{gradable} \gls{DADG}s. We will explain what makes a \gls{DADG} \gls{gradable} or non-gradable and show how to check if a \gls{DADG} is \gls{gradable} in linear time. In the second section, we will use the proof that a \gls{DADG} that is realizable in $M$ must be \gls{gradable}, and in the third section we will prove the other direction.

In the forth section we will study the other case. Specifically, we will prove that if $M$ is compact and $H_1(M;\Z)$ is infinite then any \gls{DADG} is realizable in $M$. Lastly, in the fifth section, we will show how to enhance the \gls{DADG} structure of the intersection graph even further, so that it encodes even more information about the topology of the surface.

We note that the results of this chapter have been submitted for publication as the article \cite{BH1}. While the two contain the same results, some of the notations we used in \cite{BH1} are different than the ones we used in this thesis. For instance, we define the enhanced structure of the intersection graph in a way that is more similar to Li's original definition, and we use the term \gls{ADG} instead of \gls{DADG}.

\section{The complexity of gradability}

Recall Definitions~\ref{H1} and \ref{Grading}. The aforementioned definition of a \gls{gradable} \gls{DADG} is not very applicable - it only states that a \gls{DADG} is \gls{gradable} iff it has a grading. In this section, we study the question ``what \gls{DADG}s are \gls{gradable}". Specifically, we explain what obstructions may prevent a \gls{DADG} from being \gls{gradable}, and show that it is possible to check whether a given \gls{DADG} is \gls{gradable} in linear time.

\begin{defn} \label{GradeObs}
A ``grade obstructing" loop of an \gls{DADG} is a loop (a path whose ends both lay on the same vertex $v$) with one \gls{preferred} end and one non-\gls{preferred} end at $v$. For example, the loop in Figure~\ref{fig:Figure 0x}A is grade obstructing, while the loop in Figure~\ref{fig:Figure 0x}B is not.
\end{defn}

\begin{figure}
\begin{center}
\includegraphics{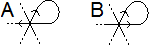}
\caption{A grade obstructing loop and a non grade obstructing loop}
\label{fig:Figure 0x}
\end{center}
\end{figure}

\begin{rem} \label{GradeObsRem}
1) A \gls{gradable} \gls{DADG} cannot have grade obstructing loops, since the grade of such a loop would have to be $a(v)=g(e)=a(v)+1$. 

2) If an \gls{DADG} has no grade obstructing loops, then the sets of \gls{preferred} edges at $v$ and non-\gls{preferred} edges at $v$ are mutually exclusive. This simplifies the following definition.
\end{rem}

\begin{defn} \label{RelGrade}
1) Given a \gls{DADG} $G$ with no grade obstructing loops, and two edges $e$ and $f$ that share a vertex $v$ (which must be a triple value since its degree cannot be 1), define the ``grading difference" $\Delta g(e,v,f)$ to be $1$ if $f$ is \gls{preferred} at $v$ and $e$ is not, $-1$ if it is the other way around, and $0$ if either both $f$ and $g$ are \gls{preferred} or if they are both non-\gls{preferred}.

2) The grading difference of a path $e_0,v_0,e_1,v_1,...,v_{r-1},e_r$ in $G$ is the sum $\sum_{k=1}^r(\Delta g(e_{k-1},v_{k-1},e_k)$.
\end{defn}

\begin{lem} \label{GradeRem}
1) If an \gls{DADG} has a grading $g$, then the grading difference of a path $e_0,v_0,e_1,v_1,...,v_{r-1},e_r$ is equal to $g(e_r)-g(e_0)$.

2) An \gls{DADG} is \gls{gradable} iff it has no grade obstructing loop and, for every pair of edges $e$ and $f$, every path between $e$ and $f$ has the same grading difference.

3) One can check if an \gls{DADG} $G$ is \gls{gradable}, and therefore construct a grading, in linear $O(|E|)$ time where $E$ is the set of $G$'s edges.
\end{lem}

\begin{proof}
1) For a short part $e,v,f$, this follows directly from Definitions~\ref{Grading} and \ref{RelGrade}(1). Induction implies the general case.

2) ($\Leftarrow$): The first part is Remark~\ref{GradeObsRem}, and the second follows from (1).

($\Rightarrow$): For every connected component $G'$ of $G$ (that is not a double circle), do the following: choose one edge $e$ in $G'$ and give it the grade $0$. Next, for every other edge $f$ in $G'$, choose a path $e=e_0,v_0,e_1,...,e_r=f$ and set the grade $g(f)$ of $f$ to be the relative grade of this path. By assumption, this is independent of the path. If $f$ shares a vertex $v$ with another edge $h$, then $e=e_0,v_0,e_1,...,e_r=f,v,h$ is a path from $e$ to $h$, and so $g(h)=\sum_{k=1}^r(\Delta g(e_{k-1},v_{k-1},e_k)+\Delta g(f,v,h)=g(f)+\Delta g(f,v,h)$. This holds for every adjacent pair of edges. In particular, if $v$ is a vertex and $f$ is non-\gls{preferred} at $v$, then, for the number $a(v)=g(f)$, every non-\gls{preferred} edge $h$ at $v$ upholds $g(h)=g(f)+\Delta g(f,v,h)=a(v)$ and every \gls{preferred} edge $h$ at $v$ upholds $g(h)=g(f)+\Delta g(f,v,h)=a(v)+1$, and so $g$ is a grading.

3) It takes $O(|E|)$ time to go over the edges of $G$ and check if any of them is a grade-obstructing loop. If no such loop exists, we will assign each edge $f$ of $G$ a number $g(f)$ which, if the graph is \gls{gradable}, will be a grading. We say that the algorithm ``reached" (resp. ``exhausted") a vertex if it assigned a grading to at least one (resp. all) of the edges of this vertex. We begin by choosing one edge $e$ and grading it $g(e)=0$. For each vertex of $e$, we set $a(v)=g(e)-1=-1$ / $a(v)=g(e)=0$ if $e$ is respectively \gls{preferred} / non-\gls{preferred} at $v$.

Next, we choose a vertex $v$ that the algorithm has reached but has not exhausted (currently, this means that $v$ is one of the vertices of $e$) and go over the edges of $v$. If a \gls{preferred} / non-\gls{preferred} edge $f$ has yet to be graded, then grade it $g(f)=a(v)+1$ / $g(f)=a(v)$ respectively, then look at the other vertex $w$ of $f$. If this is the first time the algorithm reaches $w$, set $a(w)=g(f)-1$ / $a(w)=g(f)$ if $f$ is respectively \gls{preferred} / non-\gls{preferred} at $w$. If the algorithm reached $w$ before, then $a(w)$ has already been set previously. In order for $g$ to be a grading, $w$ must uphold $a(w)=g(f)-1$ / $a(w)=g(f)$, depending on if $f$ is \gls{preferred} at $v$ or not. Check if this equality holds.

If the equality holds, move on to the other edges of $v$ and do the same. Since $v$ has no more than 6 edges, this takes $O(1)$ time. When you have exhausted $v$, move on to another vertex $G$ that the algorithm has reached but has yet to exhaust. Continue like this until either a) you grade an edge $f$ whose ``other vertex" $w$ has already been reached and for which the appropriate equality $a(w)=g(f)-1$ / $a(w)=g(f)$ fails, or b) if you have not reached such an edge but there are no more vertices that the algorithm reached but has yet to exhaust.

If you stop because of (a) then $G$ is not \gls{gradable}. In order to see this, notice that if you reached a vertex $v$ via an edge $e_v$, and then you grade another edge $f$ at $v$, then $g(f)=\Delta (g_v,v,f)+g(e_v)$. This can be proven on a case per case basis. For instance, if both $f$ and $e_v$ are \gls{preferred} at $v$, then $\Delta (g_v,v,f)=0$ and according to the above $a(v)=g(e_v)-1$ and $g(f)=a(v)+1=g(e_v)=\Delta (g_v,v,f)+g(e_v)$ as required. Induction implies that every $f$ that the algorithm grades has a path $e=e_0,v_0,e_1,...,e_r=f$ such that $g(f)$ is equal to the grading difference of this path. Indeed, it holds for $e$ itself, and if you assume that it holds for every edge you graded before, and in particular for $e_v$, then $g(e_v)$ is equal to the grading difference of the path $e=e_0,v_0,e_1,...,e_r=g_v$ and $g(f)=g(e_v)+(g(f)-g(e_v))=\sum_{k=1}^r(\Delta g(e_{k-1},v_{k-1},e_k)+\Delta (g_v,v,f)$ - the grading difference of the path $e=e_0,v_0,e_1,...,e_r,v,f$.

Now, if you grade an edge $f$ whose other vertex $w$ has already been reached, and the appropriate equality $a(w)=g(f)-1$ / $a(w)=g(f)$ fails, then similar considerations imply that $g(f) \neq \Delta (g_w,w,f)+g(e_w)$. We have proven that there is one path from $e$ to $f$ whose grading difference is equal to $g(f)$, but there is another such path $e=h_0,w_0,h_1,...,w_{r-1},h_r=e_w,w,f$, for which $g(e_w)=\sum_{k=1}^r(\Delta g(h_{k-1},w_{k-1},h_k)+\Delta (g_v,v,f)$ but $g(f)=g(e_w)+(g(f)-g(e_w)) \neq \sum_{k=1}^r(\Delta g(h_{k-1},w_{k-1},h_k)+\Delta (g_w,w,f)$. Since these two paths have different grading differences, (2) implies that $G$ is not \gls{gradable}.

If the algorithm stopped because of (b), then it provided a grading $g(f)$ for every edge $f$ in the connected component of $G$ that contains $e$. Since the equality never failed, every vertex $v$ and every \gls{preferred} / non-\gls{preferred} edge $f$ at $v$ upholds $a(v)=g(f)-1$ / $a(v)=g(f)$. This means that $g$ is indeed a grading of this connected component. If there are any vertices left that the algorithm has not reached yet, then they belong to a different connected component. Choose a new ungraded edge $e$ and grade it $g(e)=0$, and then proceed to grade its connected component. Eventually, either you will reach stop condition (a), meaning that $G$ is not \gls{gradable}, or you will exhaust all the vertices of $G$, in which case you finished grading all of $G$.

In total, the algorithm went over every edge $f$ of $G$, determined $g(f)$, and either determined $a(w)$ for one or both of its vertices, or checked if it upheld the equity $a(w)=g(f)-1$ / $a(w)=g(f)$. This takes $O(|E|)$ time.
\end{proof}

\begin{rem} \label{tree}
If the graph part of an \gls{DADG} $G$ is a forest, then the algorithm will never reach the stop condition (a), and so $G$ is \gls{gradable}.
\end{rem}

\section{Gradings and winding numbers}

In the next two subsections, we will answer the following question: given a 3-manifold $M$ for which $H_1(M;Z)$ contains only elements of finite order (such a group is called periodic or torsion), what \gls{DADG}s can be realized via a \gls{genericsurface} in $M$. In particular, we will prove the following theorem:

\begin{thm} \label{ADGThm}
Let $M$ be an oriented 3-manifold for which $H(M;\Z)$ is periodic.

1) If $M$ has no boundary, then a \gls{DADG} $G$ can be realized as the intersection graph of an oriented \gls{genericsurface} $S$ in $M$ iff $G$ is \gls{gradable} and has no DB values.

2) If $M$ has a boundary, then a \gls{DADG} $G$ can be realized as the intersection graph of an oriented \gls{genericsurface} $S$ in $M$ iff $G$ is \gls{gradable}.
\end{thm}

\begin{res} \label{GradeRes}
In \cite{Li1}, Li showed that a \gls{DG} with no DB values or branch values is realizable iff any arc in it is composed of an even number of edges. Theorem~\ref{ADGThm} implies a generalization of this -  a general \gls{DG} is realizable via an orientable \gls{genericsurface} iff every \underline{closed} arc is composed of an even number of edges.
\end{res}

\begin{proof} [Proof of Result~\ref{GradeRes}]
If a \gls{DG} is realizable via an orientable \gls{genericsurface}, then any orientation of the surface gives the \gls{DG} an \gls{ADG} structure (arrows). Choosing a direction of progress for each double arc turns it into a \gls{DADG}. This \gls{DADG} is realizable and therefore \gls{gradable}. The grading of each subsequent edge on an arc will have a different parity than the grading of the previous edge and, in particular, closed arcs must have an even number of edges on them.

On the other hand, given a \gls{DG} that upholds this condition (every closed arc must have an even number of edges), it is possible to give the \gls{DG} a ``short grading" - number the edges with only 0 and 1 in such a way that \gls{consecutive} edges have different numbers. Clearly, the only obstruction to this is the existence of closed arcs with an odd number of edges. Now, one \gls{EoE} in every \gls{consecutive} pair will belong to an edge whose grade is 1, and the other will belong to an edge whose grade is 0. You can give the \gls{DG} an \gls{ADG} structure that matches this grading by choosing the former half-edges to be \gls{preferred}. This graded \gls{ADG} is realizable, and in particular, the underlying \gls{DG} is realizable via an orientable surface.
\end{proof}

In the remainder of this section we will prove the ``only if" direction of the items of Theorem~\ref{ADGThm}. The ``if" direction will be proven in the next section. One part of the ``only if direction" is trivial - a \gls{genericsurface} in a bounderyless 3-manifold cannot have DB values. In order to prove the other part, that the intersection graph of a \gls{genericsurface} is a \gls{gradable} \gls{DADG}, we use 3-dimensional winding numbers:

\begin{defn} \label{FaceBody}
Let $S$ be a \gls{genericsurface} in a 3-manifold $M$.

1) A face (resp. body) of $S$ is a connected component of $S \setminus X(i)$ (resp. $M \setminus S$).

2) Each face $V$ is an embedded surface in $M$, and there is a body on each side of it. We say these two bodies are adjacent (via $V$). A priori, it is possible that these two bodies are in fact two parts of the same body, and even that $V$ is a one-sided surface. In these cases, this body will be self adjacent, but this does not happen in any of the cases we are interested in.

3) If $S$ has an orientation, then each face $V$ is two sided, and the arrows on the face point towards one of its two sides. We say that the body on the side that the arrows point toward is ``greater" (via $V$) than the body on the other side of $V$.

4) A choice of ``winding numbers" for $S$ is a choice of an integer $w(U) \in \Z$, for every body $U$ of $S$, such that if $U_1$ and $U_2$ are adjacent, and $U_1$ the greater of the two, then $g(U_1)=g(U_2)+1$.
\end{defn}

\begin{lem} \label{Winding}
If $M$ is a connected and orientable 3-manifold, $H_1(M;\Z)$ is periodic, and $S$ is an oriented \gls{genericsurface}, then $S$ has a choice of winding numbers.
\end{lem}

\begin{proof}
Pick one body $U_0$ to be ``the exterior" of the surface and set $w(U_0)=0$. Next, define the winding numbers for every other body $U$ as follows:

Take a smooth path from $U_0$ to $U$ that is in general position to $S$ (it intersects $S$ only at faces, and does so transversally), and set $w(U)$ to be the signed number of times it crosses $S$, the number of times it intersects it in the direction of the orientation minus the times it crosses it against the orientation. This is well defined, since any two such paths $\alpha$ and $\beta$ must give the same number. Otherwise, the composition $\beta^{-1} \ast \alpha$ is a 1-cycle whose intersection number with the 2-cycle represented by $S$ is non-zero. This implies that this 1-cycle is of infinite order in $H_1(M;\Z)$ - contradicting the fact that this group is periodic.

It is also clear that if $U_1$ and $U_2$ are adjacent and $U_1$ is the greater of the pair, then $g(U_1)=g(U_2)+1$.
\end{proof}

\begin{rems} \label{WNRems}
1) It is clear that two different choices of ``winding numbers" for $S$ will differ by a constant, and that the one we created is unique in satisfying $w(U_0)=0$.

2) We can do a similar process on a loop $\gamma$ in $\R^2$ instead of a surface in a 3-manifold. If we choose the component $U_0$ of $\R^2 \setminus \gamma$ to be the actual exterior, then this will produce the usual winding numbers - $w(U)$ will be the number of times $\gamma$ winds around a point in $U$. 
\end{rems}

We will use the winding numbers to induce a grading in the following manner: the neighborhood of a double value includes 4 bodies, with the possibility that some of them are, in fact, different parts of the same body. If the surface has a orientation and winding numbers, then there is a number $g$ such that two of these bodies have the WN $g$, one has the WN $g+1$ and one has the WN $g-1$. Figure~\ref{fig:Figure 1x}A depicts this:

\begin{figure}
\begin{center}
\includegraphics{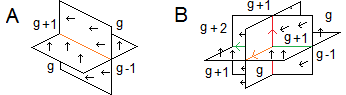}
\caption{The winding number of bodies around an edge of $X(i)$ and a triple value}
\label{fig:Figure 1x}
\end{center}
\end{figure}

Due to continuity, this will be the same value $g$ for all the double values on the same edge (or double circle). We call this number the grading of the edge, and name the grading of an edge $e$ $g(e)$. This is indeed a grading in the sense of Definition~\ref{Grading}. In order to prove this, we need to show that at every triple value of the surface all the \gls{preferred} ends of edges have the same grading, which is greater by 1 than the grading of all the non-\gls{preferred} ones. This can be seen in Figure~\ref{fig:Figure 1x}B, which depicts the winding numbers of the bodies around an arbitrary triple value. Indeed, you can see that the \gls{preferred} ends of edges - the ones going up, left and outwards (toward the reader) have the grading $g+1$, while the other edges have the grading $g$. This proves the ``only if" direction of Theorem~\ref{ADGThm}.

\section{The cross surface of a DADG}

In order to prove the ``if" direction of Theorem~\ref{ADGThm} we will first prove a partial result. We will limit the discussion to connected \gls{DADG}s with no DB values.

\begin{lem} \label{NoDB}
Every \textbf{connected}, \textbf{\gls{gradable}} \gls{DADG} $G$ \textbf{without DB values} can be realized via a closed \gls{genericsurface} $S$ in $S^3$.
\end{lem}

\begin{rem} \label{Connected1}
It may be assumed that $S$ is connected. Otherwise, one of its components will contain the connected intersection graph, and you may delete the other components.
\end{rem}

We begin with the unique case where the \gls{DADG} is a double circle. The \gls{genericsurface} from Figure~\ref{fig:Figure 2x}A has a single double circle as its intersection graph. It is the surface of revolution of the curve from Figure~\ref{fig:Figure 2x}B around the blue axis. Both figures have indication for the orientation. The intersection graph will be the revolution of the orange dot where the curve intersects itself, and will thus be a circle. The underlying surface is clearly a sphere.

\begin{figure}
\begin{center}
\includegraphics{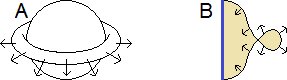}
\caption{A surface whose intersection graph is a double circle}
\label{fig:Figure 2x}
\end{center}
\end{figure}

Any other connected \gls{DADG} is a ``graph \gls{DADG}" - it will have no double circles. In this case, we begin by constructing a part of the matching \gls{genericsurface} - the regular neighborhood of the intersection graph. Li defined something similar in \cite{Li1} (p.3723, figure 2) which he called a ``cross-surface", and we will use the same notation.

\begin{defn} \label{CrossSurf}
Given a \gls{DADG} $G$ that has no DB values and no double circle, a ``cross-surface" $X_G$ of $G$ is a shape in $S^3$ that is built via the following two steps:

1) For every triple value $TV_k$ of $G$, embed a copy of Figure~\ref{fig:Figure 3x}A in $S^3$. This shape is called the ``vertex neighborhood" $TVN_k$ of $TV_k$. The triple value in the vertex neighborhood $TVN_k$ will be the $k$th triple value of the surface $S$. Similarly, for every branch value $BV_k$ of $G$, embed a vertex neighborhood $BVN_k$ that looks like Figure~\ref{fig:Figure 3x}B in $S^3$. Make sure that the different vertex neighborhoods will be pairwise disjoint.

\begin{figure}
\begin{center}
\includegraphics{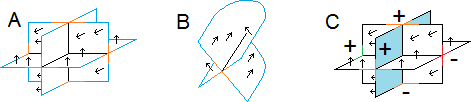}
\caption{The vertices neighborhoods and their gluing zones}
\label{fig:Figure 3x}
\end{center}
\end{figure}

Recall Definition~\ref{Pref}(2) of the ends, and the length, of an edge of the intersection graph of a \gls{genericsurface}. In each vertex neighborhood one can see the 1 or 6 ends-of-edges of the vertex that it contains. Each \gls{EoE} meets the boundary of the vertex neighborhood at one point. Given such an intersection point, we refer to its regular neighborhood \textbf{inside the boundary of the vertex neighborhood} as its ``gluing zone". In Figure~\ref{fig:Figure 3x}A and B, we colored the gluing zones in orange and the rest of the boundary of the vertex neighborhoods in blue.

Recall that $BV_k$, the $k$th branch value of $G$, is a pair $(r,s)$ where $r$ is an integer and $s$ is binary. For the \gls{DADG} of $S$ to be equal to $G$, the $s$th end of the $r$th edge of the intersection graph of $S$ must end in the $k$th branch value of the surface - the branch value inside the vertex neighborhood $BVN_k$ (the 0th end of the $r$th edge is the beginning of the edge and the 1st end is the ending). In order to reflect this, we index the gluing zone of this vertex neighborhood as the ``$(r,s)$th gluing zone".

Similarly, $TV_k$, the $k$th triple value of $G$, has the form $(((r_1,s_1),(r_2,s_2)),$ $((r_3,s_3),(r_4,s_4)),((r_5,s_5),(r_6,s_6)))$ where each $r_i$ is an integer and each $s_i$ is binary. For the \gls{DADG} of $S$ to be equal to $G$, the $s_i$th end of the $r_i$th edge of the intersection graph of $S$ must end in the triple value inside the vertex neighborhood $BVN_k$. In order to reflect this, for each $i=1,...,6$, we index one of the gluing zones of this vertex neighborhood as the ``$(r_i,s_i)$th gluing zone".

You must choose which gluing zone corresponds to which $i=1,...,6$ in a way that reflects the structure of $G$. Firstly, recall that for every $l=1,2,3$, the ends of edges $(r_{2l-1},s_{2l-1})$ and $(r_{2l},s_{2l})$ represent two \gls{consecutive} ends of edges. In order to reflect this, make sure that the $(r_{2l-1},s_{2l-1})$th and $(r_{2l},s_{2l})$th gluing zones that are on opposite sides of the vertex neighborhood, such as the zones marked red and green in Figure~\ref{fig:Figure 3x}C.

Additionally, recall that $(r_{2l},s_{2l})$ is the \gls{preferred} \gls{EoE} among the two. Each pair of ``opposite sides" gluing zones is separated by one of the 3 intersecting surface sheets at $TVN_k$. For instance, in Figure~\ref{fig:Figure 3x}C, the blue surface separates the red and green gluing zones. The orientation on this surface points toward one of the gluing zones. In this case - the green one. In order to reflect the fact that the $(r_{2l},s_{2l})$th \gls{EoE} is \gls{preferred}, it must correspond to the gluing zone that the orientation points toward. The other gluing zone, in our case the red one, will be the $(r_{2l-1},s_{2l-1})$th gluing.

In Figure~\ref{fig:Figure 3x}C we also marked the gluing zones toward which the orientation points with ``+", and the other gluing zones with ``-". Another way to phrase the last requirement is that the gluing zones marked ``+" (resp. ``-") will correspond to the ends of edges $(r_i,s_i)$ for which $i$ is even (resp. odd).

\begin{figure}
\begin{center}
\includegraphics{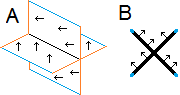}
\caption{The X-bundle of an edge and a cross-section of it}
\label{fig:Figure 4x}
\end{center}
\end{figure}

2) The previous step identified each \gls{EoE} of $G$ with a unique gluing zone of the cross surface, and thus with a unique \gls{EoE} of the surface $S$. In this step we will add the length of the edge to the cross surface. Assume that $G$ has $|E|$ edges. For each $r=0,...,|E|-1$, embed a matching copy of the shape in Figure~\ref{fig:Figure 4x}A into $S^3$.

This shape is a bundle over a closed interval, whose fiber looks like the ``X" in Figure~\ref{fig:Figure 4x}B. We therefore call this shape ``the $r$th X-bundle of $G$". The piece of double line that goes through it will serve as the length of the $r$th edge of the intersection graph of the cross surface. In order to do this, the embedding of the $r$th X-bundle must adhere to the following rules:

(a) The boundary of each X-bundle is composed of two parts - the fibers at the ends of the interval, colored orange, and the (union of the) ends of all the fibers, colored blue. Make sure to embed the $r$th X-bundle so that one end fiber coincides with the $(r,0)$th gluing zone and the other coincides with the $(r,1)$th gluing zone. Additionally, ensure that the ``length" of the X-bundle (the X-bundle sans the end fiber) is disjoint from the vertex neighborhoods, and that X-bundles of different edges do not touch one another.

(b) Note that both the vertex neighborhoods and the X-bundles have arrows on them, which represent orientations. When you embed the X-bundles, these orientations on them must match, as in Figure~\ref{fig:Figure 5x}A, and unlike Figure~\ref{fig:Figure 5x}B. This way they will merge into a continuous orientation on the entire cross surface.

The resulting shape is the cross-surface. It is similar to a \gls{genericsurface} but it has a boundary - the union of all the ``blue parts" of the boundaries of the vertex neighborhoods and the X-bundles.

In order to define the \gls{DADG} structure of the cross surface, you must choose a direction of progress on each edge. Do so in such a way that, along the length of the $r$th end, inside the $r$th X-bundle, the direction of progress points from the end fiber that is glued to the $(r,0)$th gluing zone and towards the end fiber that is glued to the $(r,1)$th gluing zone. This way, the way we indexed the gluing zones implies that: (1) for every $k$, if the $(r,s)$th \gls{EoE} in $G$ resides on the $k$th branch value / triple value, then so does the $(r,s)$th \gls{EoE} of the intersection graph, (2) if the $(r,s)$th and the $(r',s')$th ends of edge are \gls{consecutive} in $G$, then the same holds for the intersection graph, and (3) the same one of these ends of edge is \gls{preferred} at $G$ and at the intersection graph. This implies that the \gls{DADG} structure of the intersection graph is equal to $G$.

\begin{figure}
\begin{center}
\includegraphics{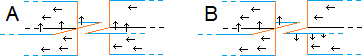}
\caption{The gluing must preserve the orientation}
\label{fig:Figure 5x}
\end{center}
\end{figure}
\end{defn}

The boundary of the cross-surface is the union of many embedded intervals in $S^3$ - the ``blue parts" of the boundaries of the vertex neighborhoods and the X-bundles. Since each end of every interval coincides with an end of one other interval, and the intervals do not otherwise intersect, their union is an embedded compact 1-manifold in $S^3$. The cross-surface induces an orientation on this 1-manifold, the usual orientation that an oriented manifold induces on its boundary. It is depicted in the left part of Figure~\ref{fig:Figure 6x}A. 

We will show that the boundary of the cross surface of a connected \gls{DADG} $G$ is also the oriented boundary of an embedded surface which is disjointed from the cross surface. It follows that the union of the cross surface and the embedded surface, with the orientation on the embedded surface reversed, will be a closed and oriented \gls{genericsurface} whose intersection graph will be isomorphic to $G$. This will prove Lemma~\ref{NoDB}.

\begin{figure}
\begin{center}
\includegraphics{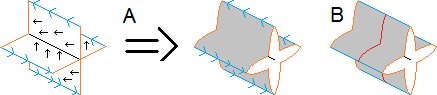}
\caption{Thickening the cross-surface into a handle body, and the handle body's meridians}
\label{fig:Figure 6x}
\end{center}
\end{figure}

In order to prove that such an embedded surface exists, we begin by ``thickening" the cross surface as in Figure~\ref{fig:Figure 6x}A. Figure~\ref{fig:Figure 6x}A only shows how to do this to an X-bundle, but you can similarly do this for all the vertex neighborhoods. This results in a handle body $H$ in $S^3$ and our 1-cycle is on its boundary. It will suffice to prove that the 1-cycle is the boundary of some embedded surface in the complement of $H$. This happens iff is the cycle is a ``boundary" in the homological sense - the is equal to $0$ in $H_1(\overline{S^3 \setminus H};\Z)$.

Given any loop $\gamma$ in the intersection graph, we define a functional $f_\gamma:H_1(\overline{S^3 \setminus H}$ $;\Z) \to \Z$ such that $f_\gamma(c)$ is the linking number of $\gamma$ and a representative of $c$. It is well-defined, since cycles in $\overline{S^3 \setminus H}$ are disjoint from $\gamma$, and since the linking number of $\gamma$ with any boundary in $H_1(\overline{S^3 \setminus H})$ is $0$, as the boundary bounds a surface in $\overline{S^3 \setminus H}$ which is disjoint from $\gamma$.

In case the genus of $G$, and therefore of the intersection graph and of $H$, is $n$, then the intersection graph has $n$ simple cycles $C_1,...,C_n$, such that each cycle $C_i$ contains an edge $e_i$ that is not contained in any of the other cycles. For every cycle $C_i$, we take a small meridian $m_i$ around the edge $C_i$ (as depicted in red in Figure~\ref{fig:Figure 6x}B). It follows that $f_{c_i}([m_j])=\delta_{ij}$ where $\delta$ is the Kronecker delta function. Additionally, since $\overline{S^3 \setminus H}$ is the complement of an $n$-handle body, $H_1(\overline{S^3 \setminus H}) \equiv \Z^n$. We will prove that:

\begin{lem}
These meridians form a base of $H_1(\overline{S^3 \setminus H})$.
\end{lem}

\begin{proof}
First, we show that the meridians are independent. This is because a boundary in $\overline{S^3 \setminus H}$ would have $0$ as the linking number with every $c_i$, but the linking number of a non-trivial combination $x=\sum a_i[m_i]$ with any $c_j$ will be $a_j$, and for some $j$, $a_j \neq 0$. Second, notice that this implies that $N=Span_{\Z}\{[m_1],...,[m_n]\}$ is a maximal lattice in $H_1(\overline{S^3 \setminus H}) \equiv \Z^n$, and therefore has a finite index.

Third, had $N$ been a strict subgroup of $H_1(\overline{S^3 \setminus H};\Z)$, then there would be an element $y \in H_1(\overline{S^3 \setminus H};\Z) \setminus N$. Define $b_i=lk(y,c_i)$ and $y'=y-\sum_{i=1}^n b_i[m_i]$. $y'$ will have $0$ as the linking number with every $c_i$, but it will not belong to $N$. The finite index of $N$ implies that $ky' \in N$ for some $k$, but $lk(ky',c_i)=k \dot 0=0$ for all $i$, and thus $ky'=0$. This means that $y'$ is a non-zero element of finite order in $H_1(\overline{S^3 \setminus H};\Z) \equiv \Z^n$, but no such element exists.
\end{proof}

\begin{lem} \label{partial}
Let $G$ be a connected \gls{DADG} that has no DB values, is not a double circle, and is \gls{gradable}. Then the linking number of the boundary of its cross surface with any simple cycle in the intersection graph of this cross surface is $0$.
\end{lem}

\begin{proof}
Let $C$ be a simple cycle in the intersection graph. It is composed of distinct vertices and edges $e_0,v_1,e_1,v_2,...,v_n,e_n=e_0$. Each $v_i$ is a triple value, since it is not a degree-1 vertex. We will perturb $C$ until it is in general position to the cross surface and calculate the intersection number of the ``moved $C$" with the cross-surface. This will be equal to the linking number of $C$ and the boundary of the cross-surface.

We begin by pushing each edge $e_i$ away from its matching X-bundle in a direction that agrees with the orientation on both of the surfaces that intersect in this X-bundle, as in Figure~\ref{fig:Figure 7x}. 

\begin{figure}
\begin{center}
\includegraphics{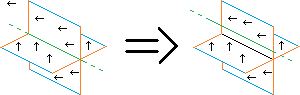}
\caption{Moving the intersection graph away from the cross surface}
\label{fig:Figure 7x}
\end{center}
\end{figure}

We need to continue this ``pushing" at the vertex neighborhood of each $v_i$. Figures~\ref{fig:Figure 8x}, \ref{fig:Figure 9x} and \ref{fig:Figure 10x} demonstrate how to push away the half-edges from their original position. The half-edges we push are colored green, and the arrows on them indicate the direction of the cycle - the half-edge whose arrow points toward (resp. away from) the triple value is a part of $e_{i-1}$ (resp. $e_i$). Continuity dictates that we must always push in the direction indicated by the orientations on the surface as we did in Figure~\ref{fig:Figure 7x}, and Figures~\ref{fig:Figure 8x}, \ref{fig:Figure 9x} and \ref{fig:Figure 10x} indeed comply with this.

Each of the three figures depicts a different situation with regards to which of the two half-edges, if any, is \gls{preferred} at $v_i$. Figure~\ref{fig:Figure 8x} depicts the case where both the half-edges are \gls{preferred}. In this case, after being pushed away from the cross-surface, $C$ will not intersect the cross surface at the neighborhood of $v_i$.

\begin{figure}
\begin{center}
\includegraphics{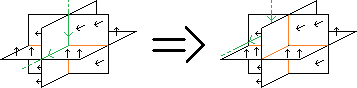}
\caption{Moving the intersection graph away from a triple value, when both sides are \gls{preferred}}
\label{fig:Figure 8x}
\end{center}
\end{figure}

Figure~\ref{fig:Figure 9x} depicts the case where the half-edge that is a part of $e_{i-1}$, the one entering the triple value, is not \gls{preferred}, and the half-edge that is a part of $e_i$, the one exiting the triple value, is \gls{preferred}. In this case, after being pushed away from the cross-surface, $C$ will intersect the cross-surface once, and it will do so agreeing with the direction of the orientation on the surface (that's why there is a little $+1$ next to the intersection). 

\begin{figure}
\begin{center}
\includegraphics{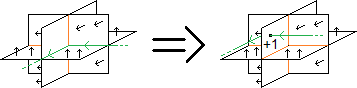}
\caption{Moving the intersection graph away from a triple value, when only one side is \gls{preferred}}
\label{fig:Figure 9x}
\end{center}
\end{figure}

Figure~\ref{fig:Figure 9x} depicts the case where the two half-edges are not \gls{consecutive}, but even if they were, the same thing would happen - $C$ would intersect the cross-surface once, in agreement with the orientation. The only difference would be that the half-edge that was exiting $v_i$ would have continued leftwards instead of turning outwards towards the reader. Furthermore, had the half-edge coming from $e_{i-1}$ been \gls{preferred} and the one coming from $e_i$ had not, then the pushing would still occur as in Figure~\ref{fig:Figure 9x}, except that the arrows on the green line would point the other way. In this case, $C$ would still intersect the cross surface once after the pushing, but it would be against the direction on the orientation.

Lastly, Figure~\ref{fig:Figure 10x} depicts the case in which both half-edges are not \gls{preferred}. In this case, after being pushed away from the cross-surface, $C$ will intersect the cross-surface twice in the neighborhood of $v_i$. One intersection, marked $+1$, is in the direction of the orientation, and the other intersection, marked $-1$, is against it.

\begin{figure}
\begin{center}
\includegraphics{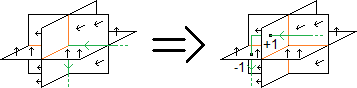}
\caption{Moving the intersection graph away from a triple value, when both sides are non-\gls{preferred}}
\label{fig:Figure 10x}
\end{center}
\end{figure}

Let $G$ be a grading of the intersection graph. Since $e_{i-1}$ and $e_i$ share a vertex, the difference between their grading is at most $1$. If $g(e_i)-g(e_{i-1})=1$ (resp. $-1$), then $e_i$ (resp $e_{i-1}$) is \gls{preferred} and $e_{i-1}$ (resp. $e_i$) is not. We just showed that in this case the signed number of intersections between the ``pushed away" $C$ and the cross-surface is $1$ (resp. $-1$). If $g(e_i)-g(e_{i-1})=0$ then either both $e_i$ and $e_{i-1}$ are \gls{preferred}, in which cases $C$ does not intersect the cross-surface around $v_i$, or they are both non-\gls{preferred}, in which case they intersect once with and once against the orientation.

In all cases, the signed number of intersections between the pushed $C$ and the cross-surface around $v_i$ is equal to $g(e_i)-g(e_{i-1})$. The pushed $C$ does not intersect the cross-surface anywhere else, and so their intersection number is $\sum_{i=1}^n(g(e_i)-g(e_{i-1}))=g(e_n)-g(e_0)=0$. Since $C$ did not cross the the boundary of the cross surface during the pushing, this ($0$) is equal to the linking number of $C$ and the boundary.
\end{proof}

Having proven Lemmas~\ref{partial} and \ref{NoDB}, we can now prove the ``if" direction of the items of Theorem~\ref{ADGThm}:

\begin{proof}
1) Each connected component $G_k$ of $G$ is \gls{gradable} and lacks DB values, and thus has a realizing surface $S_k$ in $S^3$. Remove a point from $S^3 \setminus S_k$ in order to regard $S_k$ as a surface in $\R^3$, and embed these copies of $\R^3$ as disjoint balls in the interior of $M$.

2) If $G$ has no DB values the proof of (1) holds. Otherwise, define a new \gls{DADG} $G'$ in which each DB value of $G$ is replaced with a branch value. Realize $G'$, via (1), with a closed \gls{genericsurface} $S$ for which $F$ is connected.

Take a small ball around each of the branch values that replaces a DB value of $G$, as in Figure~\ref{fig:Figure 11x}A. Figure~\ref{fig:Figure 11x}B depicts the intersection of the surface with the boundary of the ball. It is an ``8-figure" as in Figure~\ref{fig:Figure 11x}C, and the orange dot (the intersection in the 8-figure) is the intersection of the boundary with the intersection graph. If you remove this ball from $S^3$, then instead of ending at the branch value, the edge will end at the orange dot in the 8-figure, which will become a DB value. It follows that after removing all these balls, the intersection graph will be an \gls{DADG} isomorphic to $G$.

\begin{figure}
\begin{center}
\includegraphics{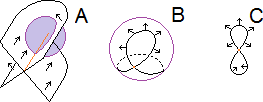}
\caption{Turning a branch value into a DB value}
\label{fig:Figure 11x}
\end{center}
\end{figure}

The \gls{genericsurface} now lays in $S^3$ minus some number of balls. Choose one spherical boundary component and connect it via a path to each of the other ones. Make sure that the path is in general position to the \gls{genericsurface} - it may intersect it only at faces and will do so transversally. Thicken these paths into narrow 1-handles and remove them from the 3-manifold. This may remove some disc from the surface, but will not effect its intersection graph. You now have a \gls{genericsurface} that realizes $G$ in $D^3$. Remove a point from the boundary of $D^3$, making it diffeomorphic to the closed half space $\{(x,y,z) \in \R^3|z \geq 0\}$ which can be properly embedded in any 3-manifold with a boundary. This finishes the proof.
\end{proof}

\begin{rem} \label{Connected2}
If needed, you can make sure that the underlying surface $F$ is connected. This involves modifying the surface in two ways.

a) You can modify the proof of item (1) to produce a connected surface $S$. Begin by assuming that each $S_k$ is connected via Remark~\ref{Connected1}. Pick a face $v_k$ in each $S_k$. The orientation on $v_k$ points towards a body $U_k$. When you remove a point from $S^3$, make sure you remove it from $U_k$. This way, $U_k$ (minus a point) becomes the exterior body of $S_k \subseteq \R^3$. When you embed the copies of $\R^3$ in $M$, the orientation on all $v_k$s will point towards the same connected component of $M \setminus \bigcup i_k(F_k)$. You may connect each $V_k$ to $V_{k+1}$ with a handle going through this component as in Figure~\ref{fig:Figure 14x} (ignore the letters ``A" and ``B" in the drawing). This connects the $i_k$s without sacrificing the orientation or changing the intersection graph.

In item (2) you take a surface from item (1) and modify it. It is clear that none of these modifications can disconnect the surface, and so (2) may also produce a connected surface.

b) If $S$ is connected but $F$ has more than one connected component, then the images of some pair of connected components must intersect generically at a double line. This is depicted in the left part of Figure~\ref{fig:Figure 12x}, where the vertical surface comes from one connected component of $F$ and the horizontal comes from another. Connect them via a handle in an orientation preserving way, as in the right part of Figure~\ref{fig:Figure 12x}, thereby decreasing the number of connected components of $F$. Continue in this manner until $F$ is connected. 
\end{rem}

\begin{figure}
\begin{center}
\includegraphics{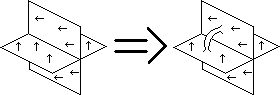}
\caption{Turning a disconnected surface into a connected one}
\label{fig:Figure 12x}
\end{center}
\end{figure}

\section{Compact 3-manifolds with an infinite homology group} \label{InfHom}

In this section, we deal with 3-manifolds whose first homology group contains an element of infinite order.

\begin{thm} \label{ADGThm2}
If $M$ is an oriented, compact and boundaryless 3-manifold with an infinite first homology group, then any \gls{DADG} $G$ with no DB values can be realized as the intersection graph of an oriented \gls{genericsurface} in $M$. If $M$ has a boundary then any \gls{DADG} $G$ can be realized in $M$.
\end{thm}

The proof relies on two lemmas:

\begin{lem} \label{Once}
$M$ has a connected, compact, oriented and properly embedded surface $\Sigma \subseteq M$ that is non-dividing ($M \setminus \Sigma$ is connected).
\end{lem}

\begin{proof}
$H_2(M;\Z)$ is generated by 2-cycles of the form $[\Sigma]$ where $\Sigma \subseteq M$ is a connected, compact, oriented and properly embedded surface. If the statement of the lemma is false, then each such surface divides $M$ into two connected components and will therefore be a boundary in $H_2(M;\Z)$. This implies that $H_2(M;\Z) \equiv \{0\}$. According to Poincar\'{e}'s duality, \\$\{0\} \equiv H_2(M;\Z)/Tor(H_2(M;\Z)) \equiv H_1(M;\Z)/Tor(H_1(M;\Z))$. This implies that every element of $H_1(M;\Z)$ is of finite order, contradicting the assumption.
\end{proof}

\begin{lem} \label{OneBody}
If $G$ is \gls{gradable}, then there is a \gls{genericsurface} $S$, which realizes $G$, and for which $M \setminus S$ is connected (equivalently, $S$ has only one body).
\end{lem}

\begin{proof}
Take the \gls{genericsurface} $\Sigma$ from Lemma~\ref{Once}, and a subset $M' \subseteq M$ that is disjoint from $\Sigma$ and is homomorphic to a half-space (if $M$ has a boundary) or to $\R^3$ (if it does not). According to Theorem~\ref{ADGThm}, there is a \gls{genericsurface} $S'$ in $M'$ which realizes $G$. Connect some face $V$ of the \gls{genericsurface} to $\Sigma$ with a handle, as in Figure~\ref{fig:Figure 13x} (the handle does not intersect $\Sigma$ or $S'$). If needed, reverse the orientation of $\Sigma$ so that the resulting surface will be continuously oriented.

\begin{figure}
\begin{center}
\includegraphics{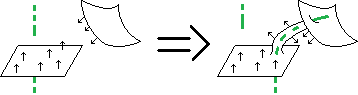}
\caption{Giving the surface a face that has the same body on both sides}
\label{fig:Figure 13x}
\end{center}
\end{figure}

You now have a new \gls{genericsurface} $S$ in $M$ whose intersection graph is still isomorphic to $G'$. Since $\Sigma$ was non-dividing, the connected sum of $V$ and $\sigma$ is a face $S$ that has the same body $A$ on both sides (as indicated by the green path which does not intersect the surface in Figure~\ref{fig:Figure 13x}). If this is $S$'s only body then you are done. If not, you can decrease the number of bodies as follows:

Let $B$ be another body of $S$ that is adjacent to $A$. Connect the face $W$ which separates $A$ and $B$ to the face $V \# \Sigma$ with a path that goes through $A$, and does not intersect $S$ except at the ends of the path. Since $V \# \Sigma$ has $A$ on both sides, you can approach it from either side. If the arrows on $W$ points toward $A$ (resp. $B$), make sure the path enters $V \# \Sigma$ from the direction the arrows point towards (resp. point away from). Next, attach the faces $V$ and $W$ with a handle that runs along this path. Figure~\ref{fig:Figure 14x} depicts the case there the arrows on $W$ point towards $A$. Reverse the direction of all arrows to get the other case.

\begin{figure}
\begin{center}
\includegraphics{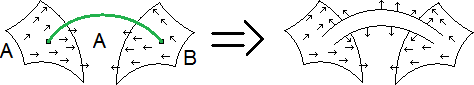}
\caption{Reducing the number of bodies}
\label{fig:Figure 14x}
\end{center}
\end{figure}

The resulting \gls{genericsurface} has one body less than $S$ since $A$ and $B$ have merged. It still realizes $G$ and has a face with the same body on both sides. Repeat this process until you get a surface with only one body.
\end{proof}

We will now prove Theorem~\ref{ADGThm2}:

\begin{proof}
Let $H$ be the graph part of $G$ - $G$ without the double circles. We use induction on the genus of $H$. If the genus is $0$, then $G$ is the union of a forest with some double circles, and Remark~\ref{tree} implies that it is \gls{gradable} and the theorem follows from Lemma~\ref{OneBody}. If the genus of $H$ is positive, pick an edge $e \in H$ such that $H \setminus \{e\}$ has a smaller genus. This means that removing $e$ does not divide the connected component of $H$ that contains $e$. Note that both ends of $e$ are on triple values, since branch values and DB values are of degree 1 and removing their single edge divides the graph.

Define a new \gls{DADG} $G'$ in the following manner: start with a copy of $G$ and cut the edge $e$ in the middle. Instead of $e$ you will get two ``new edges" $e_1$ and $e_2$. Each $e_i$ has one end on a new branch value while the other end ``replaces" one of the ends of $e$ - it enters the triple value that the said end of $e$ was on, and it retains the \gls{DADG} data - it is \gls{preferred} iff the half-edge of $e$ was \gls{preferred}, and it has the same \gls{consecutive} half-edge. Figure~\ref{fig:Figure 15x} depicts the two possible ways to construct $G'$ from $G$. 

\begin{figure}
\begin{center}
\includegraphics{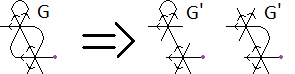}
\caption{Cutting an edge and adding two branch values to an arrowed daisy graph}
\label{fig:Figure 15x}
\end{center}
\end{figure}

$H'$, the graph structure of $G'$, has a lower genus then $H$. We assume, by induction, that there is a \gls{genericsurface} in $M$ that realizes $G'$ and \textbf{has only one body}. We will modify this surface so that it realizes $G$. Observe the new branch values at the ends of $e_1$ and $e_2$. Change the surface in a small neighborhood of each branch value as per Figure~\ref{fig:Figure 16x}A, deleting the branch value and leaving instead a ``figure 8 boundary" of the surface.

\begin{figure}
\begin{center}
\includegraphics{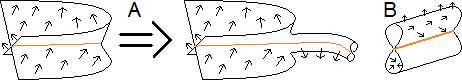}
\caption{Removing two branch values from a \gls{genericsurface} and restoring the previously cut edge}
\label{fig:Figure 16x}
\end{center}
\end{figure}

This figure 8 boundary is depicted in Figure~\ref{fig:Figure 11x}C. Take a bundle over an interval whose fibers are ``8-figures", as in Figure~\ref{fig:Figure 16x}B, and embed it in $M$ in such a way that its end-fibers coincide with the said ``figure 8 boundaries" (in a way that preserves the arrows of the orientation). Since the complement of the original surface was connected, you can make sure that the bundle does not intersect the surface anywhere except its ends. This closes $e_1$ and $e_2$ into one edge, reversing the procedure that created $G'$ from $G$, and so this new surface realizes $G$ while still having only one body. The proof follows by induction.
\end{proof}

\begin{rem} \label{Connected3}
It is possible once more to make sure that the underlying surface $F$ is connected. Firstly, you may connect the different connected components of $S$ via handles, similarly to the way you connected faces in the proof of Lemma~\ref{OneBody}. You may then proceed as in Lemma~\ref{Connected2}(b).
\end{rem}

\section{The order of the arc segments}

In the last section, we will explain how to strengthen Theorems~\ref{ADGThm} and \ref{ADGThm2}. We will do so by refining the definition of the \gls{DADG} structure of the intersection graph of a \gls{Thrice} \gls{genericsurface} so that it encodes more information regarding the topology of the surface.

According to Definition~\ref{SurfDADG}, in order to define the \gls{DADG} structure of a surface, you must order the triple values from $TV_0$ to $TV_{T-1}$. Then, for each $k=0,...,T-1$, you must order the three intersecting arc segments at the triple value $T_k$ as $TV_k^1$, $TV_k^2$ and $TV_k^3$. You then set the $k$th ``list of triple values" field of the \gls{DADG} to be $TV_k=(((r_1,s_1),(r_2,s_2)),((r_3,s_3),(r_4,s_4)),((r_5,s_5),(r_6,s_6)))$ where $(r_{2l},s_{2l})$ and $(r_{2l-1},s_{2l-1})$ are respectively the \gls{preferred} and non-\gls{preferred} ends of edges that compose the arc segment $TV_k^l$.

This leads us to consider as isomorphic \gls{DADG}s that differ only in the order of the arc segments at some triple value. For instance, if a \gls{DADG} contains the triple value $TV_k=(((5,1),(3,0)),((5,0),(9,1)),((4,0),(6,1)))$, permuting the order of the arc segment into, for instance, $TV_k=(((5,0),(9,1)),((4,0),(6,1)),$ $((5,1),(3,0)))$, will not actually change the \gls{DADG}. Indeed, these two \gls{DADG}s can represent the same \gls{genericsurface}. The difference between them represents only a difference in the way one indexes the arc segments of this surface.

Up until now, no restriction was imposed on the choice of how to index each arc segment, and thus we considered every permutation on the order of the arc segments to be an isomorphism of the \gls{DADG}. We would now like to change the definition. Figure~\ref{fig:Figure 17x} depicts 3 ways that one may index the arc segments of a triple value. In this figure, both the surface and the 3-manifold $M$ in which the surface resides are oriented. $M$ has the usual right hand orientation. The orientation of the surface imposes an orientation on each of the arc segments. This orientation points towards the \gls{preferred} side of the surface.

We would like to refine Definition~\ref{SurfDADG} so that, from now on, when choosing how to index the arc segments of a triple value, one must ensure that the triple of vectors $(TV_k^1,TV_k^2,TV_k^3)$ agrees with the orientation of $M$, as in Figures~\ref{fig:Figure 17x}A and B and unlike Figure~\ref{fig:Figure 17x}C. This implies that, from now on, only even permutations on the order of the arc segments of a triple value will be considered isomorphisms.

\begin{figure}
\begin{center}
\includegraphics{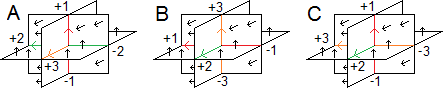}
\caption{The order of the arc segments may agree or disagree with the orientation of the 3-manifold.}
\label{fig:Figure 17x}
\end{center}
\end{figure}

We do this because even presentations preserve the topology of the neighborhood of the intersection graph while odd ones do not. In order to see this, recall Definition~\ref{CrossSurf} of a cross surface of a \gls{DADG}. In it, for every triple value $TV_k$ of $G$, one embeds a corresponding vertex neighborhood $TVN_k$ into $M$, and indexes its gluing zones in a way that corresponds to the ends of edge that reside on $TV_k$ according to $G$.

Specifically, the gluing zones on $TVN_k$ are divided into 3 pairs of ``gluing zones on opposite sides of $TVN_k$". One zone in each pair is \gls{preferred} - the orientation on surface sheets that separates the zone points towards it. For each $l=1,2,3$, we choose one pair of zones to correspond to the arc segment $TV_k^l$. In particular, the \gls{preferred} one of the zones will correspond to the \gls{preferred} \gls{EoE} from this arc segment, $(r_{2l},s_{2l})$, and the other zone will correspond to the other \gls{EoE}, $(r_{2l-1},s_{2l-1})$. One then embeds the X-bundles into $M$, and glues the matching end-fiber to the gluing zone.

As with the definition of the \gls{DADG} of a surface, we refine the definition of a cross surface and require that that the gluing zones be indexed in a way that matches the orientation of $M$. In particular, noting that the arc segment $TV_k^l$ of the cross surface is the line that connects the gluing zones $(r_{2l-1},s_{2l-1})$ and $(r_{2l},s_{2l})$ and points towards the latter, the triple $(TV_k^1,TV_k^2,TV_k^3)$ must agree with the orientation of $M$.

It is possible to see this in Figures~\ref{fig:Figure 18x}A-C. In accordance with Figure~\ref{fig:Figure 17x}, we use the colors red, green and orange to respectively indicate the arc segments $TV_k^1$, $TV_k^2$ and $TV_k^3$. The gluing zones are colored in correspondence with their arc segment, e.g zones $(r_1,s_1)$ and $(r_2,s_2)$ are colored red. We indicate the \gls{preferred} zones with a ``+" and the other ones with a ``-". For instance, the green zone marked with ``+" is $(r_4,s_4)$. In Figures~\ref{fig:Figure 18x}A-B the gluing zones are indexed correctly, in accordance with the orientation of $M$, and in figure C they are indexed wrongly.

\begin{figure}
\begin{center}
\includegraphics{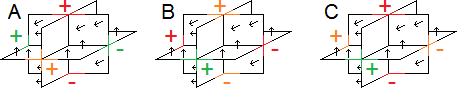}
\caption{The order of the gluing zones in regard to the orientation of the 3-manifold}
\label{fig:Figure 18x}
\end{center}
\end{figure}

While there are still 3 ways to index the gluing zones correctly, these differ up to an even permutation, which implies that the resulting cross surface is essentially unique. Indeed, if you built the cross surface according to one indexing, and wish to see what would happen if you used another indexing, simply cut the X-bundles from the vertex neighborhood $TVN_k$, rotate it in a why that would turn the shape in Figure~\ref{fig:Figure 18x}A into that in Figure~\ref{fig:Figure 18x}B, and re-glue.

The resulting cross surface is clearly isomorphic to the original one. In fact, there are neighborhoods $H_1,H_2 \subseteq M$ of the two cross-surfaces and an orientation preserving homeomorphism $f:H_1 \to H_2$ that sends the first cross-surface to the second one in a manner preserving the orientation on them. Note that, if $G$ was originally the \gls{DADG} structure of a \gls{genericsurface} $S$, then this implies that the cross surface of $G$ is homeomorphic to a neighborhood of the intersection graph of $S$, and so one can recreate such a neighborhood using only the said \gls{DADG} structure.

On the other hand, cross surfaces that differ by an odd permutation on one of the triple values do not even have to be homeomorphic. This holds even for cross surfaces with only 2 triple values. For instance, constructing the cross surfaces of the \gls{DADG} $((((0,0)(1,1)),(2,0)(3,1)),(4,0)(5,1))),(((1,0)(0,1)),(3,0)(2,1)),$ $(5,0)(4,1))))$ and $((((0,0)(1,1)),(2,0)(3,1)),(4,0)(5,1))),(((1,0)(0,1)),(5,0)$ \\ $(4,1)),(3,0)(2,1))))$ reveals that the former has 12 boundary components while the latter has only 8. We am referring to the connected components of the boundary of the cross surface, the same boundary studied in Lemma~\ref{partial}.

Refining the definition of the \gls{DADG} of a surface and the cross surface of a \gls{DADG} will not interfere with any of the results given in this thesis. In particular, Theorems~\ref{ADGThm} and \ref{ADGThm2} will still hold and their proofs will still work - they can still be used to create surfaces that realize any \gls{DADG}. The difference is that, after the refinement, a surface must fulfill more requirements in order to realize a \gls{DADG}, which means that the theorems are stronger.

%
%
%
\bibliographystyle{amsplain}
\bibliography{MyBib}
%
%
%
%

\end{document}